%% file: arxiv.tex
\documentclass{article}

\usepackage{microtype}
\usepackage{graphicx}
\usepackage{subfigure}
\usepackage{booktabs} 

\usepackage{hyperref}

\usepackage[accepted]{icml2021}

\icmltitlerunning{Newton Method over Networks is Fast up to the Statistical Precision}

\usepackage{lipsum}
\usepackage{amsthm} 
\usepackage{amsmath}
\usepackage{amssymb}
\usepackage{enumitem}
\usepackage{multirow}

\newtheorem{theorem}{Theorem}
\newtheorem{lemma}{Lemma}

\newtheorem{corollary}[theorem]{Corollary}
\newtheorem{remark}[theorem]{Remark}
\newtheorem{assumption}[theorem]{Assumption}

\newcommand{\alg}{DiRegINA }

\DeclareMathOperator*{\argmin}{argmin}

\newcommand{\norm}[1]{\left\lVert#1\right\rVert}

\newcommand{\bx}{x}

\newcommand{\Lhessian}{L}
\newcommand{\Lgrad}{Q}
\newcommand{\tO}[1]{\tilde{O}\left(#1\right)}

\newcommand{\hxnu}[1]{\bx^{#1+}}

\newcommand{\by}{y}

\newcommand{\bW}{{W}}

\newcommand{\bJ}{{J}}
\newcommand{\bI}{I}

\newcommand{\0}{\mathbf{0}}

\newcommand{\bepsilon}{\delta}
\newcommand{\gradtrack}{s}

\newcommand{\tF}{\widetilde{F}}

\newcommand{\VV}{\mathcal{V}}

\newcommand{\GG}{\mathcal{G}}
\newcommand{\EE}{\mathcal{E}}
\newcommand{\bbeta}{B}

\newcommand{\bH}{{H}}

\begin{document}

\twocolumn[
\icmltitle{Newton Method over Networks is Fast up to the Statistical Precision}



\icmlsetsymbol{equal}{*}

\begin{icmlauthorlist}
\icmlauthor{Amir Daneshmand}{purdue}
\icmlauthor{Gesualdo Scutari}{purdue}
\icmlauthor{Pavel Dvurechensky}{Weierstrass,HSE}
\icmlauthor{Alexander Gasnikov}{MIPT,HSE}
\end{icmlauthorlist}

\icmlaffiliation{purdue}{School of Industrial Engineering, Purdue University, West-Lafayette, IN, USA}
\icmlaffiliation{Weierstrass}{Weierstrass Institute for Applied Analysis and Stochastics, Berlin, Germany}
\icmlaffiliation{HSE}{Higher School of Economics (HSE) University, Moscow, Russia}
\icmlaffiliation{MIPT}{Moscow Institute of Physics and Technology, Dolgoprudny, Russia}

\icmlcorrespondingauthor{Gesualdo Scutari}{gscutari@purdue.edu}

\icmlkeywords{Machine Learning, ICML, Distributed Optimization, Statistical Similarity, Distributed Newton Method, Statistical Precision, Gradient Tracking}

\vskip 0.3in
]



\printAffiliationsAndNotice{} 

\begin{abstract}
We propose a distributed   cubic regularization of the Newton method for solving (constrained) empirical risk minimization  problems over a network of agents, modeled as undirected  graph. The algorithm employs   an  {\it inexact, preconditioned} Newton step at each agent's side:   the gradient of the centralized  loss is iteratively estimated via a gradient-tracking consensus mechanism  and the Hessian   is   subsampled  over the local data sets. No Hessian matrices are thus  exchanged over the network.    We derive global complexity bounds for convex and strongly convex   losses. Our analysis reveals an interesting interplay between sample and iteration/communication complexity: 
{\it statistically accurate} solutions are achievable roughly in   the same number of iterations of the centralized cubic Newton, with a communication cost per iteration of the order of $\widetilde{\mathcal{O}}\big(1/\sqrt{1-\rho}\big)$, where $\rho$  characterizes the connectivity of the network.   This represents a significant communication saving with respect to that of  existing, statistically oblivious, distributed Newton-based methods over networks. 
\end{abstract}\vspace{-.7cm}

\section{Introduction}\label{sec:intro}\vspace{-.1cm}
  
 We study  Empirical Risk Minimization (ERM) problems over a network of $m$ agents, modeled as undirected graph. Differently from master/slave systems, no centralized node is assumed in the network (which will be referred to as {\it meshed} network).  Each agent $i$   has access to $n$  i.i.d. samples  $z_i^{(1)},\ldots ,z_i^{(n)}$    drawn from an unknown, common distribution   on $\mathcal{Z}\subseteq\mathbb{R}^p$;  the associated   
  empirical risk  reads
  \vspace{-0.1cm}
 \begin{equation} 
 f_i(x)\triangleq \frac{1}{n}\sum_{j=1}^n \ell\big(x;z_{i}^{(j)}\big),\vspace{-0.1cm}
\label{eq:local_loss} \vspace{-0.1cm}
\end{equation}
 where $\ell:\mathbb{R}^d\times \mathcal{Z}\to \mathbb{R}$ is the loss function, assumed to be (strongly) convex in $x$, for any given $z\in \mathcal{Z}$.  Agents aim to   minimize  the total empirical risk  over the $N\!=m n$ samples, resulting in the following ERM over networks:\vspace{-0.1cm}\begin{equation} 
\widehat{x}\in \underset{\bx\in \mathcal{K}}{\text{argmin}}  \,F\left(\bx\right)\triangleq \frac{1}{m}\sum_{i=1}^m f_i\left(\bx\right),\vspace{-0.2cm}
\tag{P}\label{eq:P}
\end{equation}
where $\mathcal{K}\!\subseteq \mathbb{R}^d$ is  convex and  known to   the agents.  

  Since the functions $f_i$ can be accessed only locally and routing local data to other agents   is infeasible or highly inefficient, solving (\ref{eq:P}) calls for the design of distributed algorithms that alternate between a local computation procedure at each agent's side, and a round of  communication among neighboring nodes. 
 The cost of communications is often considered the bottleneck for distributed computing, if     compared with local (possibly parallel) computations (e.g., \cite{Bekkerman_book11,Lian17}). Therefore, our goal is   developing  {\it communication-efficient}  distributed algorithms 
that solve \eqref{eq:P} within the  {\it statistical} precision.

   
   The provably faster convergence rates of second order methods over gradient-based algorithms  make them potential candidates for communication saving (at the cost of more computations).   
   Despite the success of Newton-like methods to solve ERM in a centralized setting (e.g., \cite{ADA-Newton16,Bottou18}), including  master/slave architectures  \cite{DISCO,DANE,Ma17,DANCE20,pmlr-v108-soori20a},    their distributed counterparts on meshed networks  are not on par:     convergence  rates provably faster than those of first order methods are achieved at  high  communication  costs \cite{Uribe20,Zhang20}, cf. Sec.~\ref{sec:lit_review}. 
   
 
 We claim that stronger guarantees of second order methods over meshed networks can be certified if a {\it statistically-informed} design/analysis is put forth, in contrast with  statistically agnostic approaches that look  at (\ref{eq:P}) as    deterministic optimization  and  target    any arbitrarily small   suboptimality. 
 To do so, we  build  on the following two key insights. 
     
    $\bullet$ \textbf{Fact 1 (statistical accuracy):} When  it comes to learning problems,   the 
        ERM (\ref{eq:P}) is a surrogate of the population minimization \vspace{-0.2cm}
\begin{equation} 
x^\star\in  \underset{\bx\in \mathcal{K}}{\text{argmin}}\,  F_P\left(\bx\right)\triangleq \mathbb{E}_{Z\sim \mathbb{P}} \, \ell\left(\bx;Z\right).\vspace{-0.1cm}
 \label{eq:Population}
\end{equation}
 The ultimate goal is to estimate ${x}^\star$ via the ERM (\ref{eq:P}).  Denoting by $x_{\varepsilon}\in \mathcal{K}$ the estimate returned by the algorithm, we have the risk decomposition (neglecting the approximation error due to the use of a specific set of models $x\in \mathcal{K}$):\vspace{-.2cm}
\begin{equation}\label{eq:suboptimality-rule}
	 \begin{array}{lll}
	&	F_P(x_{\varepsilon})-F_P(x^\star) \medskip \\
	&= \underset{\leq \text{statistical error}}{\big\{F_P(x_{\varepsilon})-F(x_{\varepsilon})\big\}} \medskip +  {\big\{F(x_{\varepsilon})-F(x^\star)\big\}}\\& \quad + \underset{\leq \text{statistical error}}{\big\{F(x^\star)-F_P(x^\star)\big\}} \medskip\\ & \leq \mathcal{O}(\text{statistical error})+ \underset{=\text{optimization error}}{\big\{F(x_{\varepsilon})-F(\widehat{x})\big\}}\vspace{-.2cm}
	\end{array}
\end{equation}
where the statistical error  is usually of the order   $\mathcal{O}(1/\!\sqrt{N})$ or $\mathcal{O}(1/N)$ (cf.~Sec.~\ref{sec:setup}).
 It is thus sufficient to reach an optimization accuracy  $F(x_{\varepsilon})-F(\widehat{x})=\mathcal{O}(\text{statistical error})$. This can potentially save communications. 
 
 $\bullet$ \textbf{Fact 2 (statistical similarity):}  Under  mild  assumptions on the loss functions  and i.i.d samples across the agents (e.g., \cite{DISCO,pmlr-v119-hendrikx20a}),  it holds  with high probability (and uniformly on $\mathcal Z$)\vspace{-0.1cm}
 \begin{equation}
 \label{eq:intro_stat_sim}
 	 \left\|\nabla^2f_i(\bx) - \nabla^2 F(\bx)\right\|\leq  {\beta} = \widetilde{\mathcal{O}}(1/\sqrt{n}),\;\; \forall\bx\in\mathcal{K},\vspace{-0.1cm}\end{equation}      
with $\widetilde{\mathcal{O}}$  hiding  log-factors and the dependence on $d$.  
In words,  the local empirical losses $f_i$ are statistically similar to each other and the average $F$, especially for large $n$.

The key insight of Fact 1   is that one can 
target  suboptimal solutions of (\ref{eq:P}) within the statistical error. 
This is different from seeking a distributed optimization method that achieves any arbitrarily small empirical suboptimality. Fact 2 suggests  a further reduction in the communication complexity via  {\it statistical preconditioning}, that is,    subsampling the Hessian of $F$ over the local data sets, so that no Hessian matrix has to be transmitted over the network.  This paper shows that, if  synergically combined,  these two facts can improve the communication complexity of distributed second order methods over meshed networks. \vspace{-0.2cm}

\subsection{Major contributions}\vspace{-0.1cm}
We propose  and analyze a   decentralization of the cubic regularization of the Newton method \cite{Nesterov--Cubic06} over meshed networks. The algorithm employs  a local computation procedure performed in parallel by the agents coupled with a round of (perturbed) consensus mechanisms that aim to track {\it locally} the gradient of $F$ (a.k.a. gradient-tracking) as well as   enforce an agreement on the local optimization directions.   The optimization procedure is an   inexact, preconditioned (cubic regularized)  Newton step whereby the gradient of $F$   is estimated by  gradient tracking  while the Hessian  of $F$  is   subsampled  over the local data sets.    Neither a line-search nor communication of Hessian matrices over the network are performed.

We established for the first time  {\it global} convergence for different classes of ERM problems, as summarized in   Table~\ref{table:rate}. Our results are of two types: i) classical complexity analysis (number of communication rounds)  for arbitrary solution accuracy (right panel); ii) and  complexity bounds for statistically optimal solutions (left panel, $V_N$ is the statistical error). Postponing to   Sec~\ref{sec:main_results} a detailed discussion of these results, here we highlight some  key novelties  of our findings.  {\bf Convex ERM:}  For convex $F$, if arbitrary  $\varepsilon$-solutions are targeted, the algorithm exhibits a two-speed behavior:   1) a first rate of the order of  $\mathcal{O}((1/\sqrt{1-\rho})\cdot\sqrt{LD^3/\varepsilon^{1+\alpha}})$, as long as  $\varepsilon=\Omega(\Lhessian D^3\beta^2)$;  up to the network dependent factor $1/\sqrt{1-\rho}$, this (almost) matches the rate of the centralized   Newton method   \cite{Nesterov--Cubic06}; and 2) the slower rate  $\mathcal{O}(({1}/{\sqrt{1-\rho}})\cdot(LD^3\beta^2)/\varepsilon)$, which is due to the local subsampling of the global  Hessian $\nabla^2 F$; this term is dominant for smaller values of   $\varepsilon$.  The interesting fact is that  $\varepsilon=\Omega(\Lhessian D^3\beta^2)$ is of the order of the statistical error $V_N$. Therefore,  
rates of the order of centralized ones are provably achieved  up to statistical precision (left panel).  {\bf Strongly Convex ERM ($\beta<\mu$):} 
The communication complexity shows a three-phase rate behaviour  (right panel); for arbitrarily small $\varepsilon>0$, the worst-case communication complexity 
is linear, of the order of $\widetilde{\mathcal{O}}\big(({1}/{\sqrt{1-\rho}})\cdot (\beta/\mu)\cdot \log(1/\varepsilon)\big)$.   Faster rates are certified when $\varepsilon=\mathcal{O}(V_N)$ (left panel). Note that the region of superlinear convergence is a false improvement when  the first term $m^{1/4} \sqrt{\Lhessian D/\mu}$ is dominant, e.g.,    $F$ is   ill-conditioned and  $n$ is not large. 
 This term is unavoidable  \cite{Nesterov--Cubic06}--unless more refined function classes are considered, such as  self-concordant or quadratic ($L=0$). The left panel shows improved rates in the latter case or under an initialization  within a $\mathcal{O}(1/\sqrt{n})$-neighborhood of the solution.   {\bf Strongly Convex ERM ($\beta\geq\mu$):}  This is   a common setting when   $F_P$ is convex and a regularizer is used in the ERM (\ref{eq:P}) for learnability/stability purposes;  typically,   $\mu=\mathcal{O}(1/\sqrt{N})$, $\beta=\mathcal{O}(1/\sqrt{n})$. We proved linear rate for arbitrary $\varepsilon$-values. Differently from the majority of first-order methods over meshed networks (cf. Sec. \ref{sec:lit_review}),   this rate does not depend on the condition number of $F$ but on the generally   more favorable ratio   $\beta/\mu$. 
Furthermore, when $\varepsilon=\mathcal{O}(V_N)$, the rate does not improve  over   the convex case. \\
In summary, we   propose   a    second-order method solving   convex and strongly convex problems over meshed networks that,  for the first time, enjoys    global complexity    bounds   and   communication complexity close to oracle complexity of centralized methods up to the statistical precision.
\vspace{-0.2cm}

\renewcommand{\arraystretch}{1.2}
\begin{table*}[t]  
\caption{\small Communication complexity  of \alg to $\varepsilon>0$ suboptimality for  (strongly) convex ERM. {\bf Right column:}  arbitrary $\varepsilon$ values. {\bf Left column:} $\varepsilon=\Omega(V_N)$, $V_N$ is the statistical error  [cf.~\eqref{eq:suboptimality-rule}]. The other parameters are:   $\mu$ and $L$ are   the strong convexity constant of $F$ and Lipschitz constant of $\nabla^2 F$, respectively; $D$ and  $D_p$ are  estimates of the optimality gap at the initial point; 
$\beta$ measures the similarity of $\nabla^2 f_i$ [cf. \eqref{eq:intro_stat_sim}]; $\rho$ characterizes the connectivity of the network; and  $\alpha>0$ is an arbitrarily small constant.}
\label{table:rate}
\vspace{-0.3cm}
\begin{center}
\begin{small}
\resizebox{2.1\columnwidth}{!}{	\begin{tabular}{l l| c|c}
	\toprule
	{\bf Problem}   &  & {\bf  $\varepsilon=\Omega(V_N)$ (statistical error)}&    {\bf $\varepsilon>0$ (arbitrary)  }
	\\
	\midrule\hline
	\parbox{2cm}{\vspace{0.1cm}\scriptsize \bf Convex\  \\$\mu=0$ \\  $V_N= {\mathcal{O}}(1/{\sqrt{N}})$ }& \parbox{1.6cm}{\scriptsize Thm. \ref{thm:cvx_case} \vspace{1pt}\\Cor. \ref{corr:cvx_case}}& $\widetilde{\mathcal{O}}\left({\frac{1}{\sqrt{1-\rho}}\cdot \sqrt{\frac{\Lhessian D^3}{V_N^{1+\alpha}}}}\right)$ & $\widetilde{\mathcal{O}}\bigg(\frac{1}{\sqrt{1-\rho}} \cdot  \left\{ \sqrt{\frac{\Lhessian D^3}{\varepsilon^{1+\alpha}}}+\frac{\Lhessian D^3\beta}{\varepsilon^{1+\alpha/2}}\right\}\bigg)$ 
	\\
	\hline
	&\multirow{2}{1.5cm}{\scriptsize $\Lhessian>0$\smallskip \\Thm. \ref{thm:scvx_case_beta_leq_mu}\smallskip \\Cor.\ref{corr:scvx_case_beta_leq_mu}}  & \multirow{2}{6.2cm}{$
	\widetilde{\mathcal{O}}\bigg(\frac{1}{\sqrt{1-\rho}}\Big\{m^{1/4}\sqrt{\frac{\Lhessian D}{\mu}}
	+\log\log\left(\frac{\mu^3}{m\Lhessian^2V_N}\right)\Big\}\bigg)$}
	& 
	$\widetilde{\mathcal{O}}\bigg(\frac{1}{\sqrt{1-\rho}}\Big\{m^{1/4}\cdot\sqrt{\frac{\Lhessian D}{\mu}}
	+\log\log\left(\frac{\mu^2}{\beta^2}\min\left\{1,\frac{\beta^2\mu}{m\Lhessian^2 }\cdot \frac{1}{\varepsilon}\right\}\right)$
\\
\multirow{3}{2cm}{\scriptsize {\bf Strongly-convex}\smallskip \\ $0<\beta<\mu$\smallskip \\$V_N={\mathcal{O}}({1}/{N})$\smallskip \\ $\mu=\mathcal{O}(1)$} & & & $+\frac{\beta}{\mu}\log\left( \frac{\beta^2\mu}{m\Lhessian^2\varepsilon}\right)\Big\}\bigg)$
	\\ 	\cline{2-4}
	 & \parbox{1.5cm}{\scriptsize ~ \\$\Lhessian>0$\\ +initialization \\ {(Cor. \ref{corr:scvx_case_beta_leq_mu_initialization_iid})}}  &  $	\widetilde{\mathcal{O}}\left({\frac{1}{\sqrt{1-\rho}}\cdot \left\{\log\log\left(\frac{\mu^3}{m\Lhessian^2}\cdot \frac{1}{V_N}\right)\right\}}\right),\quad \beta=\mathcal{O}\left(\frac{1}{\sqrt{n}}\right)$ & $\widetilde{\mathcal{O}}\left({\frac{1}{\sqrt{1-\rho}} \left\{\log\log\left(\frac{\mu^2}{\beta^2}\cdot \min\left(1,\frac{\beta^2\mu}{m\Lhessian^2\varepsilon}\right)\right)+\frac{\beta}{\mu}\log\left(\frac{\beta^2\mu}{m\Lhessian^2\varepsilon}\right)\right\}}\right)$
	 	\\ 	\cline{2-4}
	 & \parbox{1.5cm}{\scriptsize ~\\ $\Lhessian=0$\\{(Thm. \ref{thm:scvx_case_beta_leq_mu_Qcase}, appendix \ref{sec:scvx_case_beta_leq_mu_Qcase_proof})}} &  $\widetilde{\mathcal{O}}\left(\frac{1}{\sqrt{1-\rho}}\cdot \log\log\left(\frac{D_p}{V_N}\right)\right),\quad \beta=\mathcal{O}\left(\frac{1}{\sqrt{n}}\right)$ 
	 &$\widetilde{\mathcal{O}}\left({\frac{1}{\sqrt{1-\rho}}\cdot \left\{\log\log\left(\frac{D_p}{\varepsilon
	}\right)+\frac{\beta}{\mu}\log\left(\frac{D_p\beta^2}{\mu^2\varepsilon}\right)\right\}}\right)$
\\
		\hline
		\multirow{2}{2cm}{\scriptsize {\bf Strongly-convex}\\\smallskip {\bf(regularized) }\smallskip \\$0<\mu\leq \beta$   \smallskip \\$V_N={\mathcal{O}}({1}/\sqrt{N})$} & \parbox{1.5cm}{\scriptsize $\Lhessian>0$\\(Thm. \ref{thm:scvx_case_beta_geq_mu})} & \parbox{5cm}{$\widetilde{\mathcal{O}}\bigg(\frac{1}{\sqrt{1-\rho}}  m^{1/2}\sqrt{\frac{\Lhessian D}{V_N}}\bigg),\quad  \left\{\hspace{-0.5cm}\begin{array}{cc}
	&\mu=\mathcal{O}(V_N)\\ &\beta=\mathcal{O}(\frac{1}{\sqrt{n}})\end{array}\right.$}& $\widetilde{\mathcal{O}}\bigg(\frac{1}{\sqrt{1-\rho}}\left\{\sqrt{\frac{\Lhessian D}{\mu}}\left(1+m^{1/4}\sqrt{\frac{\beta}{\mu}}\right)+\frac{\beta}{\mu}\log\Big(\frac{\beta^2\mu}{m\Lhessian^2\varepsilon}\Big)\right\}\bigg)$
		\\
		\cline{2-4}
		&\parbox{1.5cm}{\scriptsize $\Lhessian=0$\\{(Thm. \ref{thm:scvx_case_beta_geq_mu_QUAD}, appendix \ref{subsec:scvx_beta_geq_mu_Qcase})}} &  \parbox{6.5cm}{$\widetilde{\mathcal{O}}\left(\frac{1}{\sqrt{1-\rho}}\cdot m^{1/2}\cdot\log\left(\frac{1}{V_N}\right)\right) ,\quad  \left\{\hspace{-0.5cm}\begin{array}{cc}
	&\mu=\mathcal{O}(V_N)\\ &\beta=\mathcal{O}(\frac{1}{\sqrt{n}})\end{array}\right.$}& $\widetilde{\mathcal{O}}\left(\frac{1}{\sqrt{1-\rho}}\cdot \frac{\beta}{\mu}\cdot\log\left(\frac{1}{\varepsilon}\right)\right) $
		\\
	\bottomrule
	\end{tabular}}
\end{small}
\end{center}
\vskip -0.1in
\end{table*}


\subsection{Related Works}\label{sec:lit_review}\vspace{-0.2cm}
The literature of  distributed optimization   is vast; here we review  relevant  methods applicable to  {\it meshed networks}, omitting the less relevant work considering only  master-slave systems (a.k.a star networks).  \\ 
\textbf{$\bullet$ Statistically oblivious methods:}
Despite being vast and providing different communication and oracle complexity bounds, the literature (e.g.,  \cite{jakovetic2014fast,shi2015extra,Arjevani-ShamirNIPS15,nedic2017achieving,scaman2017optimal,lan2017communication,uribe2020dual,rogozin2020towards}) on decentralized {\bf first-order methods} for minimizing $Q$-Lipschitz-smooth and $\mu$-strongly convex global objective $F$ mostly focuses on the particular case where $n=1$ in \eqref{eq:local_loss} and $\mathcal{K}=\mathbb{R}^d$, and does not take into account statistical similarity of the  risks. The best convergence  results for nonaccelerated first-order methods certify linear rate, scaling with the condition number $\kappa=Q/\mu$ ($Q$ is the Lipschitz constant of $\nabla F$);   Nesterov-based acceleration  improves the dependence to $\sqrt{\kappa}$ \cite{gorbunov2020recent}.  This performance can still be unsatisfactory when 
$1+\beta/\mu< \kappa$ (resp. $1+\beta/\mu< \sqrt{\kappa}$).
 This is the typical situation of ill-conditioned problems, such as many learning problems where 
  the regularization parameter that is optimal for test predictive performance is very small \cite{pmlr-v119-hendrikx20a}.  
  For instance,  consider the ridge-regression problem  with optimal regularization parameter $\mu=1/\sqrt{mn}$ (Table 1 in \cite{DISCO}), we have: $\kappa=\mathcal{O}(\sqrt{m\cdot n})$ while $\beta/\mu=\mathcal{O}(\sqrt{m})$. Notice that the former  grows with the local sample size $n$, while  the latter is independent.

A number of {\bf second-order methods} were proposed for  distributed optimization over  meshed networks, with typical results being local superlinear convergence \cite{jadbabaie2009distributed,wei2013distributed,tutunov2019distributed} or global linear convergence no better than that of first-order methods \cite{MokhtariNewton2,Network-Newton17,MokhtariDQM16,MokhtariNewton3,So2020}. 
Improved upon first-order methods global bounds are achieved by exploiting expensive sending local Hessians over the network--such as  \cite{Zhang20}, obtaining communication complexity bound ${\mathcal{O}}(({mL\|\nabla f(x_0)\|}/{\mu^2})+\log\log({1}/{\varepsilon}))$--or employing double-loop schemes \cite{uribe2020distributed} wherein at each iteration, a distributed first-order method is called to find the Newton direction, obtaining iteration complexity   ${\mathcal{O}}(\sqrt[3]{{LD^3}/{\varepsilon}})$ at the price of excessive communications per iteration. Furthermore, these schemes cannot handle constraints. 
To the best of our knowledge,   no  distributed second-order method over meshed networks has been proved to globally converge  with communication complexity bounds  even up to a network dependent factor close to the standard  \cite{Nesterov--Cubic06} bounds ${\small {\mathcal{O}}\big(\sqrt{({LD^3})/{\varepsilon}}\big)}$ for convex and ${\mathcal{O}}(\sqrt{{LD}/{\mu}}+\log\log({\mu^3}/{L^2\varepsilon}))$ for $\mu$-strongly convex problems.   Table \ref{table:rate} shows the first results of this genre. 






\textbf{$\bullet$ Methods exploiting statistical similarity:}
Starting the works \cite{DANE,Arjevani-ShamirNIPS15} several papers studied the idea of statistical preconditioning to decrease the communication complexity over    star networks, 
for different problem classes; example include \cite{DANE,reddi2016aide,yuan2019convergence} (quadratic losses), \cite{DISCO} (self-concordant losses),   \cite{GIANT}   (under $n>d$), and  \cite{Fan2020} (composite optimization), with  \cite{pmlr-v119-hendrikx20a,dvurechensky2021hyperfast} employing acceleration.   None of these methods are implementable over  meshed networks, because they   rely on  a centralized (master) node. To our knowledge, Network-DANE \cite{NetDane} and SONATA \cite{sun2019distributed} are the only two methods   that leverage statistical similarity to enhance convergence of distributed methods 
over meshed networks; \cite{NetDane} studies  strongly convex quadratic losses while \cite{sun2019distributed} considers general objectives, achieving a communication complexity of $\widetilde{\mathcal{O}}((1/\sqrt{1-\rho})\cdot\beta/\mu\cdot \log(1/\varepsilon))$. Both schemes call at every  iteration for an exact solution of local strongly convex problems while our subproblems are based on second-order approximations, 
computationally thus less demanding. Nevertheless, our algorithm retains  same rate dependence on $\beta/\mu$. Our study covers also  non-strongly convex losses. 
\vspace{-0.2cm}

\section{Setup and Background}\label{sec:setup} \vspace{-0.1cm}
\subsection{Problem setting} \vspace{-0.1cm} We study convex and strongly convex instances of the ERM \eqref{eq:P}; specifically, we make   the following assumptions [note that, although explicitly omitted,  each $f_i(x)$ and thus $F$   depend  on the sample $z\in \mathcal{Z}$ via $\ell(x,z)$; all the assumptions below are meant to hold uniformly on $\mathcal{Z}$]. 
 
 \begin{assumption}[convex ERM]\label{convex-case} The following hold:\vspace{-0.4cm}\begin{itemize} 
\item[(i)]   $\emptyset\neq \mathcal{K}\subseteq\mathbb{R}^d$ is closed and convex;\vspace{-0.2cm}
\item[(ii)] Each $f_i:\mathbb{R}^d\times \mathcal{Z}\to \mathbb{R}$ is twice differentiable  and $\mu_i$- strongly convex on (an open set containing) $\mathcal{K}$, with $\mu_i\geq 0$; \vspace{-0.2cm}
\item[(iii)] Each $\nabla f_i$ is  $Q_i$-Lipschitz continuous on $\mathcal{K}$, where $\nabla f_i$ is the gradient with respect to $x$; let $\Lgrad_{\max}\triangleq \max_{i=1,\ldots, m}\Lgrad_i$;\vspace{-0.2cm}
\item[(iv)]   $F$ has bounded level sets. 
 \end{itemize}\end{assumption}
 
 	 \begin{assumption}[strongly convex ERM]\label{sconvex-case} Assumption~\ref{convex-case}(i)-(iii) holds and $F$ is  $\mu$-strongly convex on  $\mathcal{K}$, with $\mu> 0$.  \vspace{-0.5cm} \end{assumption}
 	 
  The following  condition is standard when studying second order methods. 
\begin{assumption}
	\label{assump:nabla2F_LC}
	$\nabla^2 F:\mathbb{R}^d\rightarrow \mathbb{R}^{d\times d}$ is $\Lhessian$-Lipschitz continuous on $\mathcal{K}$, i.e., 
	$
	\norm{\nabla^2 F(\bx)-\nabla^2F(\by)}\leq \Lhessian\norm{\bx-\by}$, for some $\Lhessian>0$ and all $\bx,\by\in\mathcal{K}$.\vspace{-0.1cm}
\end{assumption}

 \textbf{Statistical accuracy:} As anticipated in Sec.~\ref{sec:intro}, we are interested in computing estimates of the population  minimizer (\ref{eq:Population}) up to the statistical error using the ERM rule (\ref{eq:suboptimality-rule}). To do so, throughout  the paper, we  postulate the following standard  uniform convergence property, which suffices for learnability  by 
 \eqref{eq:suboptimality-rule}:
 there exists a constant $V_N$, dependent on $N=m\,n$, such that \vspace{-0.1cm}  \begin{equation}\label{eq:uniform-convergence}
     \sup_{x\in \mathcal{K}} \left | F(x)-F_P(x)\right|\leq V_N\quad \text{w.h.p.}\vspace{-0.1cm}
 \end{equation}  
The statistical accuracy  $V_N$ has been widely studied in the literature, e.g.,  \cite{Vapnik13,Bousquet02,Bartlett06,Frostig15,SHai-Shalev-book}. Consistently with these works, we
  will  assume:\vspace{-0.2cm}
  \begin{itemize}
      \item[1.]  $V_N=\mathcal{O}(1/N)$, for $\mu$-strongly convex  $F$ and $F_P$, with $0<\mu=\mathcal{O}(1)$;\vspace{-0.1cm}
      \item[2.] $V_N=\mathcal{O}(1/\sqrt{N})$ for convex   or $\mu$-strongly convex $F$, with $\mu=\mathcal{O}(1/\sqrt{N})$.\vspace{-0.2cm}
  \end{itemize}
 These  cases cover a variety of problems of practical interest. An example of  case 1 is a loss in the form $\ell(x;z)=f(x;z)+(\mu/2)\|x\|^2$, with fixed regularization parameter and $f$ convex in $x$ (uniformly on $z$),   as in ERM of linear predictors for supervised learning  \cite{sridharan2008fast}.  Case 2 captures traditional low-dimensional ($n>d$) ERM  with convex losses or regularized losses as above with optimal  regularization parameter $\mu=\mathcal{O}(1/\sqrt{N})$ \cite{Shwartz_et_al,Bartlett06,Frostig15}.


Under \eqref{eq:uniform-convergence}, the suboptimality gap at given $x\in \mathcal{K}$ reads:\footnote{\small We point out  that our results hold under (\ref{eq:suboptimality-rule2}), which can also be established using weaker conditions than \eqref{eq:uniform-convergence},
 e.g., invoking stability arguments \cite{Shwartz-JMLR2010}. }  \vspace{-0.1cm}
\begin{equation}\label{eq:suboptimality-rule2}
 	F_P(x)-F_P(x^\star) 
 \leq   \mathcal{O}(V_N) +  \big\{F(x)-F(\widehat{x})\big\},\quad \text{w.h.p.}\vspace{-0.1cm}
\end{equation}  Therefore, our ultimate goal will be 
computing $\varepsilon$-solutions $x_\varepsilon$ of \eqref{eq:P} of the order  $\varepsilon=\mathcal{O}(V_N)$. 

 \textbf{Statistical similarity:} We are interested in studying problem (P) under statistical similarity of $f_i$'s.

\begin{assumption}[$\beta$-related $f_i$'s]\label{assump:homogeneity} The local functions  $f_i$'s   are  $\beta$-related:    
	$\left\|\nabla^2 F(\bx)-\nabla^2f_i(\bx)\right\|_2\leq \beta$, for all $\bx\in\mathcal{K}$ and some $\beta\geq 0$.\vspace{-0.2cm}
\end{assumption}

The interesting case is when $1+\beta/\mu \ll \kappa\triangleq  {\Lgrad}/{\mu}$, where $\Lgrad$ is the Lipschitz constant of $\nabla F$ on $\mathcal{K}$ (uniformly on $\mathcal{Z}$). Under standard assumptions on data distributions and learning model underlying the ERM-see, e.g., \cite{DISCO,pmlr-v119-hendrikx20a}--$\beta$ is of the order $\beta=\mathcal{O}\left(1/\sqrt{n}\right)$, with high probability. In our analysis, when we target convergence to the statistical error,   
we will tacitly assume such dependence of $\beta$ on the local sample size.  Note that our bounds hold for general situations when Assumption \ref{assump:homogeneity} may hold due to some other reason besides statistical arguments. \vspace{-0.4cm}  
  
\subsection{Network setting}\vspace{-0.1cm} 
The network of agents is modeled as a fixed, undirected graph,   $\GG \triangleq (\VV, \EE)$, where $\VV \triangleq \{1,\ldots, m\}$ denotes  the vertex    set--the set of agents--while $\EE \triangleq \{(i,j) \, |\, i,j \in \VV \}$ represents the set of edges--the communication links;  $(i,j) \in \EE$ iff there exists a communication link between agent $i$ and $j$. 
The following is a standard assumption on the   connectivity.  
\begin{assumption}[On the network]\label{assump:network}
	The graph $\GG$ is connected. 
\end{assumption}

  
  

  

\section{Algorithmic Design: DiRegINA} 
We aim at decentralizing the   cubic regularization of the Newton method \cite{Nesterov--Cubic06} over undirected graphs.   The main challenge in developing such an algorithm is to track and adapt a faithful estimates of the global gradient and Hessian matrix of $F$ at each agent, without incurring in an unaffordable communication overhead while still guaranteeing convergence at fast rates. Our idea is to estimate locally  the gradient $\nabla F$ via gradient-tracking \cite{Xu-TAC:hs,NEXT16} while the Hessian $\nabla^2 F$  is replaced by the local subsampled estimates $\nabla^2 f_i$ (statistical preconditioning). The algorithm, termed \alg  {(\underline{Di}stributed \underline{Reg}ularized \underline{I}nexact \underline{N}ewton \underline{A}lgorithm)},  is formally  introduced in Algorithm \ref{alg:DiRegINA}, and commented next.

Each agent maintains and updates iteratively a local copy $x_i\in \mathbb{R}^d$ of the global optimization variable $x$ along with the auxiliary variable $s_i\in \mathbb{R}^d$, which estimates the gradient of the global objective $F$;   $x_i^\nu$ (resp. $s_i^\nu$) denotes the value of $x_i$ (resp. $s_i$) at iteration $\nu\geq 0$.  (\texttt{S.1}) is the optimization step wherein 
every agent $i$, given $x_i^\nu$ and   $s_i^\nu$, minimizes     an inexact local second-order approximation of $F$, as defined in (\ref{def:x_nuplus}).
In this surrogate function, i) $\gradtrack_i^\nu$ acts as an approximation of $\nabla F$ at $x_i^\nu$, that is, $s_i^\nu \approx   \nabla F(x_i^\nu)$; ii)  in the quadratic term, $\nabla^2 f_i(x_i^\nu)$     plays the role of $\nabla^2 F(x_i^\nu)$ (due to statistical similarity, cf. Assumption \ref{assump:homogeneity}) with   $\tau_i I$  ensuring strong convexity of the objective; and iii) the last term is the cubic regularization as in the centralized  method \cite{Nesterov--Cubic06}. In  (\texttt{S.2}), based upon exchange of the two vectors  $x_i^{\nu +}$ and $s_i^{\nu}$ with their immediate neighbors, each agent updates the estimate  $x_i^{\nu}\to x_i^{\nu+1}$   via the consensus step (\ref{DiRegINA_pseudocode}) and   $s_i^{\nu}\to s_i^{\nu+1}$ via the perturbed consensus (\ref{grad_track_update}), which in fact  tracks $\nabla F(x_i^\nu)$ \cite{Xu-TAC:hs,NEXT16}. 
 The weights $(W_K)_{i,j=1}^m$  in (\ref{DiRegINA_pseudocode})-(\ref{grad_track_update}) are free design quantities and subject to   the following  conditions, where $\mathcal{P}_K$  denotes the set of polynomials with degree less than or equal than $K=1,2,\ldots $.    
\begin{assumption}[On the weight matrix ${W_K}$]\label{assump:weight}    The matrix    ${W_K}=P^K(\overline{{W}})$, where $P_K\in \mathcal{P}_K$ with $P_K(1)=1$, and $\overline{{W}}\triangleq \big(\bar{w}_{ij}\big)_{i,j=1}^m$  is a reference matrix   satisfying the following conditions: 

\noindent\textbf{(a)} $\overline{W}$ has a sparsity pattern compliant with $\mathcal{G}$, that is \vspace{-0.2cm}
	\begin{enumerate}[label=\roman*)]
		\item $\bar{w}_{ii} > 0$, for all $i = 1, \ldots, m$;\vspace{-0.1cm}
		\item $\bar{w}_{ij} > 0$, if $(i,j) \in \mathcal{E}$; and $\bar{w}_{ij}=0$ otherwise.\vspace{-0.2cm}
	\end{enumerate}
	
\noindent\textbf{(b)}  	 $\overline{{W}}$ is doubly stochastic, i.e., ${1}^\top \overline{W} = {1}^\top$ and $\overline{W}  {1} = 1$.  

Let $\rho_K\triangleq \lambda_{\max}(W_K-11^\top/m)$ [$\lambda_{\max}(\bullet)$ denotes the largest eigenvalue of the matrix argument]
\end{assumption}




When $K=1$, $W_K=\overline{W}$, that is, a single round of communication per iteration is performed. Several rules have been proposed in the literature for $\overline{W}$ to be compliant with Assumption \ref{assump:weight}, such as  the Laplacian,  the Metropolis-Hasting, and the maximum-degree weights rules; see, e.g., \cite{Nedich_tutorial}  and references therein.  When $K>1$,   $K$  rounds  of communications per iteration $\nu$ are employed.   
For instance, this can be performed using the same   reference matrix $\overline{W}$ (satisfying Assumption~\ref{assump:weight}) in each communication exchange, resulting in     $W_K=\overline{W}^K$ and    $\rho_K={\rho}^K$, with $\rho=\lambda_{\max}(\overline{W}-11^\top/m)<1$.
Faster information  mixing   can be  obtained using suitably designed polynomials $P_K(\overline{W})$, such as Chebyshev \cite{auzinger2011iterative,scaman2017optimal} or orthogonal (a.k.a. Jacobi) \cite{Berthier2020} polynomials (notice that $P_K(1)=1$  is to ensure the doubly stochasticity of $W_K$ when $\overline{W}$ is doubly stochastic).

Although the minimization  \eqref{def:x_nuplus} may look challenging, it is showed in \cite{Nesterov--Cubic06} that its computational complexity is of the same order as of finding the standard Newton step. Importantly, in our algorithm, these are local steps made without any communications between the nodes.

\begin{algorithm}[t]	
	\caption{ \alg}\label{alg:DiRegINA}
	\textbf{Data}: $\bx^{0}_{i}\in \mathcal{K}$ and   $\gradtrack_i^0=\nabla f_i(\bx_i^0)$, $\tau_i>0$, $M_i>0$,  $\forall  i$. 
	
	\textbf{Iterate}: $\nu=1,2,...$\vspace{0.1cm}
	
	\begin{subequations}
		\texttt{[S.1] [Local Optimization]} Each agent $i$ computes $\bx_i^{\nu+}$: 		\begin{equation}
		\label{def:x_nuplus}
		\begin{aligned}
		&\bx_i^{\nu+}=\argmin_{\by\in\mathcal{K}}\,F(\bx_i^\nu)+\left\langle \gradtrack_i^\nu,\by-\bx_i^\nu\right\rangle
		\\
		&+\frac{1}{2}\left\langle \left[\nabla^2 f_i(\bx_i^\nu)+\tau_i\bI\right] \left(\by-\bx_i^\nu\right),\by-\bx_i^\nu\right\rangle+\frac{M_i}{6}\norm{\by-\bx_i^\nu}^3.
		\end{aligned}
		\end{equation}
		\texttt{[S.2] [Local Communication]} Each agent $i$ updates its local variables according to \vspace{-0.1cm}
        \begin{align}
		\bx_i^{\nu+1}= & \sum_{j=1}^m (W_K)_{i,j}\, \bx_j^{\nu+},		\label{DiRegINA_pseudocode}\\
		\gradtrack_i^{\nu+1}= & \sum_{j=1}^m (W_K)_{i,j}\, \left(	\gradtrack_j^{\nu} + \nabla f_j(\bx_j^{\nu+1})-\nabla f_j(\bx_j^{\nu})\right).\label{grad_track_update}\vspace{-0.1cm}
		\end{align}
		\textbf{end} 
	\end{subequations}
\end{algorithm}

\textbf{On the initialization:} We will study convergence of Algorithm \ref{alg:DiRegINA} under two sets of initialization for the $x$-variables, namely: i) random initialization and ii) statistically informed initialization. The latter is given by \vspace{-0.1cm}
\begin{equation}\label{eq:initialization}
\bx_i^{0}=\sum_{j=1}^m {(W_K)_{i,j}}\bx^{-1}_j,\quad\text{with}\quad  \bx^{-1}_i=\argmin_{x\in\mathcal{K}}f_i(x).\vspace{-0.1cm}
\end{equation} 
 This corresponds to a preliminary round of consensus on the local solutions $\bx^{-1}_i$. This second strategy takes advantage of the statistical similarity of $f_i$'s to guarantee, under \eqref{eq:uniform-convergence},  an initial optimality  gap of the order of:  
$p^0\triangleq \frac{1}{m}\sum_{i=1}^m\left(F(\bx_i^0)-F(\widehat{\bx})\right)=\mathcal{O}(1/\sqrt{n})$. If we further assume $\mu_i>0$, for all $i$, one can show that $p^0=\mathcal{O}(1/{n})$. 
This  will be shown to significantly   improve  the convergence rate of the algorithm, at a negligible extra communication cost  (but local computations). \vspace{-0.1cm}



\section{Convergence Analysis}\label{sec:main_results}\vspace{-0.1cm}
  
In this section, we study convergence of  \alg applied to  convex (cf. Sec.~\ref{sec:cvx_case}) and   strongly convex ERM  \eqref{eq:P}, the latter with either $\beta< \mu$ (cf. Sec.~\ref{sec:scvx_case}) or   $\beta\geq \mu>0$ (cf. Sec.~\ref{sec:scvx_case_regularized}). Our complexity results are of two type: i) classical rate bounds targeting any arbitrary ERM suboptimality  $\varepsilon>0$; and ii) convergence rates  to   $V_N$-solutions of \eqref{eq:P} (statistical error). 
Our complexity bounds are established in terms of the suboptimality gap:  \vspace{-0.1cm}
\begin{equation}\label{p_nu}
p^\nu\triangleq \frac{1}{m}\sum_{i=1}^m\left(F(\bx_i^\nu)-F(\widehat{\bx})\right), 
\end{equation}
where $\{\bx_i^\nu\}_{i=1}^m$  is the iterate  generated by  \alg at iteration $\nu$ (iterations are counted as  number of optimization steps \texttt{(S.1)}). 
Similarly to the centralized case \cite{Nesterov--Cubic06}, our bounds also depend on the following distance of initial points $\bx_i^0$, $i=1,\ldots, m$, from a given optimum $\widehat{x}$ of \eqref{eq:P}\begin{equation*}\label{def:D}
D\triangleq \!\!\!\max_{
\begin{subarray}{l }
	\bx_i\in\mathcal{K},
		\forall i
	\end{subarray}
}\left\{\max_{i=1\ldots, m}||\bx_i-\widehat{\bx}||:\sum_{i=1}^mF(\bx_i)\leq \sum_{i=1}^mF(\bx_i^0)\right\}.\vspace{-0.1cm}
\end{equation*}
Note that $D<\infty$ (cf. Assumption~\ref{convex-case}).

For the sake of   simplicity, in the   rate bounds we hide universal constants and log factors independent on $\varepsilon$ via $\widetilde{\mathcal{O}}$-notation; the exact expressions can be found in the supplementary material along with a detailed characterization of all the rate regions travelled by the algorithm. \vspace{-0.1cm}
  
\subsection{ Convex ERM \eqref{eq:P}}\label{sec:cvx_case} Our first result pertains to convex $F$ (and $F_P$). 
\begin{theorem}\label{thm:cvx_case}
Consider the ERM \eqref{eq:P} under Assumptions~\ref{convex-case}, \ref{assump:nabla2F_LC}, and \ref{assump:homogeneity} over a graph $\mathcal{G}$ satisfying Assumption \ref{assump:network};  and let $\{x^{\nu}_i\}_{i=1}^m$ be the sequence generated by \alg  
 under the following  tuning:   $M_i= \Lhessian>0$ and $\tau_i=2\beta$, for all $i=1,\ldots,m$;   $W_K=P_K(\overline{W})$ (and $P_K(1)=1$), where $\overline{W}$ is a given matrix satisfying Assumption \ref{assump:weight} with $\rho=\lambda_{\max}(\overline{W}-11^\top/m)$, and $K=\widetilde{\mathcal{O}}(\log(1/\varepsilon)/\sqrt{1-\rho})$, with   $\varepsilon>0$ being the target  accuracy. Then,    the total number of communications for \alg to make   $p^\nu\leq \varepsilon$ reads
	\begin{equation}\label{cvx_case_rate}
\begin{aligned}
\widetilde{\mathcal{O}}\bigg(\frac{1}{\sqrt{1-\rho}} \cdot  \left\{ \sqrt{\frac{\Lhessian D^3}{\varepsilon^{1+\alpha}}}+\frac{\Lhessian D^3\beta}{\varepsilon^{1+\alpha/2}}\right\}\bigg),
\end{aligned}
	\end{equation}
where $\alpha>0$ is arbitrarily small. In particular, if the   $\mathcal{G}$ is  a star or fully-connected,  $\rho=0$ and $\alpha=0$.
\end{theorem}
\vspace{-0.4cm}
\begin{proof}
	See Appendix \ref{sec:cvx_case_proof} in the supplementary material.
\end{proof}  

\vspace{-1em}
The rate expression (\ref{cvx_case_rate})  has an interesting interpretation. The multiplicative factor  $1/\sqrt{1-\rho}>1$   accounts for the rounds of  communications   per iteration (optimization steps) while the other two terms quantify the overall number of iterations   to reach the desired accuracy $\varepsilon$. Note that the first of these two terms, ${\mathcal{O}}(\sqrt{LD^3/\varepsilon^{1+\alpha}})$, is ``almost'' identical to the rate of the centralized Newton method  (with a slight difference  definition of $D$; see \cite{Nesterov--Cubic06}) while the other one, ${\mathcal{O}}((LD^3\beta)/\varepsilon^{1+\alpha/2})$, is a byproduct of the  discrepancy between local and global Hessian matrices. 
This shows a two-speed behavior of the algorithm, depending on the target accuracy $\varepsilon>0$: 1) as long as  $\varepsilon=\Omega(\Lhessian D^3\beta^2)$, ${\mathcal{O}}((LD^3\beta^2)/\varepsilon)$ can be neglected and the algorithm exhibits  almost centralized fast convergence (up to the network effect),  ${\mathcal{O}}(\frac{1}{\sqrt{1-\rho}}\sqrt{{LD^3}/{\varepsilon^{1+\alpha}}})$; 2) on the other hand, for   smaller (order of) $\varepsilon$, the rate is determined by the worst-term ${\mathcal{O}}(\frac{1}{\sqrt{1-\rho}}(LD^3\beta^2)/\varepsilon)$.

The interesting observation is that, in the setting above  and under \eqref{eq:uniform-convergence}, (\ref{eq:suboptimality-rule2}) holds with  $V_N=\mathcal{{\mathcal{O}}}(1/\sqrt{N})$ and $\beta=\mathcal{{\mathcal{O}}}(1/\sqrt{n})$. Hence,   $\varepsilon=\Omega(\Lhessian D^3\beta^2)$ is of the order of the statistical error $V_N$, as long as $m\leq n$, which is a reasonable condition.  This together with Theorem \ref{thm:cvx_case} implies that   fast rates (of the order of centralized ones)  can be certified up to the statistical precision, as formalized next. 

\begin{corollary}[$V_N$-solution]\label{corr:cvx_case}
Instate the setting of Theorem \ref{thm:cvx_case}, and  let $V_N=\mathcal{O}(1/\sqrt{N})$,  $\beta=\mathcal{O}(1/\sqrt{n})$, and $m\leq n$.  Then \alg returns a $V_N$-solution of  \eqref{eq:P}  in   	
	\begin{equation}\label{cvx_case_iid_rate}
	\widetilde{\mathcal{O}}\left({\frac{1}{\sqrt{1-\rho}}\cdot \sqrt{\frac{\Lhessian D^3}{V_N^{1+\alpha}}}}\right)\vspace{-0.2cm} 
	\end{equation}
	communications.
\end{corollary}
 

\subsection{Strongly-convex ERM \eqref{eq:P} with $\beta<\mu$}\label{sec:scvx_case}
We consider now the case of $F$ $\mu$-strongly convex and  $\beta<\mu$. 
The complementary case $\beta\geq \mu$ is studied in Sec.~\ref{sec:scvx_case_regularized}. 

\begin{theorem}\label{thm:scvx_case_beta_leq_mu}
	Instate the setting of  Theorem~\ref{thm:cvx_case} with Assumption \ref{convex-case} replaced by Assumption \ref{sconvex-case} and $K=\widetilde{\mathcal{O}}(1/\sqrt{1-\rho})$; and further assume   $\beta<\mu$. 
	 Then,    the total number of communications for \alg to make   $p^\nu\leq \varepsilon$ reads 
	\begin{equation}\label{scvx_case_rate}
	\hspace{-0.2cm}\begin{aligned}\small 
	&\widetilde{\mathcal{O}}\Bigg(\frac{1}{\sqrt{1-\rho}}\Bigg\{m^{\frac{1}{4}} \sqrt{\frac{\Lhessian D}{\mu}}
+\log\log\left[\frac{\mu^2}{\beta^2}\cdot \min\Big(1, \frac{\beta^2\mu}{m\Lhessian^2}\cdot \frac{1}{\varepsilon}\Big)\right]
	\\
	&\qquad\qquad\qquad\qquad +\frac{\beta}{\mu}\log\left[\max\Big(1,\frac{\beta^2\mu}{m\Lhessian^2} \cdot \frac{1}{\varepsilon}\Big)\right]\Bigg\}\Bigg).
	\end{aligned}
	\end{equation}	
\end{theorem}
\vspace{-0.4cm}
\begin{proof}
	See Appendix~\ref{sec:scvx_case_proof} in the supplementary material. 
\end{proof}   

 \alg  exhibits a different rate behavior, depending on the value of  $\epsilon$. We notice   three ``regions'': 1) a first phase of the order of   $\widetilde{\mathcal{O}}\big({m^{1/4}\sqrt{\Lhessian D/\mu}}\big)$ number of iterations; 
 2) the second region is of quadratic convergence, with rate of the order of $\log\log(1/\varepsilon)$;  and finally 3)   the  region of linear convergence with rate $\widetilde{\mathcal{O}}\left(\beta/\mu\log(1/\varepsilon)\right)$. This last region is not present in the rate of the centralized cubic regularization of the Newton method and is  due to the Hessians discrepancy.  Clearly, for arbitrarily small $\varepsilon>0$,  (\ref{scvx_case_rate}) is dominated by  the last term, resulting in a linear convergence. This linear rate is     
 slightly worse than that of SONATA  \cite{sun2019distributed} in  sight of first two terms in \eqref{scvx_case_rate}. This is    because  \alg is an inexact (and thus more computationally efficient) method than  \cite{sun2019distributed}.   We remark that more favorable complexity estimates   can be obtained when    $L=0$ (i.e., $f_i$'s are quadratic)--we refer the  reader to the supplementary material for details. 
 
 The algorithm does not enter in the last region if    $\varepsilon=\Omega(\beta^2\mu/(m\Lhessian^2))$.     This means that faster rate can be guaranteed up to $V_N$-solutions, 
 as stated next.

\begin{corollary}[$V_N$-solution] \label{corr:scvx_case_beta_leq_mu}
Instate the setting of Theorem \ref{thm:scvx_case_beta_leq_mu}, and  let $V_N=\mathcal{O}(1/{N})$,  $\beta=\mathcal{O}(1/\sqrt{n})$, $\mu=\mathcal{O}(1)$, and $m\leq n$.   \alg returns a $V_N$-solution of  \eqref{eq:P}  in  
 \begin{equation}\label{scvx_case_iid_rate}
\widetilde{\mathcal{O}}\bigg(\frac{1}{\sqrt{1-\rho}}\Bigg\{m^{1/4}\sqrt{\frac{\Lhessian D}{\mu}}
+\log\log\left(\frac{\mu^3}{m\Lhessian^2 V_N}\right)\Bigg\}\bigg)\vspace{-0.1cm} 
\end{equation}
communications.
\end{corollary}

When the problem is ill-conditioned (i.e. $\mu\ll 1$) the first term $m^{1/4} \sqrt{\Lhessian D/\mu}$ may dominate the $\log\log$ term in \eqref{scvx_case_iid_rate}, unless $n$  is extremely large (and thus $V_N$ very small). This term is unavoidable--it is present also in the centralized instances of Newton-type methods--unless more refined function classes are considered, such as (generalized) self-concordant \cite{Bach10,nesterov2018lectures,Sun19-Self-concordance}. In the supplementary material, we present results for quadratic losses (cf. Appendix \ref{sec:scvx_case_beta_leq_mu_Qcase_proof}). 
Here, we take another direction and show that the initialization strategy (\ref{eq:initialization}) 
is enough to get rid of the first phase.    
\begin{corollary}[$V_N$-solution + initialization]\label{corr:scvx_case_beta_leq_mu_initialization_iid}
	Instate the setting of Theorem \ref{corr:scvx_case_beta_leq_mu} and further assume: $\mu_i=\Omega{(1)}$, for all $i=1,\ldots m$, and   {$n=\Omega(\Lhessian^2/\mu^3\cdot m)$}.  \alg, initialized with (\ref{eq:initialization}), returns a $V_N$-solution of  \eqref{eq:P}  in    	
	\begin{equation}\label{scvx_case_rate_init_iid}
	\widetilde{\mathcal{O}}\left({\frac{1}{\sqrt{1-\rho}}  \left\{\log\log\left(\frac{\mu^3}{m\Lhessian^2}\cdot \frac{1}{V_N}\right)\right\}}\right)
	\end{equation}
	communications.\vspace{-0.3cm}
\end{corollary}
 
\begin{proof}
    See Appendix \ref{sec:scvx_case_beta_leq_mu_initialization_proof} in the supporting material. 
\end{proof} 
 
 \subsection{Strongly-convex ERM \eqref{eq:P} with $\beta\geq\mu$}\label{sec:scvx_case_regularized}
We now consider   the  complementary case  $\beta\geq \mu$. This is   a common setting when   $F_P$ is convex and a regularizer is used in the ERM (\ref{eq:P}), making  $F$ $\mu$-strongly convex;  typically,   $\mu=\mathcal{O}(1/\sqrt{N})$ while $\beta=\mathcal{O}(1/\sqrt{n})$. 

\begin{theorem}\label{thm:scvx_case_beta_geq_mu}
Instate the setting of Theorem \ref{thm:scvx_case_beta_leq_mu} with now  $\mu\leq \beta\leq 1$. Then,    the total number of communications for \alg to make   $p^\nu\leq \varepsilon$ reads 
\begin{equation}\label{reg_scvx_rate} 
\hspace{-0.2cm}\widetilde{\mathcal{O}}\left({\frac{1}{\sqrt{1-\rho}}\left\{\sqrt{\frac{\Lhessian D}{\mu}}\left(1+m^{\frac{1}{4}}\sqrt{\frac{\beta}{\mu}}\right)+\frac{\beta}{\mu}\log\Big(\frac{\beta^2\mu}{m\Lhessian^2} \frac{1}{\varepsilon}\Big)\right\}}\right).
\end{equation}	
\end{theorem}
\vspace{-0.3cm}
\begin{proof}
	See Appendix~\ref{subsec:scvx_beta_geq_mu} in the supplementary material. 
\end{proof}

For arbitrary small $\varepsilon>0$, the rate (\ref{reg_scvx_rate}) is dominated by the linear term. When we target $V_N$-solutions, in this setting $V_N=\mathcal{O}(1/\sqrt{N})$,   $\mu=\mathcal{O}(V_N)$ (as for the regularized ERM setting), and $\beta=\mathcal{O}(1/\sqrt{n})$, \eqref{reg_scvx_rate} becomes \vspace{-0.2cm}
\begin{equation}
\widetilde{\mathcal{O}}\left({\frac{1}{\sqrt{1-\rho}}\cdot   m^{1/2}\cdot \sqrt{\frac{\Lhessian D}{V_N}}}\right).\vspace{-0.1cm}
\end{equation}
Note that this rate is of the same order of the one achieved in  the convex setting (with no regularization)--see Corollary~\ref{corr:cvx_case}. If the functions $f_i$ are quadratic, the rate, as expected, improves and reads (see supporting material, Appendix~\ref{subsec:scvx_beta_geq_mu_Qcase})   \vspace{-0.1cm}
$$
\widetilde{\mathcal{O}}\left(\frac{1}{\sqrt{1-\rho}}\cdot   m^{1/2}\cdot \log\left(\frac{1}{{V_N}}\right)\right).
$$
Note that, on star networks ($\rho=0$), this rate improves on  that of DANE \cite{DANE}.

\begin{figure*}\centering
	\subfigure[]{
		\includegraphics[width=0.255\textwidth]{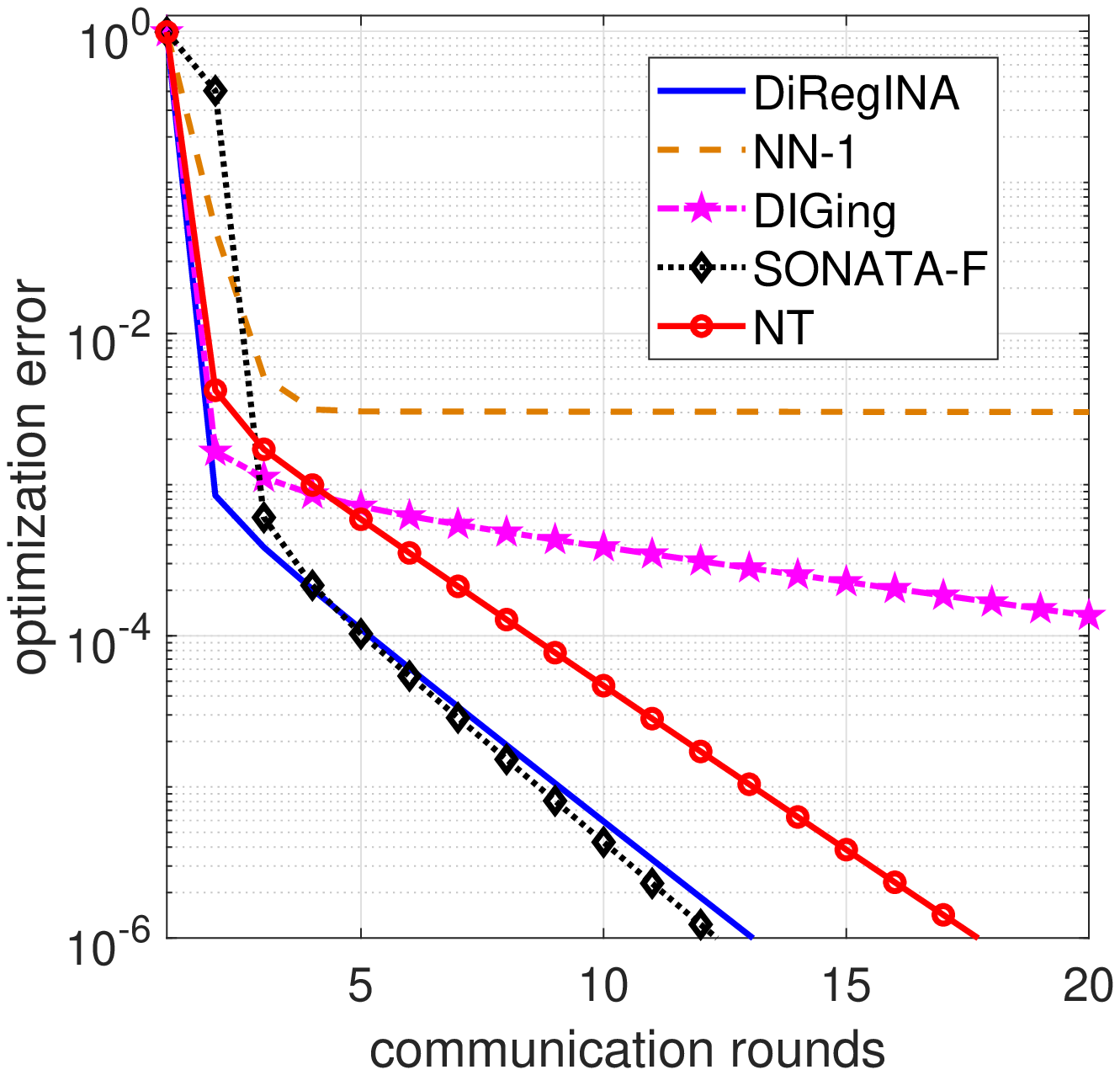}\hspace{-0.4cm}
		\label{RidgeReg_mg1}
	}
	\subfigure[]{
		\includegraphics[width=0.255\textwidth]{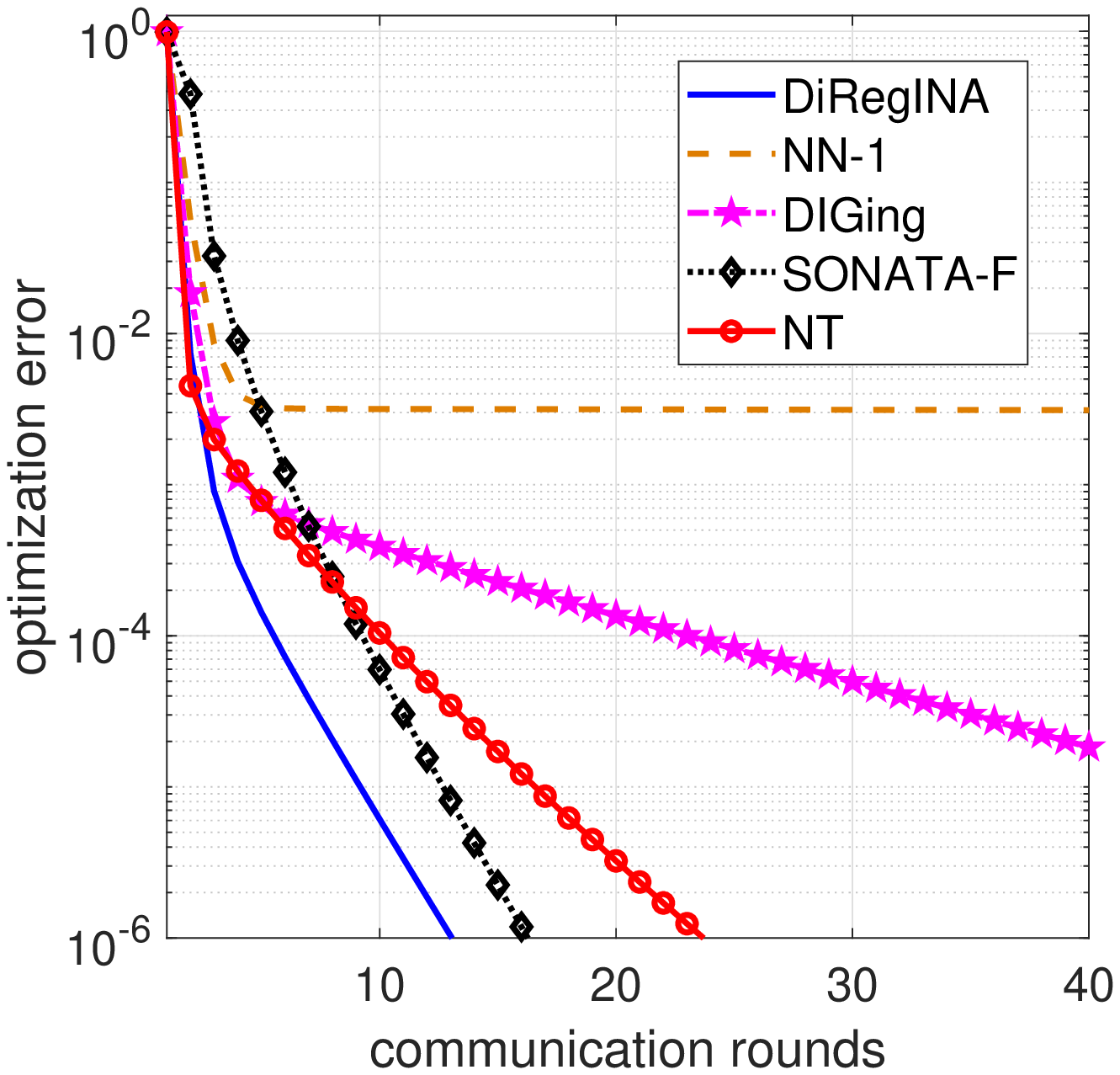}\hspace{-0.4cm}
		\label{RidgeReg_mg2}
	}
	\subfigure[]{
		\includegraphics[width=0.255\textwidth]{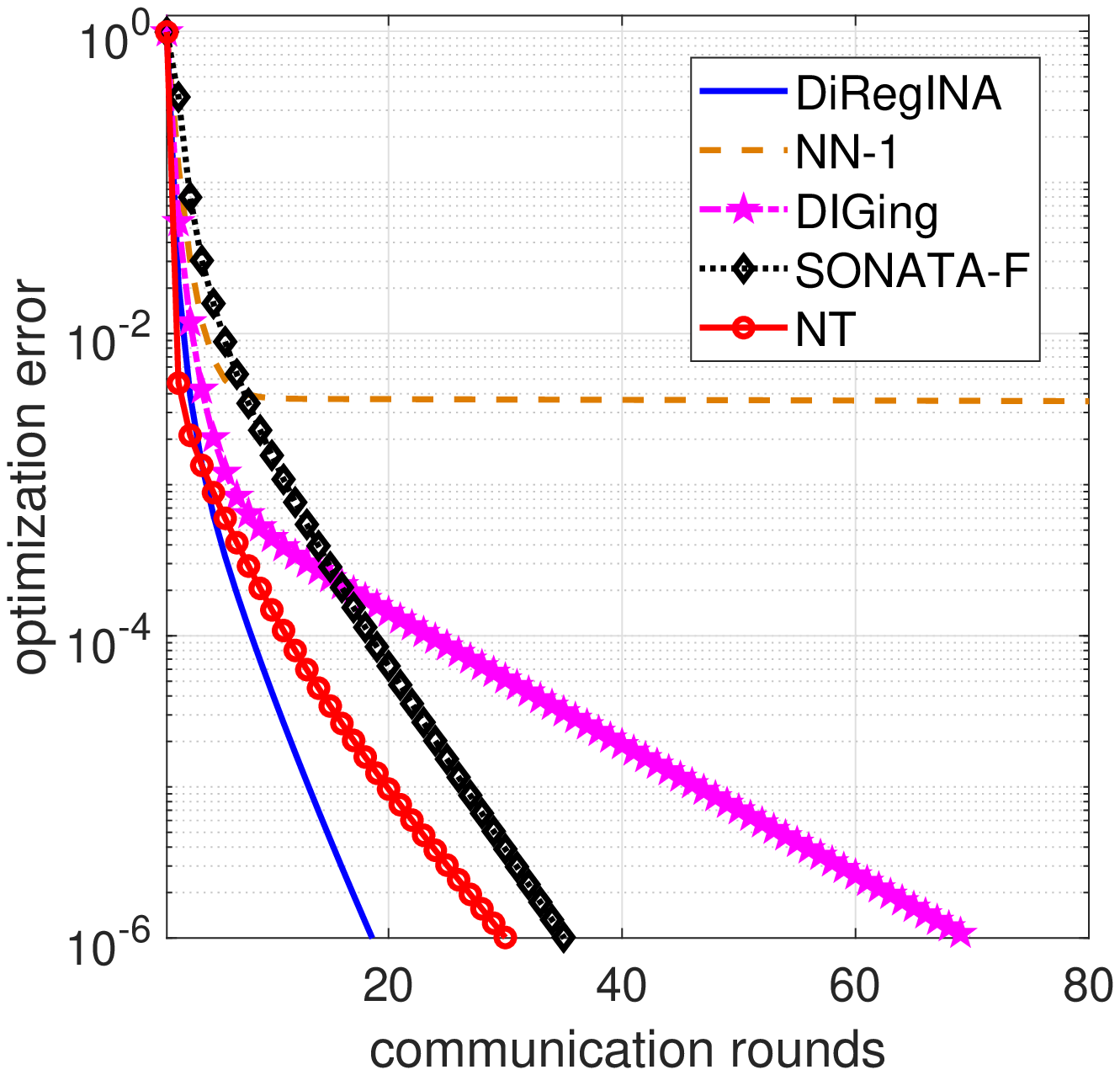}\hspace{-0.4cm}
		\label{RidgeReg_mg3}
	}
	\subfigure[]{
		\includegraphics[width=0.255\textwidth]{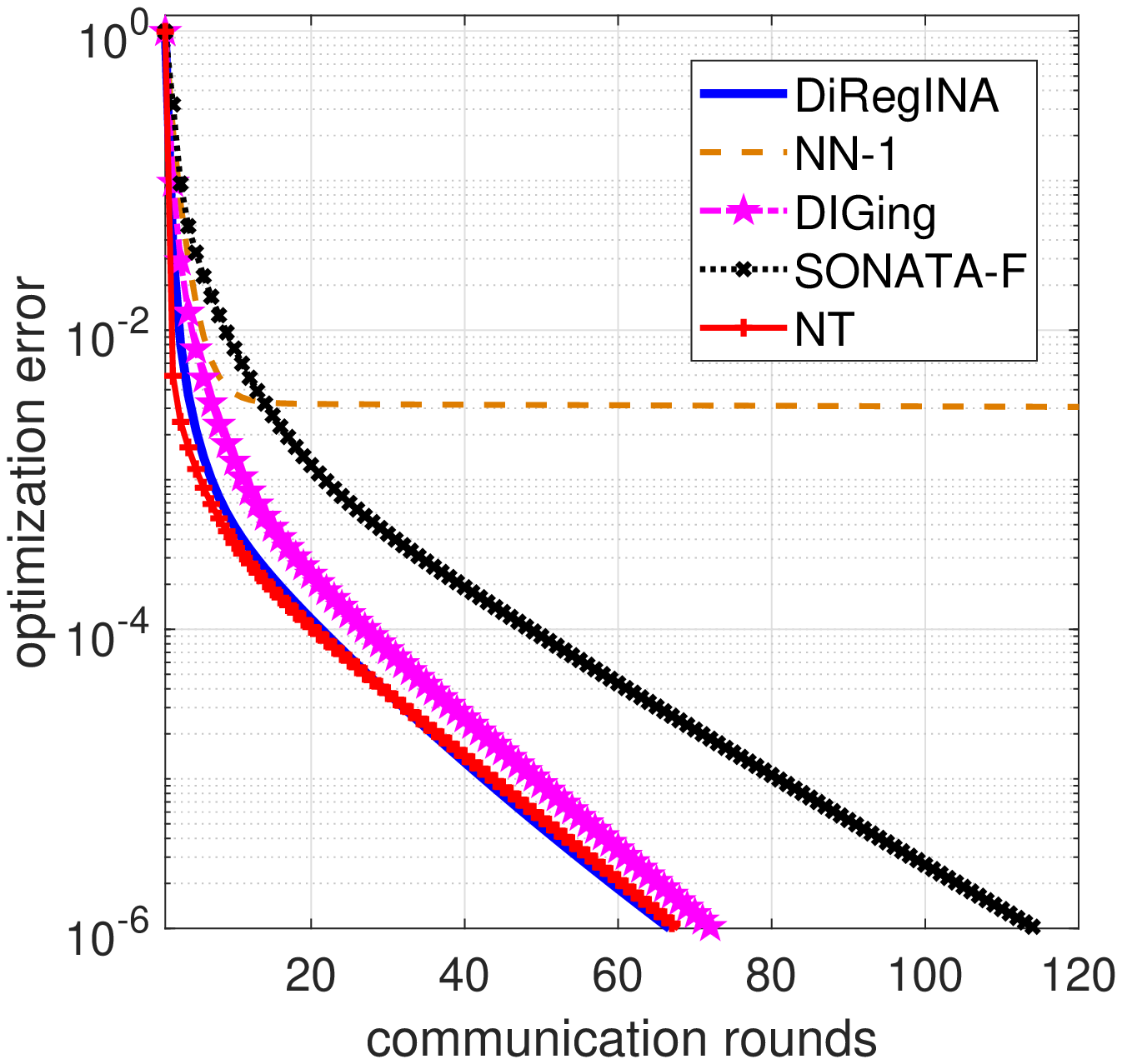}\hspace{-0.4cm}
		\label{RidgeReg_mg4} 
	}\vspace{-0.2cm}
	\caption{Distributed ridge regression:   (a) star-topology; and  Erd\H{o}s-R\'{e}nyi graph with   (b) $\rho=0.20$, (c) $\rho=0.41$, (d)  $\rho=0.69$.}
	\label{fig:RidgReg_mg}
\end{figure*}

\section{Experiments}\vspace{-0.1cm}\label{sec:numerical_experiments}
In this section we test numerically our theoretical findings on two classes of problems over meshed networks:  1)  ridge regression  and    2)   logistic regression. Other experiments can be found in the supplementary material (cf.~ Sec.~\ref{sec:numerical_experiments_appendix}).

The network  graph is generated using an Erd\H{o}s-R\'{e}nyi model $G(m,p)$, with $m=30$ nodes and different values of $p$ to span different level of  connectivity. 

We compare \alg with the following  methods:    
\\$\bullet$ {\it  Distributed (first-order) method with gradient tracking:}    we consider   SONATA \cite{sun2019distributed}   and 
DIGing \cite{nedic2017achieving}; both build on the idea of gradient tracking, with the former applicable also to constrained problems. For the SONATA algorithm, we will simulate two instances, namely: SONATA-L (L stands for linearization) and SONATA-F (F stands for full); the former uses only first-order information in the agents' local updates (as DGing) while the latter exploits functions' similarity by employing local mirror-descent-based optimization. 
\\
 $\bullet$ \emph{Distributed accelerated first-order methods:} we   consider   APAPC \cite{kovalev2020optimal} and   SSDA \cite{scaman2017optimal}, which employ Nesterov acceleration on the local optimization steps--with the former using primal gradients while the latter requiring gradients of the  conjugate functions--and Chebyshev acceleration on the consensus steps. These schemes do not leverage any similarity among the local agents' functions.  
\\
$\bullet$ \emph{Distributed  second-order methods:} We implement  i) Network Newton-K (NN-K) \cite{mokhtari2016network} with $K=1$ so that  it has the same communication cost per iteration of \alg; ii)   SONATA-F \cite{sun2019distributed}, which is a mirror descent-type distributed scheme wherein  agents need to solve \emph{exactly} a strongly convex optimization problem; and iii)   
 Newton Tracking (NT) \cite{So2020},   which has been shown the outperform the majority of distributed second-order methods. 

 All the algorithms are coded in  MATLAB R2019a, running on a computer with Intel(R) Core(TM) i7-8650U CPU@1.90GHz, 16.0 GB of RAM, and 64-bit Windows 10.\vspace{-0.2cm}

\subsection{Distributed Ridge Regression}\label{Regression_simulations} 
We train  ridge regression, LIBSVM, scaled \texttt{mg} dataset \cite{flake2002efficient}, which is an instance of \eqref{eq:P} with $f_i(x)=(1/2n)\norm{A_i x-b_i}^2+\frac{\lambda}{2}\norm{x}^2$   and $\mathcal{K}=\mathbb{R}^d$, with  $d=6$. 
We set $\lambda=1/\sqrt{N}=0.0269$; we estimate  $\beta=0.1457$ and $\mu=0.0929$.  
The graph parameter $p=0.6, 0.33, 0.28$, resulting in the  connectivity values $\rho\approx 0.20, 0.41, 0.70$, respectively.   We compared   DiRegINA, NN-1, DIGing, SONATA-F and NT, all initialized from the same identical random  point.  {The coefficients of the matrix $\overline{W}$ are chosen according to the Metropolis–Hastings rule \cite{xiao2007distributed}}. The free parameters of the algorithm are tuned manually; specifically:     DiRegINA,    $\tau=2\beta$, $M=1e-3$, and $K=1$;  NN-1,   $\alpha=1e-3$ and $\epsilon=1$; 
DIGing, stepsize equal to $0.5$; SONATA-F, $\tau=0.27$;  NT,  $\epsilon=0.08$ and $\alpha=0.1$.  This tuning corresponds to the best practical performance we observed.

In Fig.~\ref{fig:RidgReg_mg}, we plot the function residual $p^\nu$ defined in \eqref{p_nu} versus the   communication rounds  in the four aforementioned network settings. \alg    demonstrates good performance over first-order methods, and compares favorably also with SONATA-F (which has higher computational cost).  Note the change of rate, as predicted by our theory, with linear rate in the last stage.
    NN-1 is not competitive while 
    NT in some settings is comparable with \alg, but we observed to be more sensitive to the tuning.  \vspace{-0.2cm}

\begin{figure}\centering
	\subfigure[]{
		\includegraphics[width=0.25\textwidth]{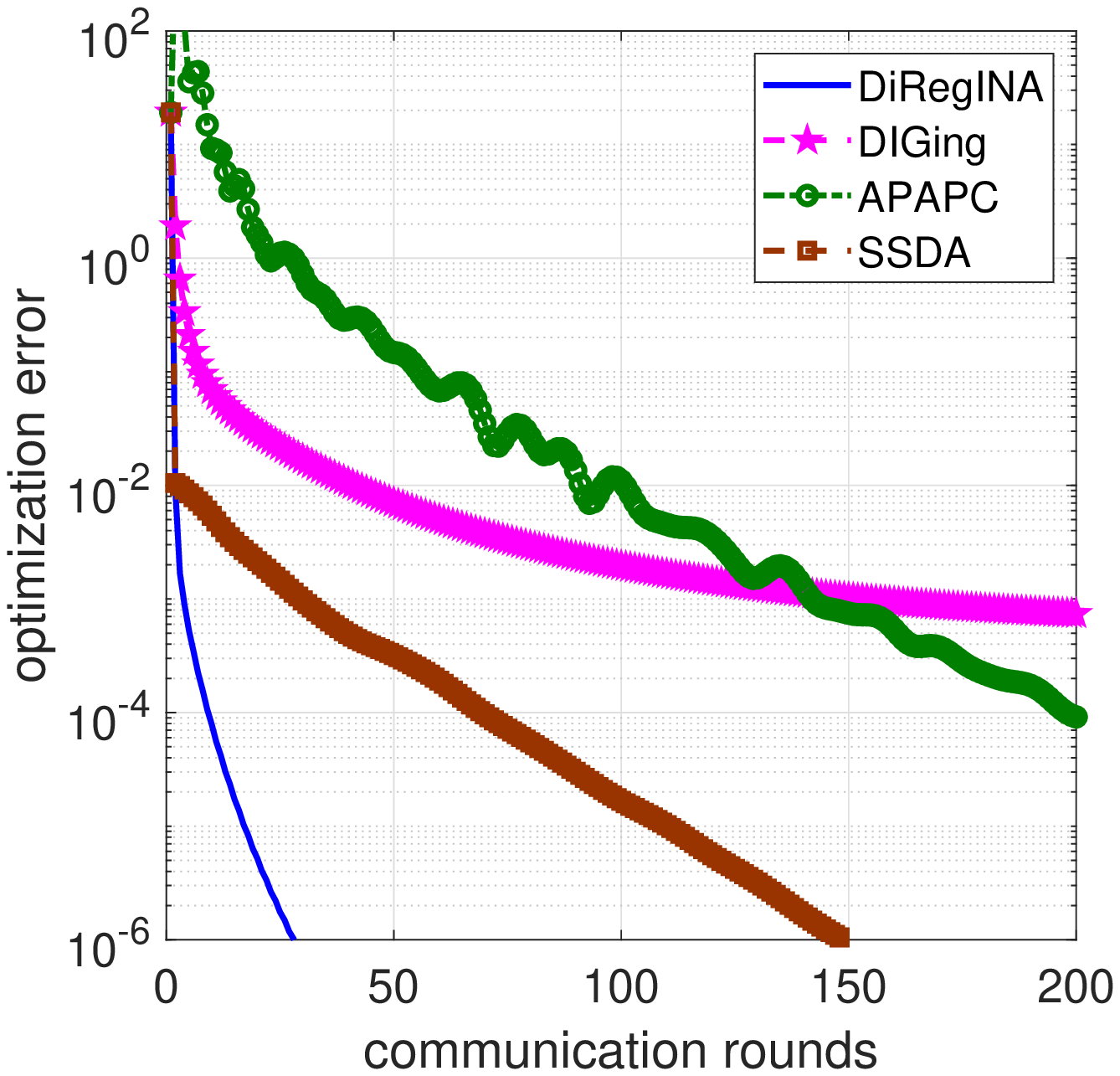}\hspace{-0.4cm}
		\label{RidgeReg_mg3_accelerated}
	}
	\subfigure[]{
		\includegraphics[width=0.25\textwidth]{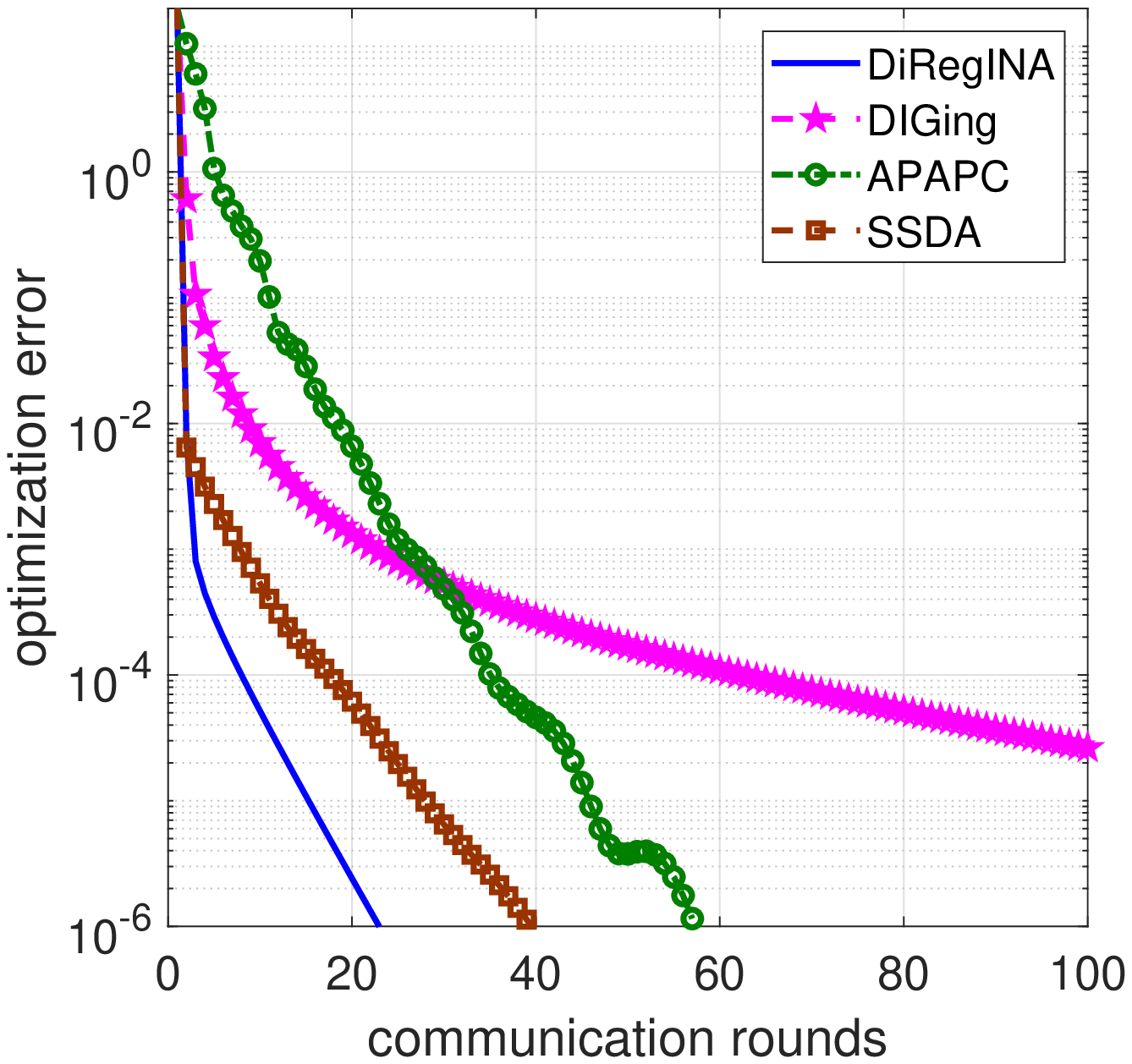}\hspace{-0.4cm}
		\label{RidgeReg_mg4_accelerated}
	}
	\caption{  Distributed ridge regression. Synthetic data on Erd\H{o}s-R\'{e}nyi graph with $\rho = 0.7$: a) $\beta/\mu=158.1$, $\sqrt{\kappa}=34.55$; b) $\beta/\mu=11.974$, $\sqrt{\kappa}=11.1$.}\vspace{-0.9cm}
	\label{fig:LogReg_accelerated}
\end{figure}

\smallskip 

The second experiment aims at comparing 
DiRegINA with   the distributed accelerated  methods    APAPC \cite{kovalev2020optimal} and  SSDA \cite{scaman2017optimal} (DIGing is used as benchmark of first-order non-accelerated schemes).  We tested these schemes on two instances of the Ridge regression problem using  synthetic data, corresponding to     $\beta/\mu \gg \sqrt{\kappa}$ and $\beta/\mu\approx \sqrt{\kappa}$. Recall that  SSDA and APAPC  converge linearly at a rate proportional to  $\sqrt{\kappa}$ while the convergence rate of  \alg depends (up to log factors) on $\beta/\mu$. The problem data are generated as follows:  the  ground truth $x^\ast\in \mathbb{R}^d$ is  a random vector,   $x^\ast\sim \mathcal{N} (\0, \bI)$, with $d=40$;  samples   $b_i\triangleq (b_i^{(j)})_{j=1}^n$, with $n=50$, are generated according to the linear model $b_i^{(j)} = a_i^{(j)\top}\bx^*+\epsilon_i^{(j)}$ where $\epsilon_i^{(j)}\sim\mathcal{N} (0, 1e-4)$. To obtain controlled values for  $\beta$, $A_i\triangleq (a_i^{(j)})_{j=1}^n$  are constructed  as follows:   we first generate $n$ i.i.d samples $A_1\triangleq (a_1^{(j)})_{j=1}^n$, with rows drawn from $\mathcal{N} (\0, \bI)$; then,  we set each    $A_i=A_1+E_k$, where $E_k$ in a random matrix with rows drawn from $\mathcal{N} (\0, \sigma\bI)$. The choices of $\sigma$ are considered resulting in two different values of $\beta$, namely:   $\sigma=1/(dn)$ and $\sigma=7.5/(dn)$, resulting in $\beta=0.31$ and $\beta= 4.08$, respectively.  {The values of the condition number read $\kappa=123.21$ and $\kappa=1.19e3$, respectively.} The network is simulated as  the Erd\H{o}s-R\'{e}nyi graph with $p=0.28$, resulting in $\rho \approx 0.7$; the number of agents is   $m=30$. The tuning of DiRegINA and DIGing is the same as in Fig.~\ref{fig:RidgReg_mg} while 
APAPC and SSDA are manually tuned for best practical performance.  

  In Fig.~\ref{fig:LogReg_accelerated},   {we plot the function residual $p^\nu$ defined in \eqref{p_nu} versus the   communication rounds; the two panels refer to two different values of   $(\beta/\mu,\sqrt{\kappa})$.} 
The figures show that even when  $\beta/\mu$ is  larger than  $\sqrt{\kappa}$, DiRegINA outperforms the accelerated first order methods; roughly, it is   from two to five time faster than the best simulated first order method.

\begin{figure}\centering
	\subfigure[]{
		\includegraphics[width=0.24\textwidth]{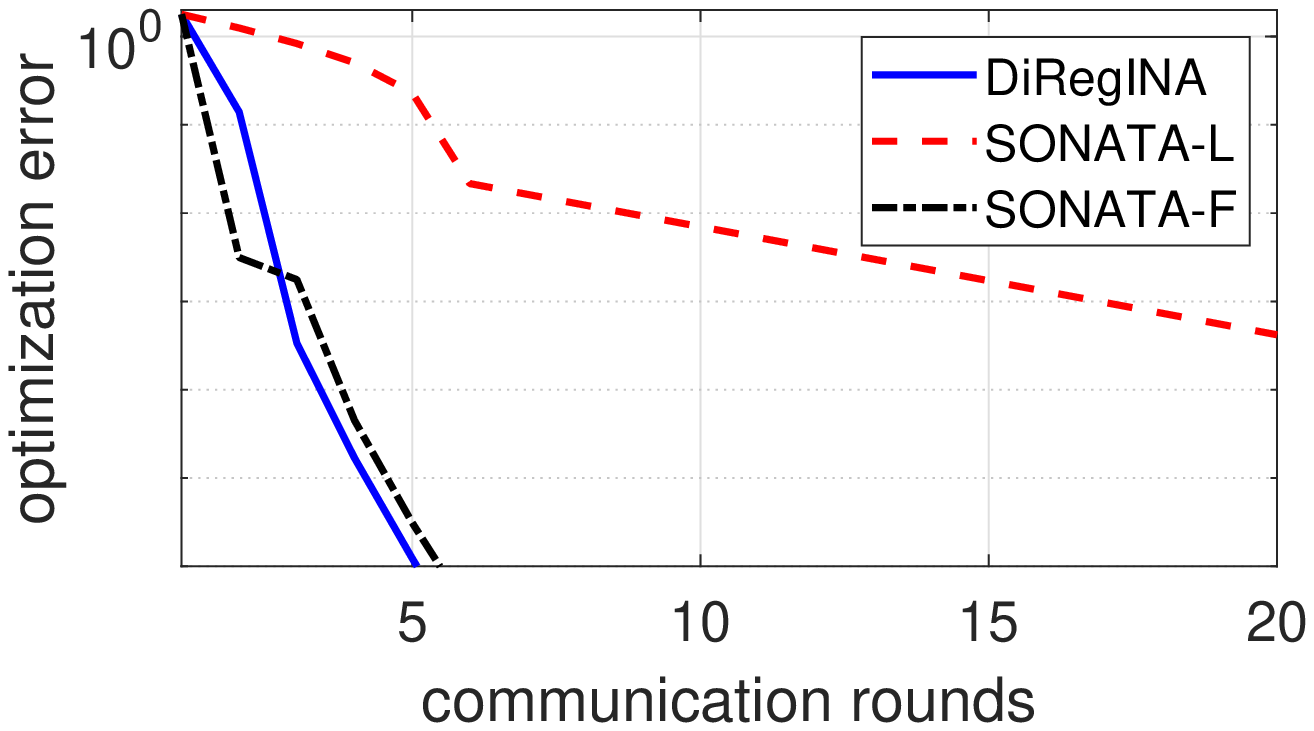}\hspace{-0.4cm}
		\label{LogReg_a4a_1}
	}
	\subfigure[]{
		\includegraphics[width=0.24\textwidth]{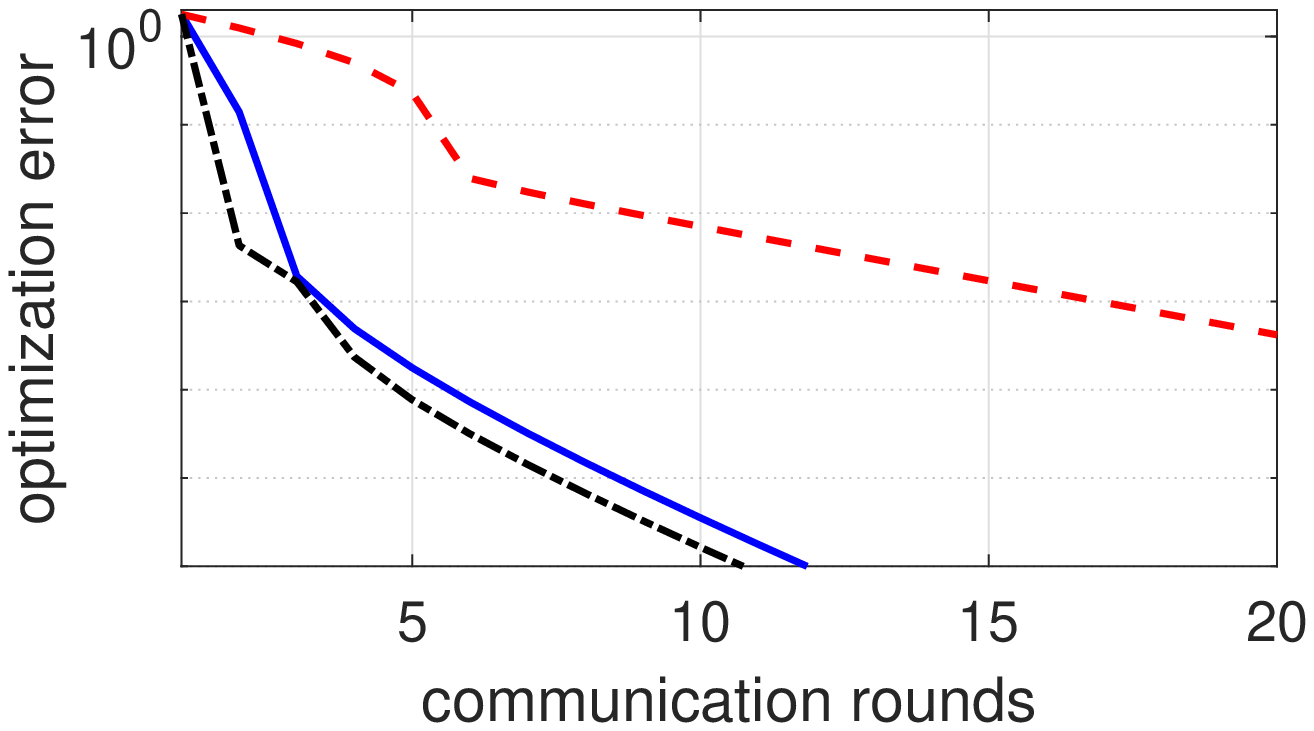}\hspace{-0.4cm}
		\label{LogReg_a4a_2}
	}\\\vspace{-0.3cm}
	\subfigure[]{
	\includegraphics[width=0.24\textwidth]{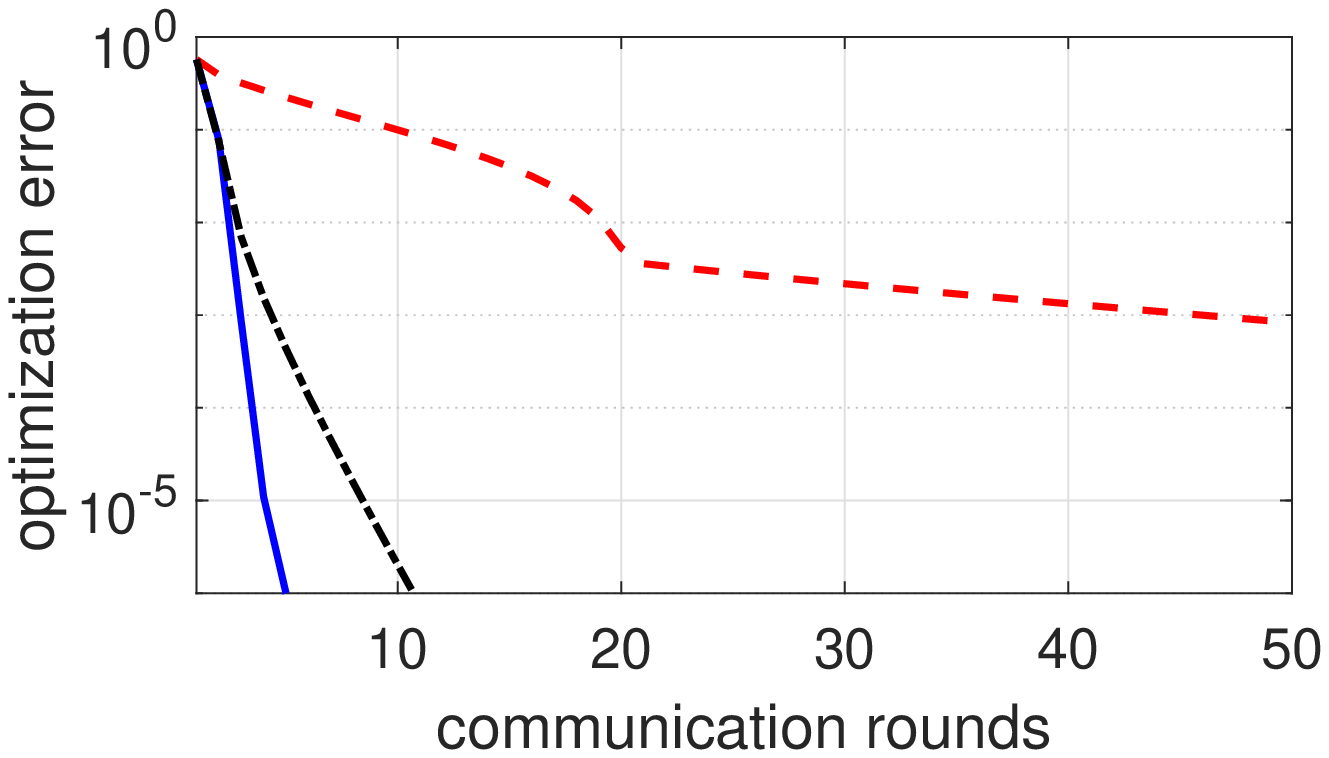}\hspace{-0.4cm}
	\label{LogReg_Synth_1}
}
\subfigure[]{
	\includegraphics[width=0.24\textwidth]{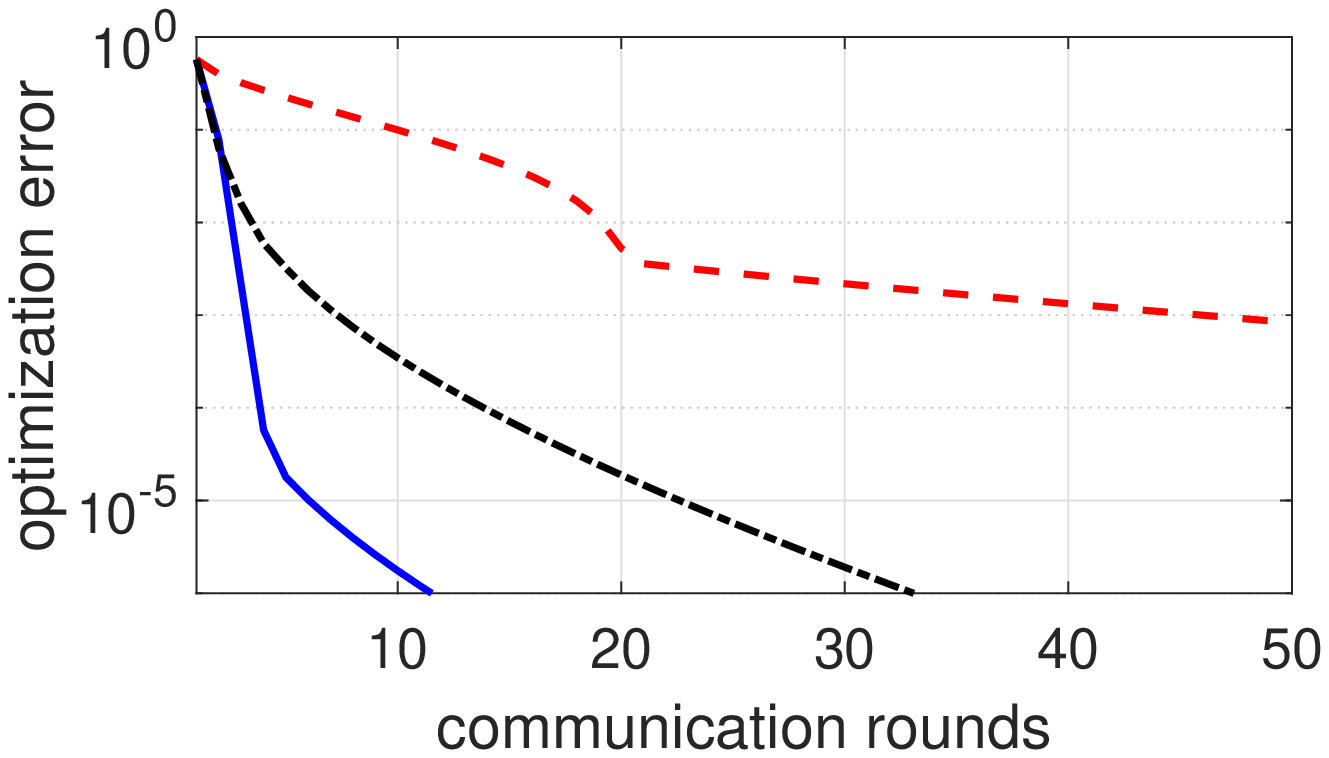}\hspace{-0.4cm}
	\label{LogReg_Synth_2} 
}\vspace{-0.2cm}
	\caption{Distributed logistic regression: 1)   \texttt{a4a} dataset on Erd\H{o}s-R\'{e}nyi graph with   (a) $\rho=0.367$ (b) $\rho=0.757$; 2)   Synthetic data on  Erd\H{o}s-R\'{e}nyi graph with  (c) $\rho=0.367$ (d) $\rho=0.757$.}\vspace{-0.4cm}
	\label{fig:LogReg_Constrauned}
\end{figure}

\subsection{Distributed Logistic Regression}\vspace{-0.1cm}\label{LR_simulations}   
We train logistic regression models, regularized by the $\ell_2$-ball constraint (with radius 1). The problem is an instance of \eqref{eq:P}, with each  $f_i(x)=-(1/n) \sum_{j=1}^n [\xi_i^{(j)}  \ln(z_i^{(j)}) +(1-\xi_i^{(j)}) \ln(1-z_i^{(j)})]$, where $z_i^{(j)}\triangleq 1/(1+e^{-\langle a_i^{(j)},x\rangle})$ and binary class labels $\xi_i^{(j)}\in\{0,1\}$ and vectors  $a_i^{(j)}$, $i=1,\ldots m$ and $j=1,\ldots, n$ are determined by the data set. We considered the  LIBSVM \texttt{a4a}  ($N=4,781$, $d=123$) and synthetic data ($N=900$, $d=150$). The latter are generated as follows:  a random ground truth $x^\ast\sim \mathcal{N} (\0, \bI)$,   i.i.d. sample $\{a_i^{(j)}\}_{i,j}$, and   $\{\xi_i^{(j)}\}_{i,j}$ are generated according to the binary model $\xi_i^{(j)} = 1$ if $\langle a_i^{(j)},\bx^*\rangle \geq 0$ and $\xi_i^{(j)} = 0$ otherwise.   We consider Erd\H{o}s-R\'{e}nyi network models with connectivity  $\rho=0.367$ and  $\rho=0.757$.  

We compare   \alg with SONATA-F and SONATA-L, since they are the only two algorithms in the list that  can handle constrained problems. We report results obtained under the following tuning: (i) both SONATA variants,  $\alpha=0.1$; and (ii)  \alg,   $M=1$ and $\tau_i=1e-3$.  {The coefficients of the matrix $\overline{W}$ are chosen according to the Metropolis–Hastings rule \cite{xiao2007distributed}}.

In Fig. \ref{fig:LogReg_Constrauned},   we plot the function residual $p^\nu$ defined in \eqref{p_nu} versus the   communication rounds, in  the different mentioned network settings. On real data [panels (a)-(b)],  \alg and  SONATA-F performs equally well, outperforming  SONATA-L (first-order method). When tested on   the synthetic problem  [panel  (c)-(d)]  with less local samples  $n$ and larger dimension $d$, \alg shows  a consistently faster rate, while SONATA-F slows down on less connected networks.  Notice also the two-phase rate of \alg, as predicted by our theory:   an initial superlinear rate up to (approximately) the statistical precision, followed by  a linear one for high accuracy.\vspace{-0.2cm}

\section{Conclusions}\vspace{-0.1cm}

We proposed  the first second-order distributed algorithm for convex and strongly convex problems over  meshed networks with \textit{global} communication complexity bounds which, up to the network dependent factor $\widetilde{\mathcal{O}}(1/\sqrt{1-\rho})$, (almost) match the iteration complexity of centralized second-order method \cite{Nesterov--Cubic06} in the regime when the desired accuracy is moderate. We  showed that this regime is reasonable when one considers ERM problems for which there is no need to optimize beyond the statistical error.
Importantly, our method avoids expensive communications of Hessians over the network and keeps the amount of information sent in each communication round similar to first-order methods.

This paper is just a starting point towards a theory of second-order methods with  performance guarantees on meshed networks under statistical similarity;  many questions remain open. An obvious one is incorporating acceleration    to improve communication complexity bounds under statistical similarity.  
{A first attempt towards this goal is the  follow-up work \cite{agafonov2021accelerated}, where  an accelerated second-order method  exploiting  statistical similarity has been analyzed for master/workers architectures. The extension to arbitrary graphs remains an open problem.}
 Second, our main goal here has been  decreasing   communications, which does not guarantee optimal  oracle (computational) complexity--this is because we did not take advantage of the   finite-sum structure of the {\it local}  optimization problems. Stochastic optimization algorithms equipped with Variance Reduction (VR) techniques  have been proved to be quite effective to obtain cheaper iterations while preserving fast convergence  \cite{johnson2013accelerating, hendrikx2020optimal}. However, these methods   do not exploit any statistical similarity, resulting in less favorable communication complexity whenever $\beta/\mu \ll Q/\mu$.  It  would be then interesting to investigate whether   VR techniques    can  improve both communication and oracle complexity when statistical similarity is explicitly employed in the algorithmic design.

\section*{Acknowledgements}
 {The work of A. Daneshmand and G. Scutari was supported by  the USA NSF Grant CIF 1719205, the ARO Grant No. W911NF1810238, and the ONR Grant No. 1317827.} 
 The work of P. Dvurechensky and  A. Gasnikov was prepared within the framework of the HSE University Basic Research Program and is supported by the Ministry of Science and Higher Education of the Russian Federation (Goszadaniye) No. 075-00337-20-03, project No. 0714-2020-0005.

\bibliography{references}
\bibliographystyle{icml2021}


\appendix
\onecolumn
\newpage

\input{appendix.tex}
\end{document}

%% file: appendix.tex
%

{\Large \textbf{Supplementary Material}}

This supplementary material is organized as follows. 
Sec.~\ref{sec:numerical_experiments_appendix} provides additional numerical experiments, complementing those in  Sec.~\ref{sec:numerical_experiments} of the main paper. 
In Sec. \ref{sec:asymptotic_convergence_proof}, we establish  asymptotic convergence of \alg and prove some intermediate results that are instrumental for  our rate analysis.
Sec.~\ref{sec:cvx_case_proof}-\ref{subsec:scvx_beta_geq_mu_Qcase} are devoted to prove   Sec.~\ref{sec:main_results} of the paper, namely:  
Theorem~\ref{thm:cvx_case} is proved in Sec. \ref{sec:cvx_case_proof};  Theorem \ref{thm:scvx_case_beta_leq_mu} and Corollary  \ref{corr:scvx_case_beta_leq_mu_initialization_iid} are proved in Sec. \ref{sec:scvx_case_proof}; and finally, Theorem \ref{thm:scvx_case_beta_geq_mu} is proved in Sec. \ref{subsec:scvx_beta_geq_mu}.

Furthermore, there are some convergence results stated in  Table \ref{table:rate} that could  not be stated in the paper because of space limit; they are reported here  in the following sections: i) the case of quadratic functions  $f_i$ in the setting of  Theorem \ref{thm:scvx_case_beta_leq_mu} is stated in Theorem \ref{thm:scvx_case_beta_leq_mu_Qcase} in Sec. \ref{sec:scvx_case_beta_leq_mu_Qcase_proof} while the  case of quadratic $f_i$'s in the setting of  Theorem \ref{thm:scvx_case_beta_geq_mu} is stated in Theorem \ref{thm:scvx_case_beta_geq_mu_QUAD},   Sec. \ref{subsec:scvx_beta_geq_mu_Qcase}.


\section{Additional Numerical Experiments}\label{sec:numerical_experiments_appendix}
\subsection*{Convex (non-strongly convex) objective}
We consider a (non-strongly) convex  instance of  the   regression problem. Specifically, we have:   $f_i(x)=(1/2n)\norm{A_i x-b_i}^2$ and $\mathcal{K}=\mathbb{R}^d$, where $A_i$ and $b_i$ are determined by the scaled LIBSVM dataset \texttt{space-ga} ($N=3107$, $d=6$, and   $\beta=0.6353$). The network is simulated as the Erd\H{o}s-R\'{e}nyi network model, with $m=30$ and two  connectivity  values, $\rho=0.3843$ and  $\rho=0.8032$.  We compared \alg with the algorithms described in  Sec.~\ref{sec:main_results}, namely: NN-1, NT,   DIGing and SONATA-F. Note that NN-1 and   NT are not guaranteed  to converge when applied to convex (non-strongly convex) functions. The tuning of the algorithm is the same as the one described in  Sec.~\ref{Regression_simulations}. In Fig. \ref{fig:RidgeReg_space_ga}, we plot the optimization error versus the communication rounds achieved by the aforementioned algorithms in the two network settings, $\rho=0.3843$ and  $\rho=0.8032$.  As already observed for the other simulated problems (cf. Sec.~\ref{Regression_simulations}),  SONATA-F   shows similar performance of \alg when running on well-connected networks while its performance deteriorates in poorly connected network.  NT seems to be non-convergent while NN1 and   DIGing converge, yet slow, to acceptable accuracy.

 \begin{figure}[h!]\centering
\vspace{-.5cm}
	\subfigure[]{
		\includegraphics[width=0.4\textwidth]{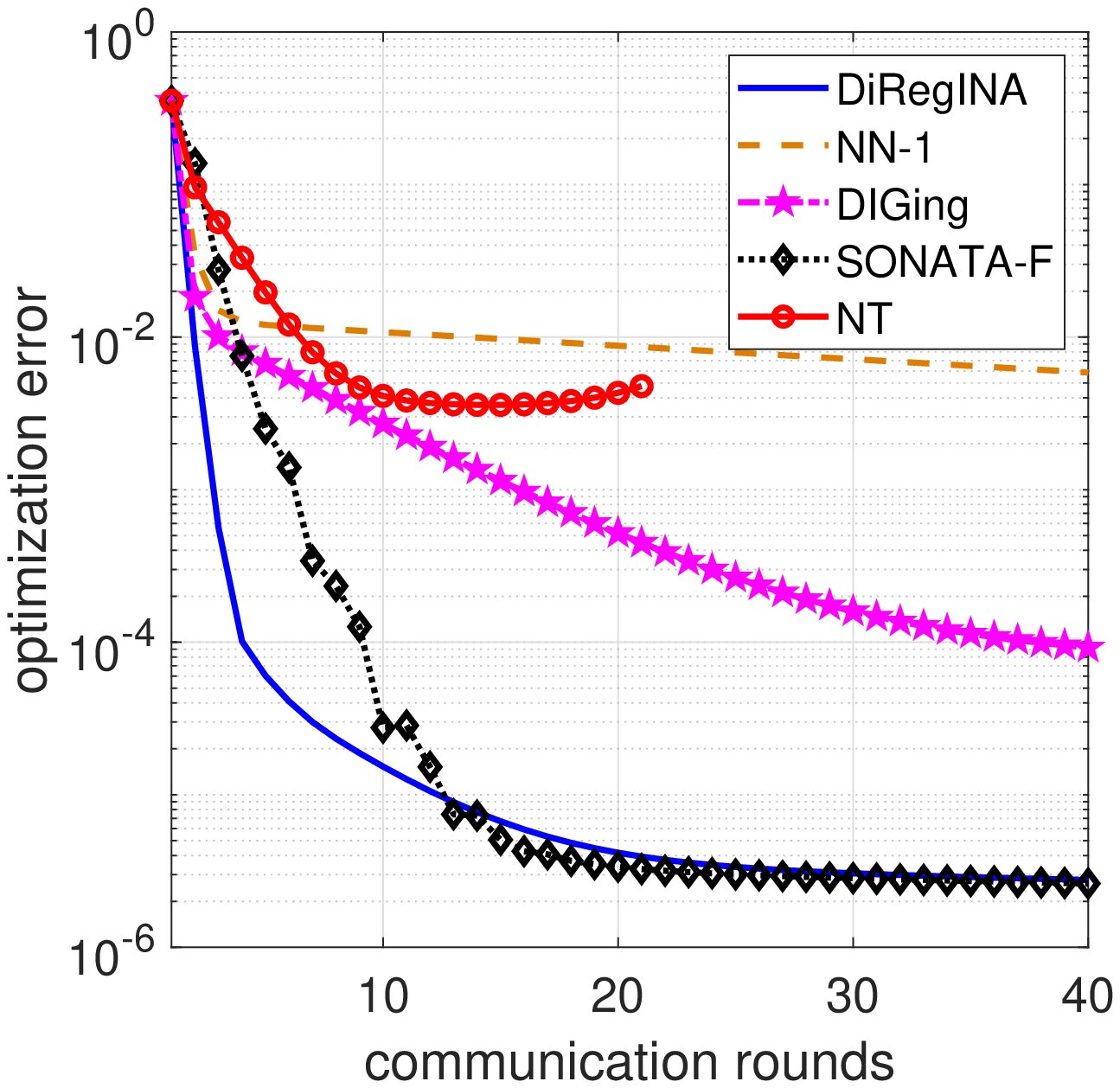}\hspace{-0.4cm}
		\label{RidgeReg_space_ga_1}
	}
	\subfigure[]{
		\includegraphics[width=0.4\textwidth]{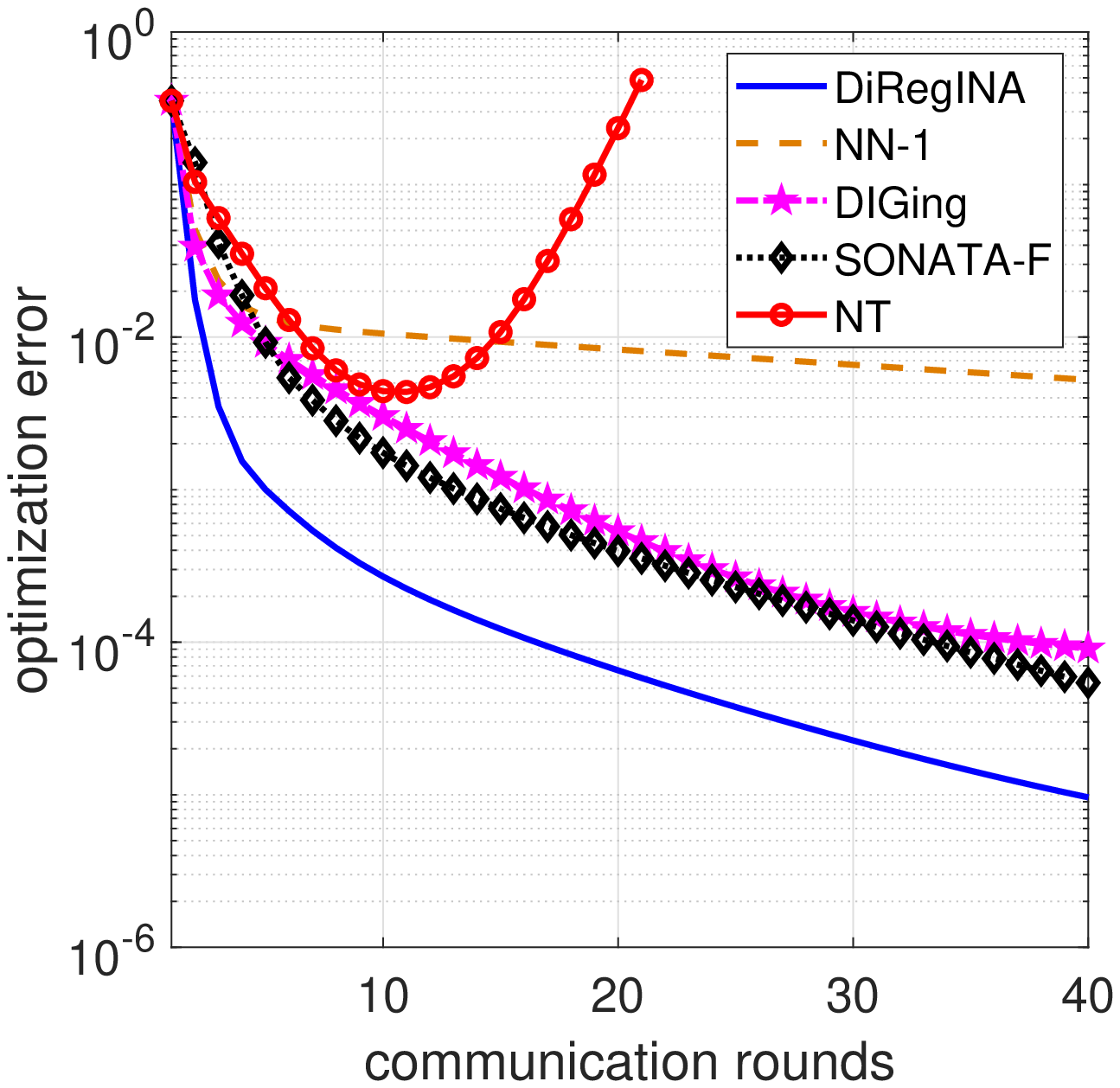}\hspace{-0.4cm}
		\label{RidgeReg_space_ga_2}
	}\\\vspace{-0.3cm}
	\caption{Distributed ridge regression  on     \texttt{space-ga} dataset and Erd\H{o}s-R\'{e}nyi graph with    (a) $\rho=0.3843$ (b) $\rho=0.8032$.}
	\label{fig:RidgeReg_space_ga}
\end{figure}

\subsection*{$O(1/\sqrt{mn})$-regularized logistic regression}
We train logistic regression models, regularized by an additive $\ell_2$-norm (with coefficient $\lambda>0$). The problem is an instance of \eqref{eq:P}, with each  $f_i(x)=-(1/n) \sum_{j=1}^n [\xi_i^{(j)}  \ln(z_i^{(j)}) +(1-\xi_i^{(j)}) \ln(1-z_i^{(j)})]+(\lambda/2)||x||^2$ and $\mathcal{K}=\mathbb{R}^d$, where $z_i^{(j)}\triangleq 1/(1+e^{-\langle a_i^{(j)},x\rangle})$ and binary class labels $\xi_i^{(j)}\in\{0,1\}$ and vectors  $a_i^{(j)}$, $i=1,\ldots m$ and $j=1,\ldots, n$ are determined by the data set.  We considered the  LIBSVM \texttt{a4a}  ($N=4,781$, $d=123$) and we set $\lambda=1/\sqrt{mn}$.  The Network is simulated according to the  Erd\H{o}s-R\'{e}nyi model with $m=30$ and connectivity  $\rho=0.3372$ and  $\rho=0.7387$.   

We compare   \alg, NN-1, DIGing, SONATA-F and NT, all initialized from the same   random  point. The free parameters of the algorithms are tuned manually;  the best practical performance are observed with the following tuning:     \alg is tuned as described in Sec.~\ref{LR_simulations},  i.e.,  $\tau=1$, $M=1e-3$, and $K=1$;  NN-1,   $\alpha=1e-3$ and $\epsilon=1$; 
DIGing, stepsize equal to $1$; SONATA-F, $\tau=0.1$;  NT,  $\epsilon=0.2$ and $\alpha=0.05$.  

In Fig. \ref{fig:RidgeReg_space_ga}, we plot  the optimization error versus the  communication rounds achieved by the aforementioned algorithms in  two network settings corresponding to $\rho=0.3372$ and  $\rho=0.7387$.    In both settings (panels  (a)-(b)), NN-1 and DIGing still exhibits slow convergence, with a slight advantage  of DIGing over NN-1.    \alg, NT and SONATA-F,  perform  similarly, with   \alg showing some improvements when the network is better connected [panel (a)]. \bigskip

\begin{figure}\centering
	\subfigure[]{
		\includegraphics[width=0.4\textwidth]{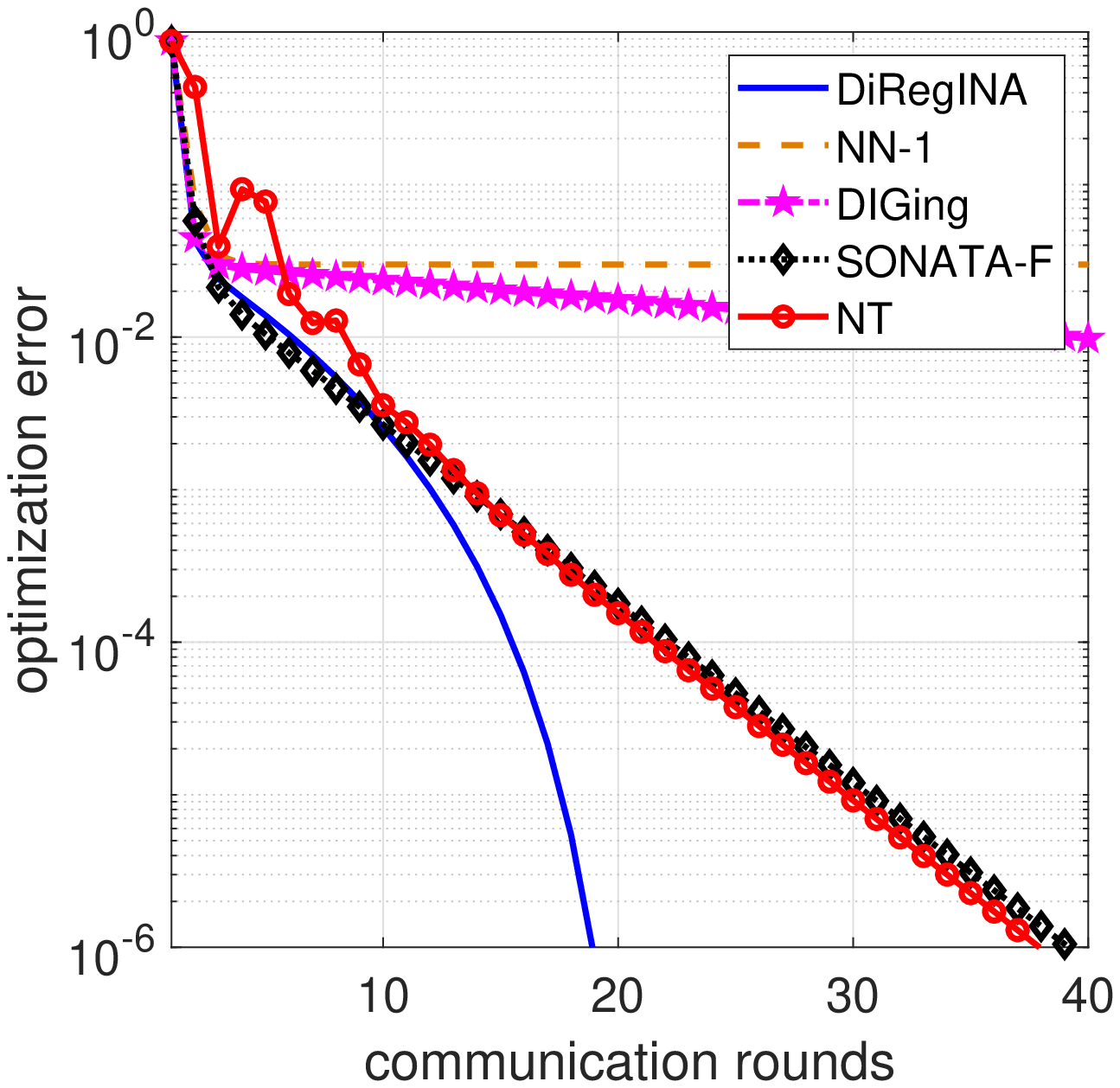}\hspace{-0.4cm}
		\label{LogReg_reg_a4a_1}
	}
	\subfigure[]{
		\includegraphics[width=0.4\textwidth]{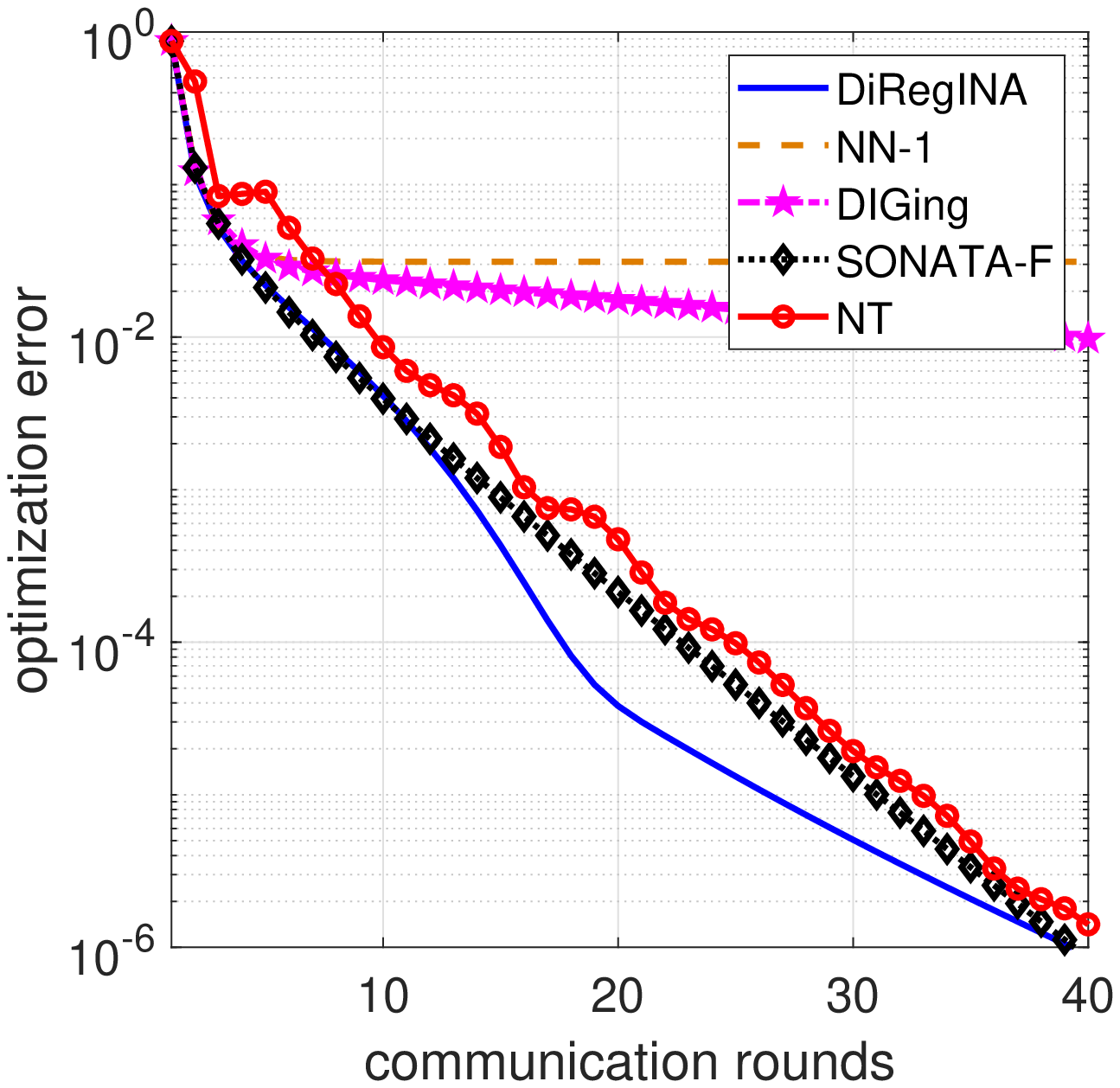}\hspace{-0.4cm}
		\label{LogReg_reg_a4a_2}
	}\\\vspace{-0.3cm}
	\caption{Distributed logistic regression  on    \texttt{a4a} dataset and Erd\H{o}s-R\'{e}nyi graph with  (a) $\rho=0.3372$ (b) $\rho=0.7387$.}
	\label{fig:LogReg_reg_a4a}
\end{figure}

\section{Notations and Preliminary Results}\label{sec:prem_results}
We begin introducing some notation  which will be used in all the proofs,  along with some preliminary results.

Define \begin{equation}\label{delta_i_def}
\bepsilon_i^\nu\triangleq \gradtrack_i^\nu-\nabla F(\bx_i^\nu) \quad  \text{and} \quad  \bbeta_i^\nu\triangleq \nabla^2f_i(\bx_i^\nu)-\nabla^2F(\bx_i^\nu),
\end{equation}
The local surrogate function $\tF_i(\by;\bx_i^\nu)$   in \eqref{def:x_nuplus} can be rewritten as
\begin{equation}\label{tF_i_rewrite}
\tF_i(\by;\bx_i^\nu)\triangleq F(\bx_i^\nu)+\left\langle \nabla F(\bx_i^\nu)+\bepsilon_i^\nu,\by-\bx_i^\nu\right\rangle+\frac{1}{2}\left\langle \left[\nabla^2F(\bx_i^\nu)+\bbeta_i^\nu+\tau_i\bI\right] (\by-\bx_i^\nu),\by-\bx_i^\nu\right\rangle+\frac{M_i}{6}\norm{\by-\bx_i^\nu}^3.
\end{equation}

Let us recall the following basic result, which is a  consequence of Assumption~\ref{assump:nabla2F_LC}.
\begin{lemma}[{\citet[Lemma 1.2.4]{nesterov2018lectures}}]
	\label{res_lemma}
	Let $F:\mathbb{R}^d\rightarrow \mathbb{R}$ be a twice-differentiable function satisfying Assumption \ref{assump:nabla2F_LC}. Then, for all $\bx,\by\in\mathbb{R}^d$,
	\begin{align}
	&\big|F(\by)-F(\bx)-\left\langle \nabla F(\bx),\by-\bx\right\rangle -\frac{1}{2}\left\langle \nabla^2 F(\bx)(\by-\bx),\by-\bx\right\rangle\big|\leq  \frac{\Lhessian}{6}\norm{\by-\bx}^3.
	\label{res_2}
	\\
	&\norm{\nabla F(\by)-\nabla F(\bx)-\nabla^2 F(\bx)(\by-\bx)}\leq  \frac{\Lhessian}{2}\norm{\by-\bx}^2.
	\label{res1}
	\end{align}
\end{lemma}

Setting $\bx=\bx_i^\nu$ in \eqref{res_2} implies
\begin{equation*}
F(\bx_i^\nu)+\left\langle \nabla F(\bx_i^\nu),\by-\bx_i^\nu\right\rangle +\frac{1}{2}\left\langle \nabla^2 F(\bx_i^\nu)(\by-\bx_i^\nu),\by-\bx_i^\nu\right\rangle\leq F(\by)+ \frac{\Lhessian}{6}\norm{\by-\bx_i^\nu}^3,\quad \forall \by\in\mathbb{R}^d,
\end{equation*}
which, together with \eqref{tF_i_rewrite}, gives the following upper bound for the surrogate function $\tF_i$ defined in \eqref{tF_i_rewrite}:
\begin{equation}
\label{F_tilde_upperbound}
\begin{aligned}
\tF_i(\by;\bx_i^\nu) \leq & F(\by)+\frac{1}{2}\norm{\by-\bx_i^\nu}^2_{(\beta+\tau_i)\bI}+\frac{M_i+\Lhessian}{6}\norm{\by-\bx_i^\nu}^3+\left\langle \bepsilon_i^\nu, \by-\bx_i^\nu\right\rangle,\quad \forall \by\in\mathbb{R}^d,
\end{aligned}
\end{equation}
where for a positive semidefinite matrix $A$, $\|x\|_A^2 \triangleq \langle Ax,x \rangle$.
We also denote
\begin{equation}\label{def:p_deltaX}
\Delta x_i^\nu\triangleq \bx_i^{\nu+}-\bx_i^\nu,\quad \delta^\nu\triangleq (\delta_i^\nu)_{i=1}^m, \quad J\triangleq 11^\top/m,
\end{equation}
where we remind that $\bx_i^{\nu+}$ is obtained by the minimization of the local surrogate function $\tF_i(\by;\bx_i^\nu)$. 
The rest of the symbols and notations are as defined in the main manuscript.

\section{Asymptotic convergence of \alg}\label{sec:asymptotic_convergence_proof}

In this section we prove the following theorem stating   asymptotic convergence of \alg.
\begin{theorem}\label{thm:asymptotic_conv}
	Let Assumptions \ref{convex-case} and \ref{assump:nabla2F_LC}-\ref{assump:network} hold,  $M_i\geq \Lhessian$ and $\tau_i=2\beta$ for all $i=1,\ldots,m$. If a reference matrix $\overline{W}$ satisfying Assumption \ref{assump:weight} is used   in steps \eqref{DiRegINA_pseudocode}-\eqref{grad_track_update},  with $\rho  \triangleq  \lambda_{\max}(\overline{W} - \bJ) < 1$ and     $K=\widetilde{\mathcal{O}}(1/\sqrt{1-\rho})$ (explicit condition is provided in eq.  \eqref{K_cond_asymptotic}), then $p^\nu\rightarrow 0$ and  $||\bx_i^\nu-\bx_j^\nu||\rightarrow 0$, as  $\nu\rightarrow \infty$ for all $i,j=1,\ldots,m$.
\end{theorem}

We prove the theorem in three main steps: \begin{itemize}
    \item[]  {\bf Step 1 (Sec. \ref{sec:opt_err_bounds}):} Deriving optimization   bounds  on the  per-iteration decrease of $p^{\nu}$;
    \item[]  {\bf Step 2 (Sec. \ref{sec:net_err_bounds}):}  Bounding the gradient tracking error  $\delta^\nu$,     which in turn affects the per-iteration decrease of $p^{\nu}$;
    \item[]  {\bf Step 3 (Sec. \ref{sec:asymptotic_convergence_Step3}):} Constructing a proper Lyapunov function based  on the error terms in the previous two steps, whose dynamics imply asymptotic convergence of \alg.   
\end{itemize}    
To simplify the derivations, we study the case of strongly convex or nonstronlgy convex $F$ together, by setting $\mu=0$ in the latter case. 

\subsection{Optimization error bounds}\label{sec:opt_err_bounds}
In this subsection we establish an upper bound for $p^{\nu+1}-p^\nu$ [cf.~\eqref{Descent_Final3}]. We begin with two technical intermediate results--Lemma  \ref{lemma:descent_cubic} and Lemma~\ref{lemma_descent_x_plus_II}.
\begin{lemma}\label{lemma:descent_cubic}
	Under Assumption \ref{convex-case}, there holds 
	\begin{equation}
	\label{eq:sufficient_descent_dampedNewton_fi}
	\begin{aligned}
	\tF_i (\bx_i^{\nu+} ;  \bx_i^\nu) 
	\leq & \; \tF_i (\bx_i^\nu ;\bx_i^\nu) -\frac{M_i}{3}\norm{\Delta\bx_i^\nu}^3-\frac{\mu_i+\tau_i}{2}\norm{\Delta\bx_i^\nu}^2.
	\end{aligned}
	\end{equation}
\end{lemma}
\begin{proof}
	By the optimality of $\bx_i^{\nu+}$ 
	in \eqref{tF_i_rewrite},  we infer
	\begin{equation}
	\label{x_hat_opt_cubic}
	\left\langle \gradtrack_i^\nu+\left[\nabla^2 f_i(\bx_i^\nu)+\tau_i\bI\right] \Delta\bx_i^\nu, \Delta\bx_i^\nu\right\rangle\leq -\frac{M_i}{2}\norm{\Delta\bx_i^\nu}^3.
	\end{equation}
	Since $\tF_i(\bx_i^\nu;\bx_i^\nu)=F(\bx_i^\nu)$, we have
	\begin{equation*}
	\begin{aligned}
	& \tF_i(\bx_i^{\nu+};\bx_i^\nu)-\tF_i(\bx_i^\nu;\bx_i^\nu)
	\\
	 \overset{\eqref{tF_i_rewrite}}{=}  &\left\langle \gradtrack_i^\nu ,\bx_i^{\nu+}-\bx_i^\nu\right\rangle+\frac{1}{2}\left\langle \left[\nabla^2 f_i(\bx_i^\nu)+\tau_i\bI\right] \Delta\bx_i^\nu,\Delta\bx_i^\nu\right\rangle+\frac{M_i}{6}\norm{\bx_i^{\nu+}-\bx_i^\nu}^3 
	\\
	 \overset{\eqref{x_hat_opt_cubic}}{\leq} & -\frac{1}{2}\left\langle \left[\nabla^2 f_i(\bx_i^\nu)+\tau_i\bI\right] \Delta\bx_i^\nu,\Delta\bx_i^\nu\right\rangle 
	-\frac{M_i}{3}\norm{\Delta\bx_i^\nu}^3
	\\
	 {\leq}  & -\frac{M_i}{3}\norm{\bx_i^{\nu+}-\bx_i^\nu}^3-\frac{\mu_i+\tau_i}{2}\norm{\bx_i^{\nu+}-\bx_i^\nu}^2.
	\end{aligned}
	\end{equation*}
	

\end{proof}

\begin{lemma}\label{lemma_descent_x_plus_II}
	Let Assumptions \ref{convex-case} and \ref{assump:nabla2F_LC}-\ref{assump:homogeneity} hold. Then, any  arbitrary $\epsilon>0$, we have 
	\begin{equation}
	\label{Taylor_exp_Fxi_subt_simplified}
	\begin{aligned}
	F (\bx_i^{\nu+}) -\tF_i (\bx_i^{\nu+};\bx_i^\nu) 
	\leq & -\frac{M_i-\Lhessian}{6} ||\Delta\bx_i^\nu||^3 - \frac{\tau_i-\beta-\epsilon}{2}||\Delta\bx_i^\nu||^2+ \frac{1}{2\epsilon}\norm{\bepsilon_i^\nu}^2.
	\end{aligned}
	\end{equation}
\end{lemma}
\begin{proof}
Taylor's theorem applied to functions $\tF_i(\cdot;\bx_i^\nu)$ and $F(\cdot)$ around $\bx_i^{\nu}$ yields
	\begin{subequations}
		\begin{align}
	\label{Taylor_exp_Fxi}
	F (\bx_i^{\nu+})  = & F(\bx_i^\nu) +\left\langle\nabla F(\bx_i^\nu), \Delta \bx_i^\nu\right\rangle+\Delta \bx_i^{\nu\top} \bH_i^\nu\Delta \bx_i^\nu,
\\
	\label{Taylor_exp_tFxi}
	\tF_i (\bx_i^{\nu+};\bx_i^\nu)  = & \tF_i(\bx_i^\nu; \bx_i^\nu) +\left\langle\nabla \tF_i(\bx_i^\nu;\bx_i^\nu), \Delta \bx_i^\nu\right\rangle+\Delta \bx_i^{\nu\top} \widetilde{\bH}_i^\nu\Delta \bx_i^\nu,
	\end{align}
	\end{subequations}
	where 
	\begin{equation*}
	\begin{aligned}
	\bH_i^\nu = & \int_{0}^1 (1- \theta)\nabla^2 F (\theta \bx_i^{\nu+} + (1 - \theta) \bx_i^\nu) d \theta,
\\
	\widetilde{\bH}_i^\nu = & \int_{0}^1 (1 - \theta)\nabla^2 \tF_i(\theta \bx_i^{\nu+} + (1 - \theta) \bx_i^\nu ; \bx_i^\nu) d \theta.
	\end{aligned}
	\end{equation*}
	Since $\tF_i(\bx_i^\nu; \bx_i^\nu)=F(\bx_i^\nu)$ and $\nabla \tF_i(\bx_i^\nu;\bx_i^\nu)=\nabla F(\bx_i^\nu)+\bepsilon_i^\nu$, subtracting \eqref{Taylor_exp_Fxi}-\eqref{Taylor_exp_tFxi} gives
	\begin{align}
	\label{Taylor_exp_Fxi_subt}
	F (\bx_i^{\nu+}) -\tF_i (\bx_i^{\nu+};\bx_i^\nu) = & \big\langle\big( \bH_i^\nu-\widetilde{\bH}_i^\nu \big)\Delta \bx_i^\nu,\Delta \bx_i^\nu\big\rangle-\left\langle\bepsilon_i^\nu,\Delta \bx_i^\nu\right\rangle.
	\end{align}
	
	Now let us simplify \eqref{Taylor_exp_Fxi_subt}. Note that the hessian of $\tF_i(\cdot;\bx_i^\nu)$ is
	\begin{equation}\label{nabla2Ftilde_hessian}
	\nabla^2\tF_i(\bx_i;\bx_i^\nu)=\nabla^2 F(\bx_i^\nu)+\bbeta_i^\nu+\tau_i\bI+M_iG(\bx_i;\bx_i^\nu),
	\end{equation}
	where
	\begin{equation*}
	G\left(\bx_i;\bx_i^\nu\right)\triangleq \frac{1}{2}\left(\norm{\bx_i-\bx_i^\nu} \bI+\frac{(\bx_i-\bx_i^\nu)(\bx_i-\bx_i^\nu)^\top}{\norm{\bx_i-\bx_i^\nu}}\right).
	\end{equation*}
	Hence,
	\begin{align}
	\begin{split}
	&  \bH_i^\nu - \widetilde{\bH}_i^\nu
	\\
	= &  \int_0^1 (1  - \theta)\nabla^2F \left(\theta   \bx_i^{\nu+} + (1 - \theta  )\bx_i^\nu \right) d \theta
	-\int_0^1 (1 - \theta)\nabla^2 \tF_i \left(\theta \bx_i^{\nu+} + (1 - \theta  )\bx_i^\nu;\bx_i^\nu \right) d \theta
	\\
	\overset{\eqref{nabla2Ftilde_hessian}}{=} &  \int_0^1 (1  - \theta)\nabla^2F \left(\theta   \bx_i^{\nu+} + (1 - \theta  )\bx_i^\nu \right) d \theta
	-\int_0^1 (1 - \theta)\left[\nabla^2 F (\bx_i^\nu)+\bbeta_i^\nu\right] d \theta -\int_0^1 (1 - \theta)\tau_i\bI d \theta
	\\
	&-\int_0^1 (1 - \theta)M_i\theta G(\bx_i^{\nu+};\bx_i^\nu) d \theta
	\\
	= &  \int_0^1 (1  - \theta)\left(\nabla^2F \left(\theta   \bx_i^{\nu+} + (1 - \theta  )\bx_i^\nu \right)-\nabla^2F(\bx_i^\nu) \right)d \theta
	\\
	&-\int_0^1 (1 - \theta)\bbeta_i^\nu d \theta-\int_0^1 (1 - \theta)\tau_i\bI d \theta-\int_0^1 (1 - \theta)M_i\theta G(\bx_i^{\nu+};\bx_i^\nu) d \theta
	\\
	\overset{{\rm(a)}}{\preceq} &  \int_0^1 (1  - \theta)\Lhessian\theta||\bx_i^{\nu+}-\bx_i^\nu||  \bI  d \theta
	\\
	&-\int_0^1 (1 - \theta)\bbeta_i^\nu d \theta-\int_0^1 (1 - \theta)\tau_i\bI d \theta-\int_0^1 (1 - \theta)M_i\theta G(\bx_i^{\nu+};\bx_i^\nu) d \theta
	\\
	=&-\frac{M_i}{6} G(\bx_i^{\nu+};\bx_i^\nu)+\frac{\Lhessian}{6} ||\bx_i^{\nu+}-\bx_i^\nu|| \bI   -\frac{\tau_i}{2}\bI-\frac{\bbeta_i^\nu}{2}
	\end{split}
	\label{alph_H_exp}
	\end{align}
	where (a) holds since $\nabla^2F$ is $\Lhessian$-Lipschitz continuous.
	Combining  \eqref{Taylor_exp_Fxi_subt} and  \eqref{alph_H_exp}, we conclude
	\begin{equation*}
	\begin{aligned}
	F (\bx_i^{\nu+}) -\tF_i (\bx_i^{\nu+};\bx_i^\nu) \leq & -\frac{M_i-\Lhessian}{6} ||\Delta\bx_i^\nu||^3   - \frac{\tau_i}{2}||\Delta\bx_i^\nu||^2-\frac{1}{2}\left\langle \bbeta_i^\nu\Delta \bx_i^\nu,\Delta \bx_i^\nu\right\rangle-\left\langle\bepsilon_i^\nu,\Delta \bx_i^\nu\right\rangle
	\\
	\leq & -\frac{M_i-\Lhessian}{6} ||\Delta\bx_i^\nu||^3 - \frac{\tau_i-\beta-\epsilon}{2}||\Delta\bx_i^\nu||^2+ \frac{1}{2\epsilon}\norm{\bepsilon_i^\nu}^2,
	\end{aligned}
	\end{equation*}
for arbitrary $\epsilon>0$, where the last inequality is due to the Cauchy-Schwarz inequality and   $|\left\langle \bbeta_i^\nu\Delta \bx_i^\nu,\Delta \bx_i^\nu\right\rangle| \leq \beta ||\Delta\bx_i^\nu||^2$, which is a consequence of  \eqref{delta_i_def} and Assumption \ref{assump:homogeneity}.
\end{proof}

We are now in a position to prove the main result of this subsection.

Combining \eqref{eq:sufficient_descent_dampedNewton_fi} in Lemma \ref{lemma_descent_x_plus_II} with \eqref{Taylor_exp_Fxi_subt_simplified} in Lemma \ref{lemma:descent_cubic}, and using $\tF_i (\bx_i^{\nu};\bx_i^\nu)=F(\bx_i^\nu)$, yields
\begin{equation*}
\begin{aligned}
F (\bx_i^{\nu+}) -F (\bx_i^{\nu})
\leq & -\left(\frac{M_i}{2}-\frac{L}{6}\right) ||\Delta\bx_i^\nu||^3 - \left(\frac{\mu_i}{2}+\tau_i-\frac{\beta+\epsilon}{2}\right)||\Delta\bx_i^\nu||^2+\frac{1}{2\epsilon}\norm{\bepsilon_i^\nu}^2.
\end{aligned}
\end{equation*}
Since under either   Assumption \ref{convex-case} or Assumption \ref{sconvex-case} combined with Assumption \ref{assump:homogeneity} it holds that $\mu_i\geq \max\left\{0,\mu-\beta\right\}$, we obtain
\begin{align}
\label{Descent_Final2}
F (\bx_i^{\nu+}) -F (\bx_i^{\nu}) 
\leq & -\left(\frac{M_i}{2}-\frac{L}{6}\right) ||\Delta\bx_i^\nu||^3 - \left(\frac{\max(0,\mu-\beta)}{2}+\tau_i-\frac{\beta+\epsilon}{2}\right)||\Delta\bx_i^\nu||^2+\frac{1}{2\epsilon}\norm{\bepsilon_i^\nu}^2.
\end{align}

Denoting  $p^{\nu+}\triangleq (1/m)\sum_{i=1}^m \left\{F(\bx_i^{\nu+})-F(\widehat{\bx})\right\}$, we derive a simple relation with $p^{\nu+1}$: 
\begin{equation}\label{pnuplus_pnu}
\begin{aligned}
p^{\nu+1}+F(\widehat{\bx})&=\frac{1}{m}\sum_{i=1}^mF\Big(\bx^{\nu+1}_i\Big)\overset{\eqref{DiRegINA_pseudocode}}{=}\frac{1}{m}\sum_{i=1}^mF\Big(\sum_{j=1}^m (W_K)_{i,j}\, \bx_j^{\nu+}\Big) \\
&\overset{(a)}{\leq} \frac{1}{m}\sum_{i,j=1}^m(W_K)_{i,j}F\Big(\bx^{\nu+}_j\Big)
\overset{(b)}{=}\frac{1}{m}\sum_{j=1}^mF\Big(\bx^{\nu+}_j\Big)=p^{\nu+}+F(\widehat{\bx}),
\end{aligned}
\end{equation}
where (a) is due to convexity of $F$ (cf.~Assumptions \ref{convex-case} and \ref{sconvex-case}) and  $\sum_{j=1}^m (W_K)_{ij}=1$ (cf.~Assumption \ref{assump:weight}); and in (b) we used  $\sum_{i=1}^m (W_K)_{ij}=1$ (cf.~Assumption \ref{assump:weight}).
Summing \eqref{Descent_Final2} over $i$ while setting   $\epsilon=\beta$, $\tau_i=2\beta$ and $M_i\geq \Lhessian/3$ (recall that it is assumed   $M_i\geq \Lhessian$), gives the desired per-iteration decrease of $p^{\nu}$ when $\norm{\bepsilon^\nu}$ is sufficiently small:
\begin{equation}
\label{Descent_Final3}
\begin{aligned}
p^{\nu+1}-p^\nu\overset{\eqref{pnuplus_pnu}}{\leq} p^{\nu+}-p^\nu
\leq  & - \frac{\max(\mu,\beta)}{2}\cdot\frac{1}{m}\sum_{i=1}^m\|\Delta\bx_i^\nu\|^2+\frac{1}{2m\beta}\norm{\bepsilon^\nu}^2.
\end{aligned}
\end{equation}

\subsection{Network error bounds}\label{sec:net_err_bounds}
The goal of this subsection is to prove an upper bound for $\norm{\bepsilon^\nu}$ in terms of the number of communication steps $K$, implying that  this error   can be made sufficiently small by  choosing    sufficiently large $K$.
For notation simplicity and without loss of generality,  we assume $d=1$;   the case $d>1$ follows  trivially.  

Recall that $\bx^\nu\triangleq (\bx_i^\nu)_{i=1}^m$,  $\gradtrack^\nu\triangleq (\gradtrack_i^\nu)_{i=1}^m$, $\bJ \triangleq   (1/m) 1_m 1_m^\top$, and 
\begin{equation*}
\bx_\bot^{\nu}\triangleq (\bI-\bJ)\bx^{\nu}= \bx^\nu - 1_m \frac{1_m^\top \bx^\nu}{m},\quad \gradtrack_\bot^{\nu}\triangleq (\bI-\bJ)\gradtrack^\nu = \gradtrack^\nu - 1_m \frac{1_m^\top \gradtrack^\nu}{m} ,\quad \Delta\bx^\nu\triangleq (\Delta\bx_i^\nu)_{i=1}^m.
\end{equation*}
Note that the vectors  $\bx_\bot^{\nu}$ and $\gradtrack_\bot^{\nu}$ are the consensus and gradient-tracking errors; 
  when $\|\bx_\bot^{\nu}\|=\|\gradtrack_\bot^{\nu}\|=0$, we have  $\bx_i^\nu=\bx_j^\nu$ and $\gradtrack^\nu_i=\gradtrack^\nu_j$ for all $i,j=1,\ldots,m$. The following holds for $\bx_\bot^{\nu}$ and $\gradtrack_\bot^{\nu}$.
\begin{lemma}[Proposition 3.5 in \citet{sun2019distributed}]
	\label{cons_track_err_prop}
	Under Assumptions \ref{convex-case} and  \ref{assump:network}-\ref{assump:weight}, for all $\nu\geq 0$,
	\begin{subequations}
		\begin{align}
		\| \bx_\bot^{\nu + 1} \| & \leq \rho_K \| \bx_\bot^\nu \| +  \rho_K \| \Delta \bx^\nu\|,\label{eq:err_x_bound_0}
		\\
		\| \gradtrack_\bot^{\nu +1}\| & \leq \rho_K \| \gradtrack_\bot^\nu\| + 2\Lgrad_{\max} \rho_K \| \bx_\bot^\nu\| +   \Lgrad_{\max} \rho_K\|\Delta \bx^\nu\|,
		\label{eq:err_s_bound_0}
		\end{align}
	\end{subequations}
	where $\rho_K  = \lambda_{\max}(W_K - \bJ) < 1$.  Note that in case of $K$-rounds of communications using a reference matrix $\overline{W}$ with $\rho  \triangleq  \lambda_{\max}(\overline{W} - \bJ) < 1$, we have $\rho_K=\rho^K$; if  Chebyshev acceleration is employed, we  have $\rho_K=\left(1-\sqrt{1-\rho}\right)^K$. 
\end{lemma}

Now let us bound $\bepsilon_i^\nu$ defined in \eqref{delta_i_def}. Note that by column-stochasticity of $\bW_K$ and initialization rule $s_i^0=\nabla f_i(\bx_i^0)$, it can be trivially concluded from \eqref{grad_track_update} that 
\begin{equation*}
1_m^\top\gradtrack^\nu=\sum_{j=1}^m\nabla f_j(\bx_j^\nu).
\end{equation*}
Hence,
\begin{equation}\label{delta_i_upperbound}
\begin{aligned}
\norm{\bepsilon_i^\nu}^2= & \Big\|\gradtrack_i^\nu-\frac{1_m^\top\gradtrack^\nu}{m}+\frac{1}{m}\sum_{j=1}^m\nabla f_j(\bx_j^\nu)-\nabla F(\bx_i^\nu)\Big\|^2
\\
\overset{(a)}{\leq}&2 \Big\|\gradtrack_i^\nu-\frac{1_m^\top\gradtrack^\nu}{m}\Big\|^2+\frac{2\Lgrad_{\max}^2}{m}\left(\sum_{j=1}^m\Big\|\bx_i^\nu\pm\frac{1_m^\top\bx^\nu}{m}-\bx_j^\nu\Big\|^2\right)
\\
\leq &2 \Big\|\gradtrack_i^\nu-\frac{1_m^\top\gradtrack^\nu}{m}\Big\|^2+\frac{4\Lgrad_{\max}^2}{m}\left(\| \bx_\bot^\nu \|^2+m~\Big\|\bx_i^\nu-\frac{1_m^\top\bx^\nu}{m}\Big\|^2\right),
\end{aligned}
\end{equation}
where  (a) is due to  $\Lgrad_{\max}$-Lipschitz continuity of $\nabla f_i$.   Summing \eqref{delta_i_upperbound} over $i$ and taking the square root, gives
\begin{equation}
\label{delta_upperbnd}
\norm{\bepsilon^\nu}\leq \tilde{\delta}^\nu\triangleq \sqrt{2} \left(\|\gradtrack_\bot^{\nu} \|+2\Lgrad_{\max}\| \bx_\bot^\nu \|\right).
\end{equation}
It remains to bound $\tilde{\delta}^\nu$ defined above:
\begin{equation*}
\begin{aligned}
\tilde{\delta}^{\nu+1} = \sqrt{2} \left(\|\gradtrack_\bot^{\nu+1} \|+2\Lgrad_{\max}\| \bx_\bot^{\nu+1} \|\right)
\overset{(a)}{\leq} & 
\rho_K \sqrt{2}\left( \| \gradtrack_\bot^\nu\| + 4 \Lgrad_{\max} \| \bx_\bot^\nu\|\right) + 3\sqrt{2} \Lgrad_{\max} \rho_K\|\Delta \bx^\nu\| 
\\
\leq &
2\rho_K \tilde{\delta}^\nu+ 3\sqrt{2} \Lgrad_{\max} \rho_K\|\Delta \bx^\nu\|,
\end{aligned}
\end{equation*}
where in (a) we used   Lemma \ref{cons_track_err_prop} [cf. \eqref{eq:err_x_bound_0}-\eqref{eq:err_s_bound_0}].
Consequently,
\begin{equation}\label{delta_sqr_decay_bnd}
\begin{aligned}
(\tilde{\delta}^{\nu+1})^2 
\leq &
8\rho_K^2 (\tilde{\delta}^\nu)^2+ 36 \Lgrad_{\max}^2 \rho_K^2\|\Delta \bx^\nu\|^2.
\end{aligned}
\end{equation}
Since $\rho_K$ decreases as $K$ increases, the latter inequality provides a leverage to make $\tilde{\delta}^{\nu+1}$ sufficiently small by choosing $K$ sufficiently large. 

\subsection{Asymptotic convergence}\label{sec:asymptotic_convergence_Step3}
We combine the results of the previous two subsections to finally prove Theorem \ref{thm:asymptotic_conv}. 
Combining 
\eqref{Descent_Final3} and \eqref{delta_upperbnd}, we obtain
\begin{align}
\label{Descent_Final3_recalled}
p^{\nu+1}
\leq & \;\; p^\nu- \frac{\max(\beta,\mu)}{2m}\|\Delta\bx^\nu\|^2+\frac{1}{2m\beta}(\tilde{\delta}^\nu)^2.
\end{align}


%

Next, we combine \eqref{delta_sqr_decay_bnd} with \eqref{Descent_Final3_recalled} multiplied by some weight $w>0$ to obtain
\begin{align}
\label{synergy_net_opt_quadraticphase}
wp^{\nu+1}+(\tilde{\delta}^{\nu+1})^2 
\leq & wp^\nu+\left(8\rho_K^2 +\frac{w}{2m\beta}\right)(\tilde{\delta}^\nu)^2- w\left(\frac{\max(\beta,\mu)}{2m}-\frac{36}{w} \Lgrad_{\max}^2 \rho_K^2\right)||\Delta\bx^\nu||^2.
\end{align}

Let  $w=c_w\beta$, for   some $0<c_w\leq 1$.  Then, if
\begin{equation}\label{w_cond1_QR}
8\rho_K^2 +\frac{w}{2m\beta}\leq c_w,\quad \frac{\max(\beta,\mu)}{4m}\geq \frac{36}{w}\Lgrad_{\max}^2 \rho_K^2,
\end{equation}
  \eqref{synergy_net_opt_quadraticphase} becomes
\begin{align}
\label{synergy_net_opt_descent}
wp^{\nu+1}+(\tilde{\delta}^{\nu+1})^2 
\leq & wp^\nu+c_w(\tilde{\delta}^\nu)^2-  \frac{w\max(\beta,\mu)}{4m} ||\Delta\bx^\nu||^2.
\end{align}

Note that by Lemma \ref{cons_track_err_prop}, condition \eqref{w_cond1_QR} holds if
\begin{equation}\label{K_cond_asymptotic}
K\geq \frac{1}{\sqrt{1-\rho}}\log\left(\max\left\{\frac{2\sqrt{2}}{\sqrt{c_w(1-\frac{1}{2m})}},\frac{12\sqrt{m}\Lgrad_{\max}}{\sqrt{c_w\beta\max(\beta,\mu)}}\right\}\right).
\end{equation}
Denoting 
\begin{equation}
\xi^\nu\triangleq wp^{\nu}+(\tilde{\delta}^{\nu})^2,
\end{equation}
let us show that $\xi^\nu\rightarrow 0$ as $\nu \to \infty$, which implies  that the optimization error $p^{\nu}$ and network error $\tilde{\delta}^{\nu}$ asymptotically vanish. 
Since $\xi^\nu\geq 0$, inequality \eqref{synergy_net_opt_descent} implies $\sum_{\nu=0}^\infty||\Delta\bx^\nu||^2<\infty$. Thus, $||\Delta\bx^\nu||\rightarrow 0$; and $||\Delta\bx^\nu||\leq D_1$, for some $D_1>0$ and all $\nu\geq 0$. 
Further,  $\{\xi^\nu\}_\nu$ is non-increasing and $||\xi^\nu||\leq D_2$ for some $D_2>0$ and all $\nu\geq 0$. Thus, $p^\nu\leq D_2/w$, which together with  Assumption \ref{convex-case}(iv) and Assumption \ref{sconvex-case}, also implies   $||\bx_i^\nu||\leq D_3$ for some $D_3$, all $i$ and $\nu\geq 0$. Using $||\Delta\bx^\nu||\rightarrow 0$ and \eqref{delta_sqr_decay_bnd}, if $8\rho_K^2<1$ (which holds under \eqref{K_cond_asymptotic}), we obtain that $\tilde{\delta}^{\nu}\rightarrow 0$. Finally, it remains to show that $p^\nu\rightarrow 0$. Using optimality condition of $\hxnu{\nu}_i$ defined in \eqref{def:x_nuplus}, we get
\begin{equation*}
\left\langle \nabla F(\bx_i^\nu)+\bepsilon_i^\nu+\left[\nabla^2F(\bx_i^\nu)+\bbeta_i^\nu+\tau_i\bI\right]\Delta\bx_i^\nu+\frac{M_i}{2}||\Delta\bx_i^\nu||\Delta\bx_i^\nu, \widehat{\bx}-\hxnu{\nu}_i\right\rangle \geq 0.
\end{equation*}
Rearranging terms gives
\begin{equation}
\begin{aligned}
&\left\langle \nabla F(\bx_i^\nu)+\nabla^2F(\bx_i^\nu)\Delta\bx_i^\nu,\widehat{\bx}-\hxnu{\nu}_i\right\rangle
\\ 
\geq & \Big\langle \frac{M_i}{2}||\Delta\bx_i^\nu||\Delta\bx_i^\nu,\hxnu{\nu}_i- \widehat{\bx}\Big\rangle
+\left\langle \bepsilon_i^\nu,\hxnu{\nu}_i- \widehat{\bx}\right\rangle +\left\langle \tilde{\bbeta}_i^\nu\Delta\bx_i^\nu,\hxnu{\nu}_i- \widehat{\bx}\right\rangle,
\end{aligned}
\label{x_hat_opt_cond}
\end{equation}
where $\tilde{\bbeta}_i^\nu\triangleq \bbeta_i^\nu+\tau_i I$. By convexity of $F$, we can write
\begin{equation}
\begin{aligned}
0\geq &F(\widehat{\bx})-F(\hxnu{\nu}_i)
\\
\geq & \left\langle \nabla F(\hxnu{\nu}_i),\widehat{\bx}-\hxnu{\nu}_i\right\rangle
\\
=& \left\langle \nabla F(\hxnu{\nu}_i)-\nabla F(\bx_i^\nu)-\nabla^2F(\bx_i^\nu)\Delta\bx_i^\nu,\widehat{\bx}-\hxnu{\nu}_i\right\rangle
+\left\langle \nabla F(\bx_i^\nu)+\nabla^2F(\bx_i^\nu)\Delta\bx_i^\nu,\widehat{\bx}-\hxnu{\nu}_i\right\rangle
\\
\overset{\eqref{x_hat_opt_cond}}{\geq}& \left\langle \nabla F(\hxnu{\nu}_i)-\nabla F(\bx_i^\nu)-\nabla^2F(\bx_i^\nu)\Delta\bx_i^\nu,\widehat{\bx}-\hxnu{\nu}_i\right\rangle+\Big\langle \frac{M_i}{2}||\Delta\bx_i^\nu||\Delta\bx_i^\nu,\hxnu{\nu}_i- \widehat{\bx}\Big\rangle
\\
&
+\left\langle \bepsilon_i^\nu,\hxnu{\nu}_i- \widehat{\bx}\right\rangle +\left\langle \tilde{\bbeta}_i^\nu\Delta\bx_i^\nu,\hxnu{\nu}_i- \widehat{\bx}\right\rangle.
\end{aligned}
\label{p_to_zero_bound}
\end{equation}
Using Lipschitz continuity of $\nabla F$, $||\Delta\bx_i^\nu||\rightarrow 0$ and  $\tilde{\delta}^{\nu}\rightarrow 0$ (hence $||\delta_i^\nu||\rightarrow 0$), we conclude that the RHS of \eqref{p_to_zero_bound} asymptotically vanishes, for all $i=1,\ldots,m$. Hence, $F(\hxnu{\nu}_i)-F(\widehat{\bx})\rightarrow 0$, for all $i=1,\ldots,m$. Using \eqref{pnuplus_pnu}, we finally obtain   $p^\nu\rightarrow 0$.  

Finally, by \eqref{delta_upperbnd} and   $\tilde{\delta}^{\nu}\rightarrow 0$, we obtain   $\|\gradtrack_\bot^{\nu} \| \to 0$ and $\| \bx_\bot^\nu \| \to 0$, implying   $||\bx_i^\nu-\bx_j^\nu||\rightarrow 0$, for all $i,j=1,\ldots,m$ as $\nu \to \infty$.
This concludes  the proof of Theorem \ref{thm:asymptotic_conv}.

\begin{remark}
Note that \eqref{delta_sqr_decay_bnd} implies
\begin{equation}
\label{delta_nu_decaying_bound}
(\tilde{\delta}^{\nu})^2
\leq  \rho_K^2\bar{D}_\delta,\quad \bar{D}_\delta\triangleq 8D_2+36Q_{\max}^2D_1^2 ,\quad \forall \nu\geq 0,
\end{equation}
since $(\tilde{\delta}^{\nu})^2\leq \xi^\nu\leq D_2$ and $||\Delta\bx^\nu||\leq D_1$, for all $\nu\geq 0$.
\end{remark}

\section{Proof of Theorem \ref{thm:cvx_case}}
\label{sec:cvx_case_proof}
We first prove a detailed ``region-based'' complexity of \alg (cf. Theorem \ref{Rate_cvx_THM}, Subsec. \ref{sec:cvx_complexity_beta_leq_1})  for the prevalent scenario $0<\beta\leq 1$ [recall that typically $\beta=\mathcal{O}(1/\sqrt{n})$]. For the sake of completeness, the case $\beta\geq 1$ is studied in  Theorem \ref{Rate_cvx_THM_2} (cf.  Subsec. \ref{sec:cvx_complexity_beta_geq_1}). Building on Theorems \ref{Rate_cvx_THM}-\ref{Rate_cvx_THM_2}, we can finally   prove the main result,  Theorem \ref{thm:cvx_case} (cf.  Subsec. \ref{proof_of_thm_appendix:cvx_case}).  

\subsection{Complexity Analysis when   $0<\beta\leq 1$}\label{sec:cvx_complexity_beta_leq_1}
\begin{theorem}[$0<\beta\leq 1$ and $\Lhessian>0$]\label{Rate_cvx_THM}
Let  Assumptions \ref{convex-case} and \ref{assump:nabla2F_LC}
-\ref{assump:network} hold along with   $0<\beta\leq 1$. Let $M_i= L>0$, $\tau_i=2\beta$, and recall the definition of $D>0$ implying $||\bx_i^0-\widehat{x}||\leq D$, for all $i=1,\ldots,m$. W.l.o.g. assume $D\geq 2/\Lhessian$. Pick an accuracy $\varepsilon>0$. If a reference matrix $\overline{W}$ satisfying Assumption \ref{assump:weight} is used   in steps \eqref{DiRegINA_pseudocode}-\eqref{grad_track_update},  with  $\rho  \triangleq  \lambda_{\max}(\overline{W} - \bJ) < 1$ and $K=\widetilde{\mathcal{O}}(\log(1/\varepsilon)/\sqrt{1-\rho})$ (the explicit expression of $K$ can be found in   \eqref{K_cond_cvx}),  then the sequence $\{p^\nu\}$  generated by \alg satisfies the following: 
\begin{itemize}
	\item[(a)] if $p^\nu  \geq   2\Lhessian D^3$, 
	\begin{equation*}
	p^{\nu+1} \leq \frac{5}{6}~p^\nu,
	\end{equation*}
	\item[(b)] if $\beta^2 \cdot (2\Lhessian D^3)\leq p^\nu 
	\leq 2\Lhessian D^3$,
	\begin{equation*}
	p^\nu\leq \frac{244\cdot \Lhessian D^3}{\nu^2},
	\end{equation*}
	\item[(c)] if $\varepsilon\leq p^\nu 
	\leq \beta^2 \cdot (2\Lhessian D^3)$,   
	\begin{equation*}
	p^\nu\leq 24^2\cdot (\Lhessian D^3)^2\cdot \frac{\beta^2}{\varepsilon}\cdot \frac{1}{\nu^2}.
	\end{equation*}
\end{itemize}
\end{theorem}

\begin{proof}
Recalling Lemma \ref{lemma_descent_x_plus_II} from the proof of Theorem \ref{thm:asymptotic_conv},  we can write
\begin{equation}
\label{Taylor_exp_Fxi_subt_simplified_recalled_SCVX}
F (\bx_i^{\nu+}) \leq \tF_i (\bx_i^{\nu+};\bx_i^\nu)+\frac{1}{2\epsilon}\norm{\bepsilon_i^\nu}^2,
\end{equation}
for arbitrary $\epsilon>0$,   $M_i\geq L$, and $\tau_i\geq \beta+\epsilon$. In addition, by the upperbound approximation of  $\tF_i(\cdot;\bx_i^\nu)$ in \eqref{F_tilde_upperbound}, there holds
\begin{equation}
\label{F_tilde_upperbound_recalled_SCVX}
\begin{aligned}
\tF_i(\by;\bx_i^\nu) \leq  &  F(\by)+\frac{1}{2}\norm{\by-\bx_i^\nu}^2_{(\beta+\tau_i+\epsilon)\bI}+\frac{M_i+\Lhessian}{6}\norm{\by-\bx_i^\nu}^3+\frac{1}{2\epsilon}\norm{\bepsilon_i^\nu}^2,\quad \forall \by\in\mathcal{K}.
\end{aligned}
\end{equation}
Let $\alpha_0\in(0,1]$. Set $\epsilon=\beta$ and $\tau_i=2\beta$. By \eqref{Taylor_exp_Fxi_subt_simplified_recalled_SCVX}-\eqref{F_tilde_upperbound_recalled_SCVX} and $\hxnu{\nu}_i$ being the minimizer of $\tF(\cdot;\bx_i^\nu)$ [see \eqref{def:x_nuplus}], we obtain
\begin{equation}
\label{F_Fast_upperbound_CVX}
\begin{aligned}
& F (\bx_i^{\nu+})-F(\widehat{x})
\\
\leq & \min_{\by\in\mathcal{K}}\left\{F(\by)-F(\widehat{x})+2\beta\norm{\by-\bx_i^\nu}^2+\frac{M_i+\Lhessian}{6}\norm{\by-\bx_i^\nu}^3+\frac{1}{\beta}\norm{\bepsilon_i^\nu}^2\right\}
\\
\leq & \min_{\alpha\in[0,\alpha_0]}\Big\{F(\by)-F(\widehat{x})+2\beta\norm{\by-\bx_i^\nu}^2+\frac{M_i+\Lhessian}{6}\norm{\by-\bx_i^\nu}^3+\frac{1}{\beta}\norm{\bepsilon_i^\nu}^2
\\
&\qquad\qquad\qquad\qquad\qquad\qquad\qquad :y=\alpha\widehat{x}+(1-\alpha)\bx_i^\nu\Big\}
\\
\leq & \min_{\alpha\in[0,\alpha_0]}\Big\{(1-\alpha)\left(F(\bx_i^\nu)-F(\widehat{x})\right)
\\
&\qquad \qquad +2\beta\alpha^2\norm{\widehat{x}-\bx_i^\nu}^2+\frac{M_i+\Lhessian}{6}\alpha^3\norm{\widehat{x}-\bx_i^\nu}^3+\frac{1}{\beta}\norm{\bepsilon_i^\nu}^2 \Big\},
\end{aligned}
\end{equation}
where the last inequality holds by the convexity of $F$. 
Note that, by definition,   $||\bx_i^0-\widehat{x}||\leq D,$ for all $i=1,\ldots, m$. Assuming $||\bx_i^\nu-\widehat{x}||\leq D$, for all $i=1,\ldots, m$, we prove descent at iteration $\nu+1$, i.e. $p^{\nu+1}<p^\nu$, unless $p^\nu=0$. Note that by  Assumption \ref{convex-case}(iv), if $\{p^\nu\}_\nu$ is non-increasing, then  $||\bx_i^\nu-\widehat{x}||\leq D$ for all $\nu\geq 0$ and $i=1,\ldots,m$. Now set $M_i=L$ in  \eqref{F_Fast_upperbound_CVX} and compute the mean over $i=1,\ldots,m$, which yields
\begin{equation}
\label{F_Fast_upperbound_CVX_avgd}
 p^{\nu+1}\overset{\eqref{pnuplus_pnu}}{\leq} p^{\nu+}\leq 
  \min_{\alpha\in[0,\alpha_0]}\Big\{(1-\alpha)p^\nu +2\beta\alpha^2D^2+\frac{\Lhessian D^3}{3}\alpha^3+\frac{1}{m\beta}\norm{\bepsilon^\nu}^2 \Big\}.
\end{equation}

Denote
\begin{equation}\label{def:C_1}
 C_1\triangleq \frac{\Lhessian D^3}{3}.
\end{equation}
Since $D\geq \frac{2}{\Lhessian}$, it holds  $2\beta D^2\leq 3\beta C_1$. Then, setting $\alpha_0=\min\{1,p^\nu/(6\beta C_1)\}$ in \eqref{F_Fast_upperbound_CVX_avgd} yields
\begin{equation}
\label{F_Fast_upperbound_CVX_avgd2}
\begin{aligned}
p^{\nu+1}
\leq & \min_{\alpha\in[0,\min\{1,\frac{p^\nu}{6\beta C_1}\}]}\Big\{(1-\alpha)p^\nu +3\beta C_1\alpha^2+C_1\alpha^3+\frac{1}{m\beta}\norm{\bepsilon^\nu}^2 \Big\}
\\
\leq & \min_{\alpha\in[0,\min\{1,\frac{p^\nu}{6\beta C_1}\}]}\Big\{(1-\alpha/2)p^\nu +C_1\alpha^3+\frac{1}{m\beta}\norm{\bepsilon^\nu}^2 \Big\}.
\end{aligned}
\end{equation}

Let us assess \eqref{F_Fast_upperbound_CVX_avgd2} over the following ``regions". Denoting by $\alpha^\ast$ the minimizer of the optimization problem at the RHS of \eqref{F_Fast_upperbound_CVX_avgd2}, we have  the following: 

\textbf{(a)} If $p^\nu\geq 6C_1$, then $\alpha^\ast=1$ and 
\begin{equation}
\label{recursion_cvx_stage1}
p^{\nu+1}\leq \frac{1}{2}p^\nu+C_1+\frac{1}{m\beta}\norm{\bepsilon^\nu}^2 \leq \left(\frac{1}{2}+\frac{1}{6}\right)p^\nu+\frac{1}{m\beta}\norm{\bepsilon^\nu}^2 ,
\end{equation}
and under the condition 
\begin{equation}\label{delta_cond1_cvx}
\frac{1}{m\beta}\norm{\bepsilon^\nu}^2\leq \frac{1}{6}p^\nu\impliedby  \frac{1}{m\beta}\norm{\bepsilon^\nu}^2\leq C_1,
\end{equation}
 \eqref{recursion_cvx_stage1} yields
\begin{equation*}
p^{\nu+1} \leq \frac{5}{6}~p^\nu.
\end{equation*}
Note that, by \eqref{delta_nu_decaying_bound} and Lemma \ref{cons_track_err_prop}, condition \eqref{delta_cond1_cvx}  holds if 
\begin{equation}\label{K_cond_cvx_1}
K\geq \frac{1}{\sqrt{1-\rho}}\cdot\frac{1}{2} \log\left(\frac{\bar{D}_\delta}{ m\beta C_1}\right).
\end{equation}

\textbf{(b)} If $6\beta^2 C_1\leq p^\nu\leq 6C_1$, then $\alpha^\ast=\sqrt{\frac{p^\nu}{6C_1}}$ and 
\begin{equation}\label{recursion_cvx_stage2}
p^{\nu+1}\leq p^\nu-\frac{(p^\nu)^{3/2}}{3\sqrt{6C_1}}+\frac{1}{m\beta}\norm{\bepsilon^\nu}^2,
\end{equation}
and if (similar to derivation of \eqref{K_cond_cvx_1})
\begin{equation}\label{delta_cond2_cvx}
K\geq \frac{1}{\sqrt{1-\rho}}\cdot\frac{1}{2} \log\left(\frac{\bar{D}_\delta}{m\beta^4 C_1}\right)\implies \frac{1}{m\beta}\norm{\bepsilon^\nu}^2 \leq \beta^3C_1\implies \frac{1}{m\beta}\norm{\bepsilon^\nu}^2 \leq \frac{(p^\nu)^{3/2}}{6\sqrt{6C_1}},
\end{equation}
 \eqref{recursion_cvx_stage2} implies
\begin{equation}
\label{recursion_cvx_stage3}
p^{\nu+1}\leq p^\nu-\frac{(p^\nu)^{3/2}}{6\sqrt{6C_1}}.
\end{equation}
Finally, since $p^\nu$ is non-increasing, 
\begin{equation*}
\begin{aligned}
\frac{1}{\sqrt{p^{\nu+1}}}-\frac{1}{\sqrt{p^{\nu}}}=&\frac{p^\nu-p^{\nu+1}}{\left(\sqrt{p^\nu}+\sqrt{p^{\nu+1}}\right)\sqrt{p^\nu p^{\nu+1}}}\overset{\eqref{recursion_cvx_stage3}}{\geq} \frac{\frac{1}{6\sqrt{6C_1}}(p^\nu)^{3/2}}{\left(\sqrt{p^\nu}+\sqrt{p^{\nu+1}}\right)\sqrt{p^\nu p^{\nu+1}}}
\\
\geq &c_0\triangleq \frac{1}{12}\sqrt{\frac{1}{6C_1}},
\end{aligned}
\end{equation*}
and consequently,
\begin{equation*}
p^\nu\leq \frac{1}{c_0^2\Big(\nu+\frac{1}{c_0\sqrt{p^{0}}}\Big)^2}\leq \frac{1}{c_0^2\nu^2}.
\end{equation*}

\textbf{(c)} If $ \varepsilon\leq p^\nu\leq 6\beta^2C_1$, then $\alpha^\ast= \frac{p^\nu}{6\beta C_1}$ and 
\begin{equation}
\label{cvx_last_stage_bound}
p^{\nu+1}\leq p^\nu-\frac{(p^\nu)^{2}}{18\beta C_1}+\frac{1}{m\beta}\norm{\bepsilon^\nu}^2,
\end{equation}
and if (similar to derivation of \eqref{K_cond_cvx_1})
\begin{equation}\label{delta_cond3_cvx}
K\geq \frac{1}{\sqrt{1-\rho}}\cdot\frac{1}{2} \log\left(\frac{36C_1\bar{D}_\delta}{m\varepsilon^2 }\right)\implies \frac{1}{m\beta}\norm{\bepsilon^\nu}^2 \leq \frac{\varepsilon^{2}}{36\beta C_1}\implies \frac{1}{m\beta}\norm{\bepsilon^\nu}^2 \leq \frac{(p^\nu)^{2}}{36\beta C_1},
\end{equation}
 we deduce from \eqref{cvx_last_stage_bound}  
\begin{equation}
\label{cvx_last_stage_bound2}
p^{\nu+1}\leq p^\nu-\frac{(p^\nu)^{2}}{36\beta C_1}.
\end{equation}
Since $p^\nu$ is non-increasing, 
\begin{equation}
\label{Rate_stage1_part3}
\begin{aligned}
\frac{1}{\sqrt{p^{\nu+1}}}-\frac{1}{\sqrt{p^{\nu}}}=&\frac{p^\nu-p^{\nu+1}}{\left(\sqrt{p^\nu}+\sqrt{p^{\nu+1}}\right)\sqrt{p^\nu p^{\nu+1}}}\overset{\eqref{cvx_last_stage_bound2}}{\geq} \frac{\frac{1}{36\beta C_1}(p^\nu)^{2}}{\left(\sqrt{p^\nu}+\sqrt{p^{\nu+1}}\right)\sqrt{p^\nu p^{\nu+1}}}
\\
\geq &\tilde{c}_0\triangleq \frac{\sqrt{\varepsilon}}{72\beta C_1},
\end{aligned}
\end{equation}
%
and consequently,
\begin{equation}\label{last_phase_rate_cvx}
p^\nu\leq \frac{1}{\tilde{c}_0^2\Big(\nu+\frac{1}{\tilde{c}_0\sqrt{p^{0}}}\Big)^2}\leq \frac{1}{\tilde{c}_0^2\nu^2}=72^2\cdot C_1^2\cdot \frac{\beta^2}{\varepsilon}\cdot \frac{1}{\nu^2}.
\end{equation}
%

Finally, combining all the  conditions \eqref{K_cond_asymptotic}, \eqref{K_cond_cvx_1},\eqref{delta_cond2_cvx}, and \eqref{delta_cond3_cvx}, the   requirement on    $K$ reads 
\begin{equation}\label{K_cond_cvx}
K\geq \frac{1}{\sqrt{1-\rho}}\cdot\frac{1}{2} \log\left(\max\left\{\frac{16}{c_w},\frac{12^2m\Lgrad_{\max}^2}{c_w\beta\max(\beta,\mu)}, \frac{\bar{D}_\delta}{\min\left\{m\beta C_1,m\beta^4 C_1,\frac{m}{36C_1}\varepsilon^2\right\}}\right\}\right),
\end{equation}
where  $\bar{D}_\delta$ and $C_1$ are  defined in \eqref{delta_nu_decaying_bound} and  \eqref{def:C_1}, respectively.
\end{proof}

\subsection{Complexity Analysis when $\beta\geq 1$}\label{sec:cvx_complexity_beta_geq_1}
\begin{theorem}[$\beta\geq 1$ and $\Lhessian>0$]\label{Rate_cvx_THM_2}
	{ 
		Let  Assumptions \ref{convex-case} and \ref{assump:nabla2F_LC}
	-\ref{assump:network} hold and  $\beta\geq 1$. Let $M_i= L>0$, $\tau_i=2\beta$, and recall the definition of $D>0$ implying $\max_{i\in[m]}||\bx_i^0-\widehat{x}||\leq D$. W.l.o.g. assume $D\geq 2/\Lhessian$. Pick an arbitrary $\varepsilon>0$. If a reference matrix $\overline{W}$ satisfying Assumption \ref{assump:weight} is used   in steps \eqref{DiRegINA_pseudocode}-\eqref{grad_track_update},  with $\rho  \triangleq  \lambda_{\max}(\overline{W} - \bJ) < 1$   and $K=\widetilde{\mathcal{O}}(\log(1/\varepsilon)/\sqrt{1-\rho})$ (the explicit expression is given in   \eqref{K_cond_cvx}), then the sequence $\{p^\nu\}$  generated by \alg satisfies the following:
}
	\begin{itemize}
		\item[(a)] if $p^\nu  \geq  \beta\cdot (2\Lhessian D^3)$, 
		\begin{equation*}
		p^{\nu+1} \leq \frac{5}{6}~p^\nu,
		\end{equation*}
		\item[(b)] if $\varepsilon\leq p^\nu 
		\leq \beta \cdot (2\Lhessian D^3)$,   
		\begin{equation*}
		p^\nu\leq 24^2\cdot (\Lhessian D^3)^2\cdot \frac{\beta^2}{\varepsilon}\cdot \frac{1}{\nu^2}.
		\end{equation*}
	\end{itemize}
\end{theorem}
\begin{proof}
Excluding $\beta$, the parameter setting is identical to Theorem \ref{Rate_cvx_THM}. Recall \eqref{F_Fast_upperbound_CVX_avgd2}, i.e.,
\begin{equation}
\label{F_Fast_upperbound_CVX_avgd2_recalled}
\begin{aligned}
p^{\nu+1}
\leq & \min_{\alpha\in[0,\min\{1,\frac{p^\nu}{6\beta C_1}\}]}\Big\{(1-\alpha/2)p^\nu +C_1\alpha^3+\frac{1}{m\beta}\norm{\bepsilon^\nu}^2 \Big\},
\end{aligned}
\end{equation}
where $C_1$ is defined in \eqref{def:C_1}. Denoting by $\alpha^\ast$ the minimizer of the optimization problem at the  RHS of \eqref{F_Fast_upperbound_CVX_avgd2}, we have:

\textbf{(a)} If $p^\nu\geq 6\beta C_1$, then $\alpha^\ast=1$ and 
under \eqref{K_cond_cvx},
 \eqref{F_Fast_upperbound_CVX_avgd2_recalled} yields
\begin{equation*}
p^{\nu+1}\leq \frac{4+1/\beta}{6}p^\nu\leq \frac{5}{6}~p^\nu.
\end{equation*}

\textbf{(b)} If $\varepsilon\leq p^\nu\leq  6\beta C_1$, then $\alpha^\ast=\frac{p^\nu}{6\beta C_1}$. Under \eqref{K_cond_cvx},  \eqref{F_Fast_upperbound_CVX_avgd2_recalled} yields
\begin{equation*}
\begin{aligned}
p^{\nu+1}
\leq & p^\nu-\frac{(p^\nu)^2}{36 \beta C_1},
\end{aligned}
\end{equation*}
and following similar steps as in  derivation of \eqref{last_phase_rate_cvx}, we obtain
\begin{equation*}
p^\nu\leq \frac{1}{\tilde{c}_0^2\nu^2}=72^2\cdot C_1^2\cdot \frac{\beta^2}{\varepsilon}\cdot \frac{1}{\nu^2}.
\end{equation*}
\end{proof}

\subsection{Proof of main theorem}\label{proof_of_thm_appendix:cvx_case}
We proceed to prove Theorem \ref{thm:cvx_case}. Given an  accuracy $0<\varepsilon\ll 1$, when $0<\beta\leq 1$, Theorem \ref{Rate_cvx_THM} gives the following expression of rate: to achieve $p^\nu\leq \varepsilon$, \alg requires
\begin{equation}
\label{CVX_O_expression_1}
O\left(\log\left(\frac{1}{6C_1}\right)+\sqrt{\frac{\Lhessian D^3}{\varepsilon}}+\frac{\beta\left(\Lhessian D^3\right)}{\varepsilon}\right)=\tO{\sqrt{\frac{\Lhessian D^3}{\varepsilon}}+\frac{\beta\left(\Lhessian D^3\right)}{\varepsilon}},
\end{equation}
iterations, while if $\beta\geq 1$, by Theorem \ref{Rate_cvx_THM_2}, \alg requires
\begin{equation*}
O\left(\log\left(\frac{1}{2\beta\Lhessian D^3}\right)+\frac{\beta\left(\Lhessian D^3\right)}{\varepsilon}\right)=\tO{\frac{\beta\left(\Lhessian D^3\right)}{\varepsilon}}
\end{equation*}
iterations. Therefore, \eqref{CVX_O_expression_1} is a valid rate complexity expression (in terms of iterations) in both discussed cases (i.e. $0<\beta\leq 1$ and $\beta\geq 1$). Now, recall that every iteration requires $K$ rounds of communications, with $K$   satisfying \eqref{K_cond_asymptotic} and \eqref{K_cond_cvx}; hence $K=\tO{1/\sqrt{1-\rho}\cdot \log(1/\varepsilon)}=\tO{1/\sqrt{1-\rho}\cdot \varepsilon^{-\alpha/2}}$,   for any arbitrary small $\alpha>0$. Therefore the final communication complexity reads
\begin{equation*}
\tO{\frac{1}{\sqrt{1-\rho}}\cdot \left\{ \sqrt{\frac{\Lhessian D^3}{\varepsilon^{1+\alpha}}}+\frac{\beta\left(\Lhessian D^3\right)}{\varepsilon^{1+\frac{\alpha}{2}}}\right\}}.
\end{equation*}

\section{Proof of Theorem \ref{thm:scvx_case_beta_leq_mu} and Corollary  \ref{corr:scvx_case_beta_leq_mu_initialization_iid} }
\label{sec:scvx_case_proof}
We begin introducing  some intermediate technical results,   instrumental to proving the main theorems, namely: i) Lemmata \ref{Deltax_lowerbound_lemm}-\ref{deltax_upperbnd} in Sec.~\ref{sec:Prelim_results_thm_scvx_case}; and ii)   a detailed ``region-based" complexity of \alg as in  in Theorem \ref{Rate_scvx_THM} (cf. Sec. \ref{sec:Rate_scvx_THM_proof}). 
We prove Theorem \ref{thm:scvx_case_beta_leq_mu} and the improved rates in case of quadratic functions  in Sec. \ref{subsec:scvx_beta_leq_mu} and  Sec. \ref{sec:scvx_case_beta_leq_mu_Qcase_proof}, respectively. Finally, Corollary  \ref{corr:scvx_case_beta_leq_mu_initialization_iid}  is proved in Sec. \ref{sec:scvx_case_beta_leq_mu_initialization_proof}.


\subsection{Preliminary results}\label{sec:Prelim_results_thm_scvx_case}
We establish necessary connections between the optimization error $p^\nu$, the network error $\norm{\bepsilon^\nu}$ and $||\Delta \bx^\nu||$ in Lemmata \ref{deltax_upperbnd}-\ref{Deltax_lowerbound_lemm}:
\begin{lemma}
	\label{deltax_upperbnd}
	Let Assumptions \ref{sconvex-case}-\ref{assump:homogeneity} hold,  $\tau_i=2\beta$, and $M_i\geq \Lhessian/3$. Then  
	\begin{equation}
	\frac{1}{m}\sum_{i=1}^m\norm{\Delta\bx_i^\nu}^2\leq \frac{8}{\mu}~p^\nu+\frac{2}{m\beta\mu}\norm{\bepsilon^\nu}^2,
	\label{deltax_upperbnd_ineq}
	\end{equation}
	where $p^\nu$ is defined in \eqref{p_nu}.
\end{lemma}

\begin{proof}
	By $\mu$-strongly convexity of $F$ and optimality of $\widehat{x}$,
	\begin{equation*}
	\begin{aligned}
	F(\hxnu{\nu}_i)-F(\widehat{x})\geq & \frac{\mu}{2}\norm{\hxnu{\nu}_i-\widehat{x}}^2
	\geq  \frac{\mu}{4}\norm{\hxnu{\nu}_i-\bx_i^\nu}^2-\frac{\mu}{2}\norm{\bx_i^\nu-\widehat{x}}^2
	\\
	\geq &\frac{\mu}{4}\norm{\hxnu{\nu}_i-\bx_i^\nu}^2-\left(F(\bx_i^\nu)-F(\widehat{x})\right).
	\end{aligned}
	\end{equation*}
Averaging the above inequalities over $i=1,\ldots,m$,  yields
	\begin{equation*}
	\frac{1}{m}\sum_{i=1}^m\norm{\Delta\bx_i^\nu}^2\leq \frac{4}{\mu}\left(p^{\nu+}+p^\nu\right),
	\end{equation*}
	where $p^{\nu+}= (1/m)\sum_{i=1}^m \left\{F(\bx_i^{\nu+})-F(\widehat{x})\right\}$. Using  \eqref{Descent_Final3} proves \eqref{deltax_upperbnd_ineq}.
\end{proof}

\begin{lemma}\label{Deltax_lowerbound_lemm}
	Let Assumptions \ref{sconvex-case}-\ref{assump:homogeneity} hold and set $\tau_i=2\beta$. Define 
	\begin{equation*}
	\omega_0\triangleq 
	 \frac{12\beta}{\sqrt{L^2+4M_{\max}^2}},\quad M_{\max}\triangleq \max_{i\in[m]}~M_i.
	\end{equation*}
	Then
	\begin{equation}
	\begin{aligned}
	\frac{1}{m}\sum_{i=1}^m \left\{F(\bx_i^{\nu +})-F(\widehat{x})\right\}
	\leq 
	& \varphi \left(\{\hxnu{\nu}_i\}_i,\{\bx_i^\nu\}_i\right) +\frac{8}{m\mu}\norm{\bepsilon^\nu}^2,
	\end{aligned}
	\label{Deltax_lowerbound_xhat}
	\end{equation}
	where
	\begin{equation*}
	\varphi\left(\{\hxnu{\nu}_i\}_i,\{\bx_i^\nu\}_i\right) = \left\{
	\begin{array}{l l}
	\frac{\Lhessian^2+4M_{\max}^2}{m\mu}\left(\sum_{i=1}^m\left\|\hxnu{\nu}_i-\bx_i^\nu\right\|^2\right)^2, & \quad \text{if\quad \hypertarget{pDeltaX_cond}{$\texttt{C}$}: } \sqrt{\sum_{i=1}^m\norm{\hxnu{\nu}_i-\bx_i^\nu}^2} \geq \omega_0;
	\\
	\frac{144\beta^2}{m\mu}\sum_{i=1}^m\norm{\hxnu{\nu}_i-\bx_i^\nu}^2, & \quad \text{if\quad \hypertarget{pDeltaX_cond_bar}{$\overline{\texttt{C}}$}: } \sqrt{\sum_{i=1}^m\norm{\hxnu{\nu}_i-\bx_i^\nu}^2} < \omega_0.
	\end{array} \right.
	\end{equation*}
	
\end{lemma}

\begin{proof}
	Recall \eqref{x_hat_opt_cond}, a consequence of  optimality of $\hxnu{\nu}_i$ (defined in \eqref{def:x_nuplus}), reads
	\begin{equation}
	\begin{aligned}
	&\left\langle \nabla F(\bx_i^\nu)+\nabla^2F(\bx_i^\nu)\Delta\bx_i^\nu,\widehat{x}-\hxnu{\nu}_i\right\rangle
	\\ 
	\geq & \Big\langle \frac{M_i}{2}||\Delta\bx_i^\nu||\Delta\bx_i^\nu,\hxnu{\nu}_i- \widehat{x}\Big\rangle
	+\left\langle \bepsilon_i^\nu,\hxnu{\nu}_i- \widehat{x}\right\rangle +\left\langle \tilde{\bbeta}_i^\nu\Delta\bx_i^\nu,\hxnu{\nu}_i- \widehat{x}\right\rangle,
	\end{aligned}
	\label{x_hat_opt_cond_recalled}
	\end{equation}
	where  $\tilde{\bbeta}_i^\nu= \bbeta_i^\nu+\tau_i I$ and recall $\Delta\bx_i^\nu=\hxnu{\nu}_i-\bx_i^\nu$ [cf. \eqref{def:p_deltaX}]. By $\mu$-strongly convexity of $F$, 
	\begin{equation}
	\begin{aligned}
	&F(\widehat{x})-F(\hxnu{\nu}_i)
	\\
	\geq & \left\langle \nabla F(\hxnu{\nu}_i),\widehat{x}-\hxnu{\nu}_i\right\rangle+\frac{\mu}{2}\norm{\hxnu{\nu}_i-\widehat{x}}^2
	\\
	=& \left\langle \nabla F(\hxnu{\nu}_i)-\nabla F(\bx_i^\nu)-\nabla^2F(\bx_i^\nu)\Delta\bx_i^\nu,\widehat{x}-\hxnu{\nu}_i\right\rangle+\frac{\mu}{2}\norm{\hxnu{\nu}_i-\widehat{x}}^2
	\\
	&+\left\langle \nabla F(\bx_i^\nu)+\nabla^2F(\bx_i^\nu)\Delta\bx_i^\nu,\widehat{x}-\hxnu{\nu}_i\right\rangle
	\\
	\overset{\eqref{x_hat_opt_cond_recalled}}{\geq}& \left\langle \nabla F(\hxnu{\nu}_i)-\nabla F(\bx_i^\nu)-\nabla^2F(\bx_i^\nu)\Delta\bx_i^\nu,\widehat{x}-\hxnu{\nu}_i\right\rangle+\frac{\mu}{2}\norm{\hxnu{\nu}_i-\widehat{x}}^2
	\\
	&+\left\langle \frac{M_i}{2}||\Delta\bx_i^\nu||\Delta\bx_i^\nu,\hxnu{\nu}_i- \widehat{x}\right\rangle
	+\left\langle \bepsilon_i^\nu,\hxnu{\nu}_i- \widehat{x}\right\rangle +\left\langle \tilde{\bbeta}_i^\nu\Delta\bx_i^\nu,\hxnu{\nu}_i- \widehat{x}\right\rangle
	\\
	\geq & -\frac{1}{2\mu}\left\|\nabla F(\hxnu{\nu}_i)-\nabla F(\bx_i^\nu)-\nabla^2F(\bx_i^\nu)\Delta\bx_i^\nu\right\|^2
	\\
	&+\left\langle \frac{M_i}{2}||\Delta\bx_i^\nu||\Delta\bx_i^\nu,\hxnu{\nu}_i- \widehat{x}\right\rangle
	+\left\langle \bepsilon_i^\nu,\hxnu{\nu}_i- \widehat{x}\right\rangle +\left\langle \tilde{\bbeta}_i^\nu\Delta\bx_i^\nu,\hxnu{\nu}_i- \widehat{x}\right\rangle,
	\end{aligned}
	\label{ineq_intermediate_1}
	\end{equation}
	and by applying Lemma \ref{res_lemma} (cf. inequality \eqref{res1}) to the first term on the RHS of \eqref{ineq_intermediate_1} along with Cauchy-schwarz inequality, yield
	\begin{equation}
	\begin{aligned}
	& F(\widehat{x})-F(\hxnu{\nu}_i)
	\\
	\geq &-\left(\frac{\Lhessian^2}{8\mu}+\frac{M_i}{4\epsilon_0}\right)\left\|\Delta\bx_i^\nu\right\|^4-\frac{M_i\epsilon_0}{4}\norm{\hxnu{\nu}_i- \widehat{x}}^2
	-\frac{1}{2\epsilon_1}\norm{\bepsilon_i^\nu}^2-\frac{\epsilon_1}{2}\norm{\hxnu{\nu}_i- \widehat{x}}^2 +\left\langle \tilde{\bbeta}_i^\nu\Delta\bx_i^\nu,\hxnu{\nu}_i- \widehat{x}\right\rangle
	\\
		\overset{(a)}{\geq}  
	&-\left(\frac{\Lhessian^2}{8\mu}+\frac{M_i}{4\epsilon_0}\right)\left\|\Delta\bx_i^\nu\right\|^4
	-\left(\frac{M_i\epsilon_0}{2\mu}+\frac{\epsilon_1}{\mu}\right)\left(F(\hxnu{\nu}_i)-F(\widehat{x})\right)
	-\frac{1}{2\epsilon_1}\norm{\bepsilon_i^\nu}^2+\left\langle \tilde{\bbeta}_i^\nu\Delta\bx_i^\nu,\hxnu{\nu}_i- \widehat{x}\right\rangle,
	\end{aligned}
	\label{ineq_intermediate_2}
	\end{equation}
	for  arbitrary $\epsilon_0,\epsilon_1>0$, where (a) is due to the  $\mu$-strongly convexity of $F$ and optimality of $\widehat{x}$. By Assumption \ref{assump:homogeneity} and some algebraic manipulations, the last term on the RHS of \eqref{ineq_intermediate_2} is lower-bounded as 
	\begin{equation}
	\begin{aligned}
	\left\langle \Delta\bx_i^\nu,\hxnu{\nu}_i- \widehat{x}\right\rangle_{\tilde{\bbeta}_i^\nu}
	\geq & -\frac{\beta+\tau_i}{2\epsilon_2}\norm{\Delta\bx_i^\nu}^2-\frac{\epsilon_2(\beta+\tau_i)}{2}\norm{\hxnu{\nu}_i- \widehat{x}}^2
	\\
	\overset{(a)}{\geq} & -\frac{\beta+\tau_i}{2\epsilon_2}\norm{\Delta\bx_i^\nu}^2-\frac{\epsilon_2(\beta+\tau_i)}{\mu}\left(F(\hxnu{\nu}_i)-F(\widehat{x})\right),
	\end{aligned}
	\label{innerprod_bound}
	\end{equation}
	with arbitrary $\epsilon_2>0$, where (a) follows from the $\mu$-strong convexity of $F$ and optimality of $\widehat{x}$.
Set 
	\begin{equation*}
	\epsilon_0=\frac{\mu}{2M_{\max}},\quad  \epsilon_1=\frac{\mu}{4},\quad  \epsilon_2=\frac{\mu}{4(\beta+\tau_{\max})},
	\end{equation*}
where $\tau_{\max}\triangleq \max_{i\in[m]}~\tau_i$; then combining  \eqref{ineq_intermediate_2}-\eqref{innerprod_bound} and averaging  over $i=1,\ldots,m$, lead to
	\begin{equation}
	\begin{aligned}
	&\frac{1}{m}\sum_{i=1}^m\left(F(\hxnu{\nu}_i)-F(\widehat{x})\right)
	\leq 
	&\frac{\Lhessian^2+4M_{\max}^2}{2m\mu}\sum_{i=1}^m\left\|\Delta\bx_i^\nu\right\|^4+\frac{8\left(\beta+\tau_{\max}\right)^2}{m\mu}\sum_{i=1}^m\norm{\Delta\bx_i^\nu}^2+\frac{8}{m\mu}\norm{\bepsilon^\nu}^2.
	\end{aligned}
	\label{ineq_intermediate_4}
	\end{equation}
	The bound \eqref{Deltax_lowerbound_xhat} is a direct consequence of \eqref{ineq_intermediate_4}, with  $\tau_i=2\beta$, for all $i=1,\ldots,m$.
\end{proof}

\subsection{Preliminary complexity results}\label{sec:Rate_scvx_THM_proof}
\begin{theorem}\label{Rate_scvx_THM}
	Let  Assumptions \ref{sconvex-case}-\ref{assump:network} hold. Let also $M_i\geq L$ and $\tau_i=2\beta$, for all $i=1,\ldots,m$, and denote
	\begin{equation*}
	\begin{aligned}
	 C_2\triangleq \xi\cdot  \frac{(M_{\max}+\Lhessian)\sqrt{2m}}{3\mu^{3/2}},\quad M_{\max}\triangleq \max_{i\in[m]}~M_i,
	\end{aligned}
	\end{equation*}
	{ for some arbitrary $\xi\geq 1$. If a reference matrix $\overline{W}$ satisfying Assumption \ref{assump:weight} is used   in steps \eqref{DiRegINA_pseudocode}-\eqref{grad_track_update},  with $\rho  \triangleq  \lambda_{\max}(\overline{W} - \bJ) < 1$ and    $K=\widetilde{\mathcal{O}}(1/\sqrt{1-\rho})$ (the explicit expression of $K$ is given in  \eqref{K_cond_scvx}), then the sequence $\{p^\nu\}$ generated by \alg satisfies the following:}
	\begin{itemize}
		\item[(a)] 	If 
		\begin{equation*}
		p^\nu\geq \underline{p}_1\triangleq \frac{\mu^3}{2m(M_{\max}+\Lhessian)^2\xi^2}\left(1+\frac{4\beta}{\mu}\right)^4,
		\end{equation*}
		then  
		\begin{equation*}
		(p^\nu)^{1/4}\leq (p^0)^{1/4}-\frac{\nu}{12\sqrt{3C_2}}.
		\end{equation*}
		
		\item[(b)] Assume [exclusively in this case (b)] $\beta\leq \mu$ and denote 
		\begin{equation*}
		\tilde{p}^\nu\triangleq p^\nu/c^2,\quad c\triangleq\frac{\mu\sqrt{\mu}}{8\sqrt{m(\Lhessian^2+4M_{\max}^2)}},\quad  \underline{p}_2\triangleq \frac{2\cdot 12^4}{\Lhessian^2+4M_{\max}^2}\cdot \frac{\beta^2\mu}{m}.
		\end{equation*}
		If $p^{\nu}\geq \underline{p}_2$ and $p^{\nu-1}\leq c^2$, then $\tilde{p}^{\nu}\leq  (\tilde{p}^{\nu-1})^{2}$.
		\item[(c)] If
		\begin{equation}\label{Last_phase_cond_SCVX}
		p^\nu<  \underline{p}_3\triangleq  \frac{9}{L^2+4M_{\max}^2}\cdot \frac{\beta^2\mu}{m},
		\end{equation}
then $\{p^\nu\}$ converges $Q$-linearly to zero with   rate

\begin{equation}
\label{decay_rate_thm}
\left(1+\frac{\max(\beta,\mu)}{4mb_2}\right)^{-1}=\left(1+\frac{1}{576}\cdot \frac{\mu\max(\beta,\mu)}{\beta^2}\right)^{-1}.
\end{equation} 
	\end{itemize}
\end{theorem}

\begin{proof}
We organize the proof into three parts, \textbf{(a)-(c)}, in accordance with the  three cases in  the statement of the theorem.  

\textbf{(a)} 
Recall Lemma \ref{lemma_descent_x_plus_II} from the proof of Theorem \ref{thm:asymptotic_conv}: 
\begin{equation}
\label{Taylor_exp_Fxi_subt_simplified_recalled_SCVX2}
F (\bx_i^{\nu+}) \leq \tF_i (\bx_i^{\nu+};\bx_i^\nu)+\frac{1}{2\epsilon}\norm{\bepsilon_i^\nu}^2,
\end{equation}
for arbitrary $\epsilon>0$, where $M_i\geq L$ and $\tau_i\geq \beta+\epsilon$. In addition, by the upperbound approximation of  $\tF_i(\cdot;\bx_i^\nu)$ in \eqref{F_tilde_upperbound}, there holds
\begin{equation}
\label{F_tilde_upperbound_recalled_SCVX2}
\begin{aligned}
\tF_i(\by;\bx_i^\nu) \leq  &  F(\by)+\frac{1}{2}\norm{\by-\bx_i^\nu}^2_{(\beta+\tau_i+\epsilon)\bI}+\frac{M_i+\Lhessian}{6}\norm{\by-\bx_i^\nu}^3+\frac{1}{2\epsilon}\norm{\bepsilon_i^\nu}^2,\quad \forall \by\in\mathcal{K}.
\end{aligned}
\end{equation}
Set $\tau_i=2\beta$ and $\epsilon=\beta$, then by  \eqref{Taylor_exp_Fxi_subt_simplified_recalled_SCVX2}-\eqref{F_tilde_upperbound_recalled_SCVX2} and $\hxnu{\nu}_i$ being the minimizer of $\tF(\cdot;\bx_i^\nu)$,  
\begin{equation}
\label{F_Fast_upperbound_SCVX}
\begin{aligned}
& F (\bx_i^{\nu+})-F(\widehat{x})
\\
\leq & \min_{\by\in\mathcal{K}}\left\{F(\by)-F(\widehat{x})+2\beta\norm{\by-\bx_i^\nu}^2+\frac{M_i+\Lhessian}{6}\norm{\by-\bx_i^\nu}^3+\frac{1}{\beta}\norm{\bepsilon_i^\nu}^2\right\}
\\
\leq & \min_{\alpha\in[0,\alpha_0]}\Big\{F(\by)-F(\widehat{x})+2\beta\norm{\by-\bx_i^\nu}^2+\frac{M_i+\Lhessian}{6}\norm{\by-\bx_i^\nu}^3+\frac{1}{\beta}\norm{\bepsilon_i^\nu}^2
:y=\alpha\widehat{x}+(1-\alpha)\bx_i^\nu\Big\}
\\
\overset{(a)}{\leq} & \min_{\alpha\in[0,\alpha_0]}\Big\{(1-\alpha)\left(F(\bx_i^\nu)-F(\widehat{x})\right)-\frac{\alpha(1-\alpha)\mu}{2}\norm{\bx_i^\nu-\widehat{x}}^2
\\
&\qquad\qquad\qquad\qquad\qquad\qquad \qquad +2\beta\alpha^2\norm{\widehat{x}-\bx_i^\nu}^2+\frac{M_i+\Lhessian}{6}\alpha^3\norm{\widehat{x}-\bx_i^\nu}^3+\frac{1}{\beta}\norm{\bepsilon_i^\nu}^2 \Big\},
\end{aligned}
\end{equation}
where (a) is due to the $\mu$-strong convexity of $F$.
If $\alpha_0=  1/(1+4\beta/\mu)$,   \eqref{F_Fast_upperbound_SCVX} implies
\begin{equation*}
 F (\bx_i^{\nu+})-F(\widehat{x})
\leq  \min_{\alpha\in[0,\alpha_0]}\Big\{(1-\alpha)\left(F(\bx_i^\nu)-F(\widehat{x})\right)+\frac{M_i+\Lhessian}{6}\alpha^3\norm{\widehat{x}-\bx_i^\nu}^3+\frac{1}{\beta}\norm{\bepsilon_i^\nu}^2 \Big\},
\end{equation*}
where by the $\mu$-strongly convexity of $F$ and optimality of $\widehat{x}$, we also deduce
\begin{equation}
\label{F_Fast_upperbound_SCVXcase}
\begin{aligned}
& F (\bx_i^{\nu+})-F(\widehat{x})
\\
\leq & \min_{\alpha\in[0,\alpha_0]}\Big\{(1-\alpha)\left(F(\bx_i^\nu)-F(\widehat{x})\right)
+\frac{M_i+\Lhessian}{6}\alpha^3\left(\frac{2}{\mu}\left(F(\bx_i^\nu)-F(\widehat{x})\right)\right)^{3/2}+\frac{1}{\beta}\norm{\bepsilon_i^\nu}^2 \Big\}.
\end{aligned}
\end{equation}
Averaging  \eqref{F_Fast_upperbound_SCVXcase} over $i=1,2,\ldots,m$ while  using \eqref{pnuplus_pnu}, yields
\begin{equation}
\label{p_upperbound_scvx}
p^{\nu+1}\leq \min_{\alpha\in[0,\alpha_0]}\Big\{(1-\alpha)p^\nu 
+C_2\alpha^3\left(p^\nu\right)^{3/2}+\frac{1}{m\beta}\norm{\bepsilon^\nu}^2 \Big\},\quad  C_2\triangleq \xi\cdot  \frac{(M_{\max}+\Lhessian)\sqrt{2m}}{3\mu^{3/2}},
\end{equation}
where $M_{\max}=\max_{i\in[m]}~M_i$ and   $\xi\geq 1$ is arbitrary. 

Denote by $\alpha^\ast$ the minimizer of the  RHS of \eqref{p_upperbound_scvx}; then if $p^\nu\geq\underline{p}_1\triangleq 1/(9C_2^2\alpha_0^4)$, we have   $\alpha^\ast=1/\sqrt{3C_2\sqrt{p^\nu}}$, and 
\begin{equation}
\label{p_recursion_firststage_scvx_final}
\begin{aligned}
p^{\nu+1}
\leq &  p^\nu-\frac{2(p^\nu)^{3/4}}{3\sqrt{3C_2}}+\frac{1}{m\beta}\norm{\bepsilon^\nu}^2.
\end{aligned}
\end{equation}
If 
\begin{equation}
\label{K_cond_scvx_a}
\frac{1}{m\beta}\norm{\bepsilon^\nu}^2\leq \frac{1}{3\sqrt{3C_2}}(\underline{p}_1)^{3/4}\implies  \frac{1}{m\beta}\norm{\bepsilon^\nu}^2\leq \frac{1}{3\sqrt{3C_2}}(p^\nu)^{3/4},
\end{equation}
  \eqref{p_recursion_firststage_scvx_final} yields
\begin{equation}
\label{p_recursion_firststage_scvx_final3}
\begin{aligned}
p^{\nu+1}\leq &  p^\nu-  \tilde{c}~(p^\nu)^{3/4},\quad \forall \nu\geq 0,\quad \tilde{c}\triangleq \frac{1}{3\sqrt{3C_2}}.
\end{aligned}
\end{equation}
Note that, by \eqref{delta_nu_decaying_bound} and Lemma \ref{cons_track_err_prop}, condition \eqref{K_cond_scvx_a}  holds if 
\begin{equation}\label{K_cond_scvx1}
K\geq \frac{1}{\sqrt{1-\rho}}\cdot\frac{1}{2} \log\left(\frac{3\bar{D}_\delta\sqrt{3C_2}}{m\beta\underline{p}_1^{3/4}}\right).
\end{equation}
We now prove by induction that \eqref{p_recursion_firststage_scvx_final3} implies 
\begin{equation}\label{induction_assumption}
(p^\nu)^{1/4}\leq l_\nu\triangleq (p^0)^{1/4}-\frac{\tilde{c}}{4}\nu,\quad \forall \nu\geq 0.
\end{equation}
Clearly, \eqref{induction_assumption} holds for $\nu=0$.  Since the  RHS of \eqref{p_recursion_firststage_scvx_final3} is increasing (as a function of $p^\nu$) when $p^\nu\geq \left(3\tilde{c}/4\right)^4=1/(9\cdot 2^{8}C_2^2)$ (which holds since $p^\nu\geq \underline{p}_1$), then   $p^\nu\leq l_\nu^4$ implies
\begin{equation*}
p^{\nu+1}\leq l_\nu^4-\tilde{c}l_\nu^3,
\end{equation*}
which also implies $p^{\nu+1}\leq l_{\nu+1}^4$, as  by definition of $l^\nu$ in \eqref{induction_assumption},
\begin{equation*}
\begin{aligned}
l_{\nu}^4-l_{\nu+1}^4=\left(l_{\nu}-l_{\nu+1}\right)\left(l_{\nu}+l_{\nu+1}\right)\left(l_{\nu}^2+l_{\nu+1}^2\right)=\frac{\tilde{c}}{4}\left(l_{\nu}+l_{\nu+1}\right)\left(l_{\nu}^2+l_{\nu+1}^2\right)\leq  \tilde{c}~l_\nu^3.
\end{aligned}
\end{equation*}

\textbf{(b)}
Recall \eqref{synergy_net_opt_descent} (from the proof of Theorem \ref{thm:asymptotic_conv}), which    under Assumptions \ref{sconvex-case}-\ref{assump:weight} and condition \eqref{K_cond_asymptotic},  reads
\begin{align}
\label{synergy_net_opt_descent_recalled}
wp^{\nu+1}+(\tilde{\delta}^{\nu+1})^2 
\leq & wp^\nu+c_w(\tilde{\delta}^\nu)^2-  \frac{w\mu}{4m} ||\Delta\bx^\nu||^2.
\end{align}
Recall also Lemma \ref{Deltax_lowerbound_lemm} when condition \hyperlink{pDeltaX_cond}{$\texttt{C}$} is satisfied, which together with \eqref{pnuplus_pnu}, implies
\begin{equation}
\begin{aligned}
p^{\nu+1}
\leq &b_1\left(\sum_{i=1}^m\left\|\hxnu{\nu}_i-\bx_i^\nu\right\|^2\right)^2+\frac{8}{m\mu}\norm{\bepsilon^\nu}^2,\quad b_1\triangleq  \frac{\Lhessian^2+4M_{\max}^2}{m\mu}.
\end{aligned}
\label{Deltax_lowerbound_xhat_recalled_refined1_alt}
\end{equation}
Note that $p^{\nu+1}\geq \underline{p}_2$ implies that condition \hyperlink{pDeltaX_cond}{$\texttt{C}$} in Lemma \ref{Deltax_lowerbound_lemm}  holds, as proved next  by contradiction. Suppose $p^{\nu+1}\geq \underline{p}_2$ but $||\Delta\bx^\nu||<\omega_0$. Then  Lemma \ref{Deltax_lowerbound_lemm} yields
\begin{equation*}
\underline{p}_2\leq p^{\nu+1}\overset{\eqref{pnuplus_pnu}}{\leq} p^{\nu +}<\frac{ 144\beta^2}{m\mu}\cdot \omega_0^2+\frac{8}{m\mu}\norm{\bepsilon^\nu}^2\overset{(a)}{\leq} \frac{2\cdot 12^4}{\Lhessian^2+4M_{\max}^2}\cdot \frac{\beta^4}{m\mu}, 
\end{equation*}
implying $\beta>\mu$, which is in contradiction with the assumption; note that  (a) holds under (similar to derivation of \eqref{K_cond_scvx1}) 
\begin{equation}
\label{K_cond_scvx_a_half}
K\geq \frac{1}{\sqrt{1-\rho}}\cdot\frac{1}{2} \log\left(\frac{\bar{D}_\delta}{18\beta^2\omega_0^2}\right)\implies \frac{8}{m\mu}\norm{\bepsilon^\nu}^2\leq \frac{144\beta^2\omega_0^2}{m\mu}.
\end{equation}

Now since $x\mapsto x^h$ is subadditive for $0\leq h\leq 1$, i.e. $(a+b)^h\leq a^h+b^h$ for any $a,b\geq 0$, \eqref{Deltax_lowerbound_xhat_recalled_refined1_alt} together with \eqref{delta_upperbnd} imply 
\begin{equation}
\begin{aligned}
-\sum_{i=1}^m\left\|\Delta\bx_i^\nu\right\|^2
\leq 
& -b_1^{-\frac{1}{2}}\left(p^{\nu+1}\right)^{\frac{1}{2}}+\sqrt{\frac{8}{m\mu b_1}}~\tilde{\delta}^\nu.
\end{aligned}
\label{Deltax_lowerbound_xhat_recalled_refined2_v3}
\end{equation}

Combining \eqref{synergy_net_opt_descent_recalled} with \eqref{Deltax_lowerbound_xhat_recalled_refined2_v3} yields
\begin{equation*}
wp^{\nu+1}+(\tilde{\delta}^{\nu+1})^2 
\leq  wp^\nu+c_w(\tilde{\delta}^\nu)^2-  \frac{w\mu}{4m\sqrt{b_1}} \sqrt{p^{\nu+1}}+\frac{w\mu}{4m}\sqrt{\frac{8}{m\mu b_1}}\tilde{\delta}^\nu,
\end{equation*}
and since $\tilde{\delta}^\nu\leq \sqrt{\xi^\nu}\leq \sqrt{D_2},\forall \nu\geq 0$ (see the discussion in Subsec. \ref{sec:asymptotic_convergence_Step3}, proof of Theorem \ref{thm:asymptotic_conv}), we get
\begin{align}
\label{synergy_net_opt_quadraticphase_final2}
wp^{\nu+1}+(\tilde{\delta}^{\nu+1})^2 
\leq & wp^\nu-  \frac{w\mu}{4m\sqrt{b_1}} \sqrt{p^{\nu+1}}+C_3\tilde{\delta}^\nu,\quad C_3\triangleq \left(c_w\sqrt{D_2} +\frac{c_w\beta \mu}{4m}\sqrt{\frac{8}{m\mu b_1}}\right).
\end{align}
Since $p^{\nu+1}\geq \underline{p}_2$,   under (similar to derivation of \eqref{K_cond_scvx1}) 
\begin{equation}\label{K_cond_scvx_b}
K\geq \frac{1}{\sqrt{1-\rho}}\cdot\frac{1}{2} \log\left(\frac{64\bar{D}_\delta m^2b_1C_3^2}{c_w^2\beta^2\mu^2\underline{p}_2}\right) \implies C_3 \tilde{\delta}^\nu\leq \frac{w\mu\sqrt{\underline{p}_2}}{8m\sqrt{b_1}},
\end{equation}
  \eqref{synergy_net_opt_quadraticphase_final2} yields
\begin{equation*}
p^{\nu+1}+c  \sqrt{p^{\nu+1}} 
\leq  p^\nu,\quad c\triangleq \frac{\mu}{8m\sqrt{b_1}}.
\end{equation*}
Denote by $\tilde{p}^\nu\triangleq p^\nu/c^2$, then we get $\tilde{p}^{\nu+1}+\sqrt{\tilde{p}^{\nu+1}}\leq \tilde{p}^{\nu}$ which implies quadratic convergence when 
$p^{\nu+1}\geq \underline{p}_2$ and $\tilde{p}^\nu\leq 1\equiv p^\nu\leq c^2$.

\textbf{(c)}
Again recall \eqref{synergy_net_opt_descent}:
\begin{align}
\label{synergy_net_opt_descent_FinalStage}
wp^{\nu+1}+(\tilde{\delta}^{\nu+1})^2 
\leq & wp^\nu+c_w(\tilde{\delta}^\nu)^2-  \frac{w\max(\beta,\mu)}{4m} ||\Delta\bx^\nu||^2.
\end{align}
Invoking  Lemma \ref{Deltax_lowerbound_lemm} under    condition  \hyperlink{pDeltaX_cond_bar}{$\bar{\texttt{C}}$}  and $\tau_i=2\beta$, along with  \eqref{pnuplus_pnu} and \eqref{delta_upperbnd}, we have  \begin{equation}
\begin{aligned}
p^{\nu+1}
\leq & b_2\sum_{i=1}^m\norm{\hxnu{\nu}_i-\bx_i^\nu}^2+\frac{8}{m\mu}(\tilde{\delta}^\nu)^2,\quad b_2\triangleq \frac{144\beta^2}{m\mu}.
\end{aligned}
\label{Deltax_lowerbound_xhat_recalled_refined_S2_1}
\end{equation}
Combining \eqref{synergy_net_opt_descent_FinalStage} and \eqref{Deltax_lowerbound_xhat_recalled_refined_S2_1} yields 
\begin{equation}\label{h2}
\begin{aligned}
w\left(1+\frac{\max(\beta,\mu)}{4mb_2}\right)p^{\nu+1}+(\tilde{\delta}^{\nu+1})^2 
\leq &wp^\nu+
\left(c_w+\frac{2w\max(\beta,\mu)}{m^2\mu b_2}\right) (\tilde{\delta}^\nu)^2,
\end{aligned}
\end{equation} 
where by choosing $c_w$ to satisfy
\begin{equation}
\label{c_w_cond}
\begin{aligned}
& \left(c_w+\frac{2w\max(\beta,\mu)}{m^2\mu b_2}\right)\leq \left(1+\frac{\max(\beta,\mu)}{4mb_2}\right)^{-1}
\quad \overset{(a)}{\equiv}    & c_w \leq \left(1+\frac{2\beta\max(\beta,\mu)}{m^2\mu b_2}\right)^{-1}\left(1+\frac{\max(\beta,\mu)}{4mb_2}\right)^{-1},
\end{aligned}
\end{equation}
[where (a) is due to $w=c_w\beta$ defined in Sec. \ref{sec:asymptotic_convergence_Step3}],   \eqref{h2} becomes 
\begin{equation*}
\begin{aligned}
w\left(1+\frac{\max(\beta,\mu)}{4mb_2}\right)p^{\nu+1}+(\tilde{\delta}^{\nu+1})^2 
\leq &wp^\nu+
\left(1+\frac{\max(\beta,\mu)}{4mb_2}\right)^{-1} (\tilde{\delta}^\nu)^2,
\end{aligned}
\end{equation*} 
implying linear convergence of $\{\xi^\nu\}_\nu$ where
\begin{equation*}
\zeta^\nu\triangleq w\left(1+\frac{\max(\beta,\mu)}{4mb_2}\right)p^{\nu}+(\tilde{\delta}^{\nu})^2,
\end{equation*}
and decay rate 
\begin{equation}
\label{decay_rate}
\left(1+\frac{\max(\beta,\mu)}{4mb_2}\right)^{-1}=\left(1+\frac{1}{576}\cdot \frac{\mu\max(\beta,\mu)}{\beta^2}\right)^{-1}.
\end{equation}
Therefore, $\{p^\nu\}_\nu$ converges $Q$-linearly with rate \eqref{decay_rate}.

Now let us derive \eqref{Last_phase_cond_SCVX} that defines this region. The goal is to identify the region where  \hyperlink{pDeltaX_cond_bar}{$\bar{\texttt{C}}$} (cf. Lemma \ref{Deltax_lowerbound_lemm}) holds. Under the condition  (similar to derivation of \eqref{K_cond_scvx1})
\begin{equation}\label{K_cond_scvx_c}
K\geq \frac{1}{\sqrt{1-\rho}}\cdot\frac{1}{2} \log\left(\frac{4\bar{D}_\delta}{\beta\mu\omega_0^2}\right)\implies \frac{2(\tilde{\delta}^\nu)^2}{\beta\mu}\leq \frac{\omega_o^2}{2},
\end{equation}
and Lemma \ref{deltax_upperbnd}, there holds
\begin{equation*}
\frac{1}{m}\sum_{i=1}^m\norm{\Delta\bx_i^\nu}^2\leq \frac{8}{\mu}p^\nu+\frac{\omega_0^2}{2m},
\end{equation*}
which implies that \hyperlink{pDeltaX_cond_bar}{$\bar{\texttt{C}}$} is necessarily satisfied when
\begin{equation*}
p^\nu<\frac{\omega_0^2\mu}{16m}=\frac{9}{L^2+4M_{\max}^2}\cdot \frac{\beta^2\mu}{m}.
\end{equation*}

Finally, unifying the conditions on $K$ derived in \eqref{K_cond_asymptotic}. \eqref{K_cond_scvx1}, \eqref{K_cond_scvx_a_half}, \eqref{K_cond_scvx_b}, \eqref{K_cond_scvx_c}, $K$ must satisfy
\begin{equation}\label{K_cond_scvx}
K\geq \frac{1}{\sqrt{1-\rho}}\cdot\frac{1}{2} \log\left(\bar{D}_\delta\cdot \max\left\{\frac{16}{\bar{D}_\delta c_w},\frac{12^2m\Lgrad_{\max}^2}{\bar{D}_\delta c_w\beta\max(\beta,\mu)},\frac{3\sqrt{3C_2}}{m\beta\underline{p}_1^{3/4}},\frac{1}{18\beta^2\omega_0^2},\frac{64m^2b_1C_3^2}{c_w^2\beta^2\mu^2\underline{p}_2},\frac{4}{\beta\mu\omega_0^2}\right\}\right),
\end{equation}
where recall that $c_w>0$ must satisfy \eqref{c_w_cond}. 
\end{proof}


\subsection{Proof of Theorem \ref{thm:scvx_case_beta_leq_mu}}\label{subsec:scvx_beta_leq_mu}
Let $M_i=\Lhessian$ for all $i=1,\ldots,m$, and set the free parameter $\xi\geq 1$ (defined in Theorem \ref{Rate_scvx_THM}) to $\xi=100\sqrt{5}$, and  define the regions of convergence,
\begin{equation*}
\begin{aligned}
\text{(R0)}: & \quad  \Omega_0\leq p^\nu,
\\
\text{(R1)}: & \quad  \Omega_1\leq p^\nu < \Omega_0, 
\\
\text{(R2)}: & \quad \max(\varepsilon,\Omega_2)\leq p^\nu <  \Omega_1,
\\
\text{(R3)}: & \quad \varepsilon\leq p^\nu <  \max(\varepsilon,\Omega_2),
\end{aligned}
\end{equation*}
where
\begin{equation*}
\Omega_0=244\cdot D^2\mu,\quad \Omega_1=c^2/2=\frac{1}{640\Lhessian^2}\cdot\frac{\mu^3}{m}, \quad  \Omega_2=\underline{p}_2=\frac{2\cdot 12^4}{5L^2}\cdot \frac{\beta^2\mu}{m},
\end{equation*}
and $c$ and $\underline{p}_2$ are defined in Theorem \ref{Rate_scvx_THM}.
%
%

Using Theorem \ref{Rate_cvx_THM}, region ($\text{R0}$) takes at most 
$\sqrt{\frac{\Lhessian D}{\mu}}$ iterations. Now using Theorem \ref{Rate_scvx_THM}, region ($\text{R1}$) lasts at most $\nu_1$ iterations satisfying
\begin{equation*}
\begin{aligned}
& \left( \Omega_1\right)^{1/4}\geq  (\Omega_0)^{1/4}-\frac{\nu_1}{12\sqrt{3C_2}}
\impliedby &\nu_1\geq 480\sqrt{3\sqrt{5}}\cdot m^{1/4}\cdot  \sqrt{\frac{\Lhessian D}{\mu}}.
\end{aligned}				 
\end{equation*}
Let us conservatively consider scenarios $\Omega_1\geq \varepsilon\geq \Omega_2$ and $\varepsilon< \Omega_2$, then the  region of quadratic convergence ($\text{R2}$) lasts for at most
\begin{equation*}
 2\log\left(2\log\left(\min\left\{\frac{c^2}{\Omega_2},\frac{c^2}{\varepsilon}\right\}\right)\right)\leq 2\log\left[2\log\left[\min\left\{\frac{1}{128\cdot 12^4}\cdot \frac{\mu^2}{\beta^2},\frac{\mu^3}{320m\Lhessian^2}\cdot \frac{1}{\varepsilon}\right\}\right]\right]:\quad c^2\geq \Omega_2,\varepsilon\leq c^2,
\end{equation*}
iterations. Note that conditions $p^\nu\geq \underline{p}_2$ and $p^\nu<\underline{p}_3$ in Theorem \ref{Rate_scvx_THM} are sufficient conditions identifying the region of quadratic and linear rate (or more specifically \hyperlink{pDeltaX_cond}{$\texttt{C}$} and \hyperlink{pDeltaX_cond_bar}{$\bar{\texttt{C}}$} in Lemma \ref{Deltax_lowerbound_lemm}); note that $\underline{p}_2$ and $\underline{p}_3$ are identical up to multiplying constants. Hence, to obtain a valid complexity of overall performance, we pessimistically associate the region of linear rate ($\text{R3}$) with $\varepsilon<p^\nu\leq \max(\varepsilon,\Omega_2)$ rather than $\varepsilon<p^\nu\leq \max(\varepsilon,\underline{p}_3)$; therefore, this region at most lasts for $O(\beta/\mu\cdot \log(\max(\varepsilon,\Omega_2)/\varepsilon))$ iterations. Thus, since the number of communications per iteration  is $\tO{1/\sqrt{1-\rho}}$ [cf. \eqref{K_cond_asymptotic}, \eqref{K_cond_cvx}, \eqref{K_cond_scvx} and note that $\varepsilon=\Omega_0$ in \eqref{K_cond_cvx}], the overall complexity reads
\begin{equation*}
\tO{\frac{1}{\sqrt{1-\rho}}\left\{\sqrt{\frac{\Lhessian D}{\mu}}\left(1+m^{1/4}\right)+\log\left[\log\left[\frac{\mu^2}{\beta^2}\cdot \min\left\{1,\frac{\beta^2\mu}{m\Lhessian^2 }\cdot \frac{1}{\varepsilon}\right\}\right]\right]+\frac{\beta}{\mu}\log\left[\max\left(1,\frac{\beta^2\mu}{m\Lhessian^2}\cdot \frac{1}{\varepsilon}\right)\right]\right\}}
\end{equation*}
communications.

\subsection{The case of quadratic $f_i$  in Theorem \ref{thm:scvx_case_beta_leq_mu}}
\label{sec:scvx_case_beta_leq_mu_Qcase_proof}
Here we refine the proof of Theorem \ref{thm:scvx_case_beta_leq_mu} to enhance the rate when $\Lhessian=0$:
\begin{theorem}\label{thm:scvx_case_beta_leq_mu_Qcase}
Let  Assumptions \ref{sconvex-case}-\ref{assump:network} hold with $\Lhessian=0$ and $\beta<\mu$.  Denote by $D_p$ an upperbound of $p^0$, i.e.  $p^0\leq D_p$ for all $\nu\geq 0$. Also choose $M_i=\Theta(\mu^{3/2}/\sqrt{mD_p})$ sufficiently small (explicit condition is provided in  \eqref{M_cond_scvx_Qcase}) and $\tau_i=2\beta$ for all $i=1,\ldots,m$. If a reference matrix $\overline{W}$ satisfying Assumption \ref{assump:weight} is used   in steps \eqref{DiRegINA_pseudocode}-\eqref{grad_track_update},  with $\rho  \triangleq  \lambda_{\max}(\overline{W} - \bJ) < 1$ and  $K=\tO{1/\sqrt{1-\rho}}$ (explicit condition is provided in \eqref{K_cond_scvx}),	then  for any given $\varepsilon>0$, \alg returns a solution with $p^\nu\leq \varepsilon$  after total
	\begin{equation*}
	\tO{\frac{1}{\sqrt{1-\rho}}\cdot \left\{\log\log\left(\frac{D_p}{\varepsilon
		}\right)+\frac{\beta}{\mu}\log\left(\frac{D_p\beta^2}{\mu^2\varepsilon}\right)\right\}}
	\end{equation*}
	communications.  {Note that when $\beta = O(1/\sqrt{n})$, $\epsilon=\Omega(V_N)$ and $n\geq m$, the above communication complexity reduces to
	\begin{equation*}
	\tO{\frac{1}{\sqrt{1-\rho}}\cdot \left\{\log\log\left(\frac{D_p}{V_N
		}\right)\right\}}.
	\end{equation*}}
\end{theorem}
\begin{proof}
Let us specialize  the results established in Theorem \ref{Rate_scvx_THM} (in particular case \textbf(b)-\textbf(c)). Note that, since $\Lhessian=0$, we can impose $p^0\leq c^2/2$ by a proper choice of $M_i$, allowing \alg to circumvent the first region (associated with case (a) in Theorem \ref{Rate_scvx_THM}) and start off in the quadratic rate region. Hence we only need to derive a sufficient condition for $p^0\leq c^2/2$. Let us first consider case (b): if $M_i=\Theta(\mu^{3/2}/\sqrt{mD_p}), \forall i$, sufficiently small,
\begin{equation}\label{M_cond_scvx_Qcase}
 M_{i}\leq \frac{\mu^{3/2}}{16\sqrt{2mD_p}},~\forall i\implies  p^0\leq \frac{\mu^3}{512mM_{\max}^2}\implies p^0\leq c^2/2,
\end{equation}
where $M_{\max}\triangleq \max_{i\in[m]}~M_i$.
Let us also evaluate the precision achieved in case (b), i.e. $\underline{p}_2$: denote by $\underline{C}_M$ such that $M_i\geq \underline{C}_M\mu^{3/2}/\sqrt{mD_p},\forall i$, then  
\begin{equation*}
\underline{p}_2\triangleq \frac{12^4}{2M_{\max}^2}\cdot \frac{\beta^2\mu}{m}\leq \frac{12^4}{2\underline{C}_M^2}\cdot \frac{\beta^2D_p}{\mu^2}.
\end{equation*}
Therefore the number of iterations to reach $\varepsilon=\Omega(\underline{p}_2)$ is $O(\log\log(c^2/\underline{p}_2))=\log\log(D_p/\varepsilon)$, and since $K=\tO{1/\sqrt{1-\rho}}$, the total number of communication will be $\tO{1/\sqrt{1-\rho}\cdot \log\log(D_p/\varepsilon)}$. 

Now let us derive the complexity when $\varepsilon=O(\underline{p}_2)$ (i.e. case (c) in Theorem \ref{Rate_scvx_THM}). Setting $\Lhessian=0$ and following similar arguments, for arbitrary precision $\varepsilon>0$, we obtain a communication complexity $\tO{1/\sqrt{1-\rho}\cdot \left\{\log\log(D_p/\varepsilon)+\beta/\mu\log(\beta^2D_p/(\mu^2\varepsilon))\right\}}$.
\end{proof}

\subsection{Proof of Corollary  \ref{corr:scvx_case_beta_leq_mu_initialization_iid}}
\label{sec:scvx_case_beta_leq_mu_initialization_proof}
Let us customize the rate established in Theorem \ref{Rate_scvx_THM} (in particular case \textbf(b)-\textbf(c)). We derive a sufficient condition for $p^0\leq c^2/2$ which guarantees that the initial point is in the region of quadratic convergence. Using initialization policy \eqref{eq:initialization}, there holds $p^0\leq C_\Delta/n$ for some $C_\Delta>0$. Hence, under 
\begin{equation*}
 n\geq 
\frac{640C_\Delta\Lhessian^2}{\mu^3} \cdot m\implies  p^0\leq \frac{\mu^3}{640mL^2}  \implies p^0\leq c^2/2,
\end{equation*}
 \alg converges quadratically to the precision 
\begin{equation*}
\underline{p}_2\triangleq \frac{2\cdot 12^4}{5\Lhessian^2}\cdot \frac{\beta^2\mu}{m}.
\end{equation*}
By $\beta=O(1/\sqrt{n})$, $\underline{p}_2=O(V_N)$. Hence, since $K=\tO{1/\sqrt{1-\rho}}$, the total number of communication will be  $\tO{1/\sqrt{1-\rho}\cdot \log\log(\mu^3/(m\Lhessian^2 V_N))}$.



\section{Proof of Theorem \ref{thm:scvx_case_beta_geq_mu} }\label{subsec:scvx_beta_geq_mu}

Let $M_i=\Lhessian$ for all $i=1,\ldots,m$, and set the free parameter $\xi=50\beta/(3\mu)$ (defined in Theorem \ref{Rate_scvx_THM}) and define the regions of convergence, 
\begin{equation*}
\begin{aligned}
\text{($\overline{\mbox{R0}}$)}: & \quad  \overline{\Omega}_0\leq p^\nu,
\\
\text{($\overline{\mbox{R1}}$)}: & \quad  \overline{\Omega}_1\leq p^\nu < \overline{\Omega}_0, 
\\
\text{($\overline{\mbox{R2}}$)}: & \quad \varepsilon\leq p^\nu <  \overline{\Omega}_1,
\end{aligned}
\end{equation*}
where
\begin{equation*}
\overline{\Omega}_0=244\cdot D^2\mu,\quad \overline{\Omega}_1=\frac{0.9}{L^2}\cdot \frac{\beta^2\mu}{m}. 
\end{equation*}
Using Theorem \ref{Rate_cvx_THM}, region ($\overline{\mbox{R0}}$)  takes at most 
$\sqrt{\frac{\Lhessian D}{\mu}}$ iteration; note that $\mu=\Omega(\beta^2)$ by assumption $n\geq m$, thus $\overline{\Omega}_0=\Omega(\beta^2 \cdot 2\Lhessian D^3)$. Now using Theorem \ref{Rate_scvx_THM}, region ($\overline{\mbox{R1}}$) lasts at most $\nu_1$ iteration satisfying 
\begin{equation*}
\begin{aligned}
& \left( \overline{\Omega}_1\right)^{1/4}\geq (\overline{\Omega}_0)^{1/4}-\frac{\nu_1}{12\sqrt{3C_2}}
\impliedby &\nu_1\geq 240\sqrt{2}\cdot \frac{\sqrt{\beta\Lhessian D\sqrt{m}}}{\mu}.
\end{aligned}				 
\end{equation*}
Finally, by case (c) in Theorem \ref{Rate_scvx_THM}, region ($\overline{\mbox{R2}}$)  lasts for $O(\beta/\mu\cdot \log(\overline{\Omega}_1/\varepsilon))$. Thus, since communication cost per iteration is $\tO{1/\sqrt{1-\rho}}$ [cf. \eqref{K_cond_asymptotic},  \eqref{K_cond_scvx}], the overall complexity is
\begin{equation*}
\tO{\frac{1}{\sqrt{1-\rho}}\left\{\sqrt{\frac{\Lhessian D}{\mu}}\left(1+m^{1/4}\cdot \sqrt{\frac{\beta}{\mu}}\right)+\frac{\beta}{\mu}\log\left(\frac{\beta^2\mu}{m\Lhessian^2}\cdot \frac{1}{\varepsilon}\right)\right\}}.
\end{equation*}

\section{The case of quadratic $f_i$ in Theorem \ref{thm:scvx_case_beta_geq_mu}}\label{subsec:scvx_beta_geq_mu_Qcase}
 {
\begin{theorem}\label{thm:scvx_case_beta_geq_mu_QUAD}
Instate the setting of Theorem \ref{thm:scvx_case_beta_geq_mu} where $\Lhessian=0$. Then,    the total number of communications for \alg to make   $p^\nu\leq \varepsilon$ reads 
\begin{equation*} 
\hspace{-0.2cm}\widetilde{\mathcal{O}}\left({\frac{1}{\sqrt{1-\rho}}\cdot \frac{\beta}{\mu}\log\Big(\frac{1}{\varepsilon}\Big)}\right).
\end{equation*}	
When $\beta = O(1/\sqrt{n})$, $\epsilon=\Omega(V_N)$ and $n\geq m$, the above communication complexity reduces to
		\begin{equation*}
	\widetilde{\mathcal{O}}\left({\frac{1}{\sqrt{1-\rho}}\cdot m^{1/2}\cdot \log\Big(\frac{1}{V_N}\Big)}\right).
	\end{equation*}
\end{theorem}
}

\begin{proof}
We customize case \textbf{(c)} in Theorem \ref{Rate_scvx_THM}, when $\Lhessian=0$. Note that \hyperlink{pDeltaX_cond_bar}{$\bar{\texttt{C}}$} in Lemma \ref{Deltax_lowerbound_lemm} holds for all $\nu\geq 0$ and  condition \eqref{K_cond_scvx_c} is no longer required. Therefore,  the algorithm converges linearly with rate \eqref{decay_rate_thm} and returns a solution within $\varepsilon$  precision  within $O\left( \beta/\mu\cdot\log(1/\varepsilon)\right)$ iterations and since $K=\tO{1/\sqrt{1-\rho}}$ [cf. \eqref{K_cond_asymptotic}] , the total number of  required communications is $\tilde{O}\left(1/\sqrt{1-\rho}\cdot \beta/\mu\cdot\log(1/\varepsilon)\right)$.
\end{proof}